\documentclass[12pt]{amsart}

\usepackage{amsmath}

\usepackage{amssymb}

\usepackage{mathtools}

\usepackage{array}

\usepackage{a4wide}
  
\usepackage{stmaryrd}

\usepackage{tikz}

\usepackage{tikz-cd}
\usetikzlibrary{cd}
\usetikzlibrary{decorations.markings}
\tikzset{double line with arrow/.style args={#1,#2}{decorate,decoration={markings,%
mark=at position 0 with {\coordinate (ta-base-1) at (0,1pt);
\coordinate (ta-base-2) at (0,-1pt);},
mark=at position 1 with {\draw[#1] (ta-base-1) -- (0,1pt);
\draw[#2] (ta-base-2) -- (0,-1pt);
}}}}

\usepackage{bm}

\usepackage{mathrsfs}

\usepackage{amsthm}

\usepackage{enumerate}

\usepackage{fancyhdr}

\usepackage{hyperref}

\usepackage[inline]{enumitem}

\usepackage{marginnote}

\theoremstyle{plain}

\newtheorem{thm}{Theorem}[section]

\newtheorem{lem}[thm]{Lemma}

\newtheorem{cor}[thm]{Corollary}

\newtheorem{prop}[thm]{Proposition}

\theoremstyle{definition}

\newtheorem{rem}[thm]{Remark}

\newtheorem*{note}{Note}

\newtheorem{defn}[thm]{Definition}

\newtheorem{eg}[thm]{Example}

\newtheorem{egs}[thm]{Examples}

\newtheorem*{blank}{}

\newcommand*\isommap{
  \xrightarrow{\raisebox{-0.5ex}[0ex][0ex]{$\sim$}}
}

\newcommand{\cO}{\mathcal{O}}

\newcommand{\bC}{\mathbb{C}}
\newcommand{\bD}{\mathbb{D}}

\newcommand{\bP}{\mathbb{P}}
\newcommand{\bQ}{\mathbb{Q}}
\newcommand{\bR}{\mathbb{R}}

\newcommand{\bV}{\mathbb{V}}

\newcommand{\bZ}{\mathbb{Z}}

\newcommand{\fx}{\mathfrak{x}}
\newcommand{\fy}{\mathfrak{y}}

\newcommand{\scB}{\mathscr{B}}
\newcommand{\scC}{\mathscr{C}}
\newcommand{\scD}{\mathscr{D}}
\newcommand{\scE}{\mathscr{E}}

\newcommand{\scH}{\mathscr{H}}
\newcommand{\scI}{\mathscr{I}}
\newcommand{\scJ}{\mathscr{J}}

\newcommand{\scL}{\mathscr{L}}
\newcommand{\scM}{\mathscr{M}}
\newcommand{\scN}{\mathscr{N}}
\newcommand{\scO}{\mathscr{O}}

\begin{document}

\title{Weight filtration and generating level}

\author{Henry Dakin}

\begin{abstract}
We study the canonical mixed Hodge module structure associated to the $\scD_X$-module $\scM(f^{-\alpha}):=\scO_X(*f)f^{-\alpha}$. We particularly focus on the weight filtration and extend many known results to the weighted setting. We obtain new relations between Hodge theory and birational geometry. We derive a general formula for the Hodge and weight filtrations on $\scM(f^{-\alpha})$, and use this to obtain results concerning the highest weight of $\scM(f^{-\alpha})$ and the generating level of weight filtration steps. Finally, we obtain expressions for several classes of divisor, including certain parametrically prime divisors.
\end{abstract}

\date{}

\maketitle

\tableofcontents

\section{Introduction}\label{sectionintro}

Consider a complex manifold $X$ of dimension $n$ and a holomorphic function $f:X \to \bC$ such that $Z:=\{\fx\in X\mid f(\fx)=0\}\subseteq X$ is a hypersurface. If $\alpha\in \bQ$, we may consider the left $\scD_X$-module $\scM(f^{-\alpha}):=\scO_X(*f)f^{-\alpha}$, where $\scO_X(*f)$ is the left $\scD_X$-module of meromorphic functions on $X$ which are holomorphic along $U:=X\backslash Z$. 

$\scM(f^{-\alpha})$ underlies a complex mixed Hodge module in a natural way, given by pulling back the rank one local system $\bV^{-\alpha}$ on $\bC^*$ with monodromy eigenvalue $e^{2\pi i \alpha}$ along $f|_U$, and then pushing forward along the natural open embedding $j:U\hookrightarrow X$. We will write $M(f^{-\alpha})\in\text{MHM}(X,\bC)$ for this mixed Hodge module. This equips $\scM(f^{-\alpha})$ with two filtrations; the \emph{Hodge filtration} $F_{\bullet}^H\scM(f^{-\alpha})$, a good filtration by sub-$\scO_X$-modules, and the \emph{weight filtration} $W_{\bullet}\scM(f^{-\alpha})$, a finite filtration by sub-$\scD_X$-modules, each of which also underlies a complex mixed Hodge module. 

These two filtrations are deeply linked to the mixed Hodge theory of $U$. Namely, there are natural mixed Hodge structures on the cohomology groups of $U$ with coefficients in the local system on $U$ given by pulling back $\bV^{-\alpha}$ along $f|_U$, as first defined by Deligne, \cite{Del71} (at least in the case $\alpha\in\bZ$). Using the theory of mixed Hodge modules, the mixed Hodge module structure (i.e. the filtrations $F_{\bullet}^H$ and $W_{\bullet}$) can be used to determine these mixed Hodge structures on cohomology groups on $U$. 

When $\alpha\in\bZ$, $\scM(f^{-\alpha}) \simeq \scO_X(*f)$ naturally (extending also to an isomorphism of mixed Hodge modules), and in this case $\scM(f^{-\alpha})$ underlies a \emph{rational} mixed Hodge module, as developed by Saito in \cite{MSai88} and \cite{MSai90}. Otherwise, the mixed Hodge module $M(f^{-\alpha})$ does not have a rational structure and complex mixed Hodge modules become necessary.

In recent work Mustaţă and Popa have in a series of papers (\cite{MP19a}, \cite{MP19b}, \cite{MP19c}, \cite{MP20}) endeavoured to better understand the Hodge filtration $F_{\bullet}^H$ on $\scM(f^{-\alpha})$ through the use of \emph{Hodge ideals}, realising many deep relations to singularity theory and to the birational geometry of pairs, as well as obtaining formulae for $F_{\bullet}^H$ and bounds on invariants like generating level using specialisability. There is also now an extensive literature dedicated towards calculating the Hodge filtration for certain classes of function $f$ (\cite{MSai09}, \cite{Z21}, \cite{CNS22}, \cite{BD24}, \cite{D24}). 

Of course, full understanding of the Hodge theory of $U$ only comes by also considering the weight filtration $W_{\bullet}\scM(f^{-\alpha})$. The literature here is far more limited. The study of the weight filtration was begun by Olano in the papers \cite{Olano21} and \cite{Olano23}, in which further relations to birational geometry are attained, as well as a general formula for \emph{weighted Hodge ideals}, analogous to the one for Hodge ideals proven by Mustaţă and Popa in \cite{MP20}. The aim of this paper is to further this study, obtaining and generalising many results analogous to those of the aforementioned authors.

To this end, there are two approaches, as employed in the work of Mustaţă, Popa, Olano. The first uses birational geometry, and a complete understanding of the mixed Hodge module structure on $\scM(f^{-\alpha})$ in the case that $f$ is simple normal crossing.

Write $D$ for the divisor associated to the function $f$. Then we choose a \emph{strong log resolution} of the pair $(X,D)$, i.e. a proper birational morphism $\pi : Y \to X$, where $Y$ is a complex manifold, where the pullback $E:=\pi^*D$ of $D$ to $Y$ is a simple normal crossing divisor, and where the restriction of $\pi$ to $Y\backslash \text{Supp}(E)$ is an isomorphism. $E$ is then the divisor associated to the function $g:=f\circ\pi$. Then it is easy to see that $M(f^{-\alpha})=\pi_*M(g^{-\alpha})$, and we may attempt to obtain results concerning the filtrations $F_{\bullet}^H$ and $W_{\bullet}$ on $M(f^{-\alpha})$ by passing through this log resolution.

It is well known that in the case $\alpha\in\bZ$ (see \cite{Del71}), the canonical mixed Hodge structure on the cohomology groups of $U$ is given by considering the hypercohomology of the \emph{log de Rham complex} $\Omega_Y^{\bullet}(\log E_{\text{red}})$, with the weight filtration given by filtering by the order of pole. Here, in local coordinates such that $g_{\text{red}}=y_1\cdots y_r$, 
\[\Omega_Y^k(\log E_{\text{red}}):=\sum\scO_Y\bigwedge_{i\in J_1}dy_i\wedge\bigwedge_{i\in J_2}\frac{dy_i}{y_i} \,\text{ and }\, W_l\Omega_Y^k(\log E_{\text{red}}) :=\sum_{|J_2|=l}\scO_Y\bigwedge_{i\in J_1}y_i\wedge\bigwedge_{i\in J_2}\frac{dy_i}{y_i},\]
where $J_2\subseteq [r]$ and $|J_1|+|J_2|=k$ in both summands.

In the case $\alpha\notin \bZ$, we instead have to twist the above log de Rham complex by $\scO_Y(-\lceil\alpha E\rceil)$, and to consider a different filtration $W_{\bullet}^{\alpha}$ on $\Omega_Y^k(\log E_{\text{red}})$. For this, we write in local coordinates $g=\prod_{i=1}^ny_i^{a_i}$ (at a point $\fy\in Y$ say), and let $I_{\fy}=\{i\in[n]\mid a_i\neq 0\}$ and $I_{\alpha,\fy}:=\{i\in I\mid \alpha a_i \in \bZ\}$. Then 
\[W_l^{\alpha}\Omega_Y^k(\log E_{\text{red}}) :=\sum_{|J_2\cap I_{\alpha,\fy}|=l}\scO_Y\bigwedge_{i\in J_1}dy_i\wedge\bigwedge_{i\in J_2}\frac{dy_i}{y_i},\]
where $J_2 \subseteq I_{\fy}$ and $|J_1|+|J_2|=k$. Then, in particular, if we define the bi-filtered complex $(C_{-\alpha}^{\bullet}, F_{\bullet},W_{\bullet})$ of induced right $\scD_Y$-modules by 
\[F_{k-n}W_{n+l}C_{-\alpha}^{\bullet}:=\scO_Y(-\lceil \alpha E\rceil)\otimes_{\scO_Y}W_l^{\alpha}\Omega_Y^{n+\bullet}(\log E_{\text{red}})\otimes_{\scO_Y}F_{\bullet+k}\scD_Y,\]
whose differential is a twist of the standard one (see Section \ref{sectionlogres} below), then\footnote{This theorem is an extension of \cite{Olano23}, Proposition 4.1, to the non-reduced situation.}:

\begin{thm}[Proposition \ref{propSNCresW} and Theorem \ref{thmpiplus} below]

Write $\scM_r(f^{-\alpha})$ for the right $\scD_X$-module associated to $\scM(f^{-\alpha})$, with induced Hodge and weight filtrations. Then 

\begin{enumerate}[label=\roman*)]

\item The filtered complex $(W_{n+l}C_{-\alpha}^{\bullet},F_{\bullet-n})$ is quasi-isomorphic to the filtered right $\scD_Y$-module $(W_{n+l}\scM_r(g^{-\alpha}),F_{\bullet-n}^H)$. \vspace{5pt}

\item $F_{k-n}^HW_{n+l}\scM_r(f^{-\alpha}) = \text{\emph{Im}} \left(R^0\pi_*F_{k-n}(W_{n+l}C_{-\alpha}^{\bullet}\otimes_{\scD_Y}\scD_{Y\to X}) \to \scM_r(f^{-\alpha})\right).$

\end{enumerate}

\label{thmpiplusintro}
    
\end{thm}

\noindent From this theorem we are able to deduce many interesting consequences, including stronger results for the lowest weight filtration step $W_n\scM(f^{-\alpha})$, and relations to familiar birational invariants like multiplier ideals and adjoint ideals, in turn giving criteria to determine the singularity type of the pair $(X,\alpha D)$.

For $Z$ an arbitrary complex manifold, and $(\scM,F_{\bullet})$ a filtered $(\scD_Z,F_{\bullet}^{\text{ord}})$-module, the \emph{generating level} of $(\scM,F_{\bullet})$ is the smallest integer $k \in\bZ$ such that
\[F_i\scD_Z\cdot F_k\scM =F_{k+i}\scM \text{ for all }i \in\bZ_{\geq 0}.\]
The following is analogous to the result obtained by Mustaţă and Popa (\cite{MP19b}, Theorem 10.1) for $(\scM(f^{-\alpha}),F_{\bullet}^H)$.

\begin{cor}[Corollary \ref{corgenlevWn} below]

The generating level of $(W_n\scM(f^{-\alpha}), F_{\bullet}^H)$ is $\leq k$ if and only if 
\[R^i\pi_*\left(\scO_Y(-\lceil\alpha E\rceil)\otimes\Omega_Y^{n-i}(\log E_{I\backslash I_{\alpha}})\right)=0 \,\,\,\text{ for all }\,\, i >k,\]
where $E_{I\backslash I_{\alpha}}$ is given locally at $\fy$ by $\prod_{i\in I_{\fy}\backslash I_{\alpha,\fy}}y_i$.
In particular, the generating level is always $\leq n-1$.

\label{corgenlevWnintro}

\end{cor}

\noindent Note that we use this result later in order to find a bound for the generating level of the other weight filtration steps, but we were not able to find such a bound without introducing further tools such as specialisability.

Now we extend results of Mustaţă, Popa, Olano (\cite{MP19a}, \cite{MP19b}, \cite{Olano21}) relating zero-th Hodge ideals to familiar birational invariants.

\begin{defn}

Let $X'$ be a complex manifold and $D'$ an effective $\bQ$-divisor on $X'$. Choose a log resolution $\mu : Y' \to X'$ of the pair $(X',D')$.

\begin{enumerate}[label=\roman*)]

\item (\cite{L04}, Definition 9.2.1). The \emph{multiplier ideal sheaf} associated to $D'$ is defined to be
\[\scJ(X',D'):=\mu_*\scO_{Y'}(K_{Y'/X'}-\lfloor \mu^*D'\rfloor)\subseteq \scO_{X'}.\]

\item (\cite{T13}, Definition 1.6). Assume in addition that $D'$ is a \emph{boundary divisor}, i.e. that all coefficients of $D'$ lie in $[0,1]$. Assume moreover that the strict transform $\widetilde{\lfloor D'\rfloor}$ of $\lfloor D'\rfloor$ along $\mu$ is smooth\footnote{Note that such a log resolution always exists.}. Then the \emph{adjoint ideal sheaf} associated to $D'$ is defined to be 
\[\text{adj}(X',D'):=\mu_*\scO_{Y'}(K_{Y'/X'}+\widetilde{\lfloor D'\rfloor}-\lfloor\mu^*D'\rfloor)\subseteq\scO_{X'}.\]

\end{enumerate}

Both of these objects are indeed independent of the choice of log resolution.
    
\end{defn}

\begin{cor}[Corollary \ref{corpiplus} below]

Assume $\alpha> 0$.

\begin{enumerate}[label=\roman*)]

\item $F_0^HW_n\scM(f^{-\alpha})=\scJ(X,\alpha D)f^{-\alpha}$.\vspace{5pt}

\item $F_0^H\scM(f^{-\alpha})=\scJ(X,(\alpha-\epsilon) D)f^{-\alpha}$, for any rational $0<\epsilon<<1$.\footnote{This is well known, and was proven in \cite{MP19b}.}\vspace{5pt}

\item Write $\left\{\alpha D\right\}:=\alpha D-\lceil\alpha D\rceil +D_{\text{\emph{red}}}$ for the boundary divisor associated to $\alpha D$. Then 
\[F_0^HW_{n+1}\scM(f^{-\alpha})=\text{\emph{adj}}(X,\{\alpha D\})\scO_X(\left\{\alpha D\right\}-\alpha D)f^{-\alpha}.\footnotemark\]
    
\end{enumerate}

\footnotetext{For $\alpha =1$ and $D$ reduced, this was proven in \cite{Olano21}.}

\label{corpiplusintro}
    
\end{cor}

As a consequence, we obtain relations between the mixed Hodge module structure on $\scM(f^{-\alpha})$ and certain singularity ``thresholds" for the pairs $(X,\alpha D)$ as we vary $\alpha$.

\begin{cor}[Corollary \ref{corpair} below]

Assume $\alpha> 0$.

\begin{enumerate}[label=\roman*)]

\item $F_0^H\scM(f^{-\alpha})=\scO_Xf^{-\alpha}$ if and only if the pair $(X,\alpha D)$ is log canonical.\footnote{Again, this was proven in \cite{MP19b}.}

\item $F_0^HW_{n+1}\scM(f^{-\alpha})=\scO_Xf^{-\alpha}$ if and only if the pair $(X,\alpha D)$ is purely log terminal. 

\item $F_0^HW_n\scM(f^{-\alpha})=\scO_Xf^{-\alpha}$ if and only if the pair $(X,\alpha D)$ is Kawamata log terminal. 

\end{enumerate}

\label{corpairintro}

\end{cor}

\noindent The second method we employ to study the Hodge and weight filtrations on $\scM(f^{-\alpha})$ is to use the so-called \emph{specialisability} of $\scM(f^{-\alpha})$. This posits the existence of the \emph{Kashiwara-Malgrange $V$-filtration} associated to $\scM(f^{-\alpha})$, which is in general used to describe the nearby and vanishing cycle sheaves and which is integrally used in the construction of the category of mixed Hodge modules. In Section \ref{sectionVfilt} we will recall details concerning specialisability further. We then obtain a formula for the Hodge and weight filtrations in terms of this $V$-filtration, which will be used essentially throughout the rest of the paper.

Consider the following map, the \emph{graph embedding} of the function $f$:
\[i_f : X \to X\times \bC \,;\, \fx\mapsto(\fx,f(\fx)).\]
The \emph{Kashiwara-Malgrange $V$-filtration} for $\scM(f^{-\alpha})$ along $\{t=0\}\subseteq X\times\bC$ is a rational decreasing filtration $V^{\bullet}$ of sub-$\scD_X[\partial_tt]$-modules on $i_{f,+}\scM(f^{-\alpha})$, which may be uniquely defined axiomatically (see Definition \ref{defnspecial} below). The most important of these axioms is the nilpotency of the operator $\partial_tt-\gamma$ on $\text{gr}_V^{\gamma}i_{f,+}\scM(f^{-\alpha})$, for any $\gamma\in\bQ$. By this nilpotency, we may define, for any $\gamma\in\bQ$, a finite filtration $K_{\bullet}$ on $V^{\gamma}i_{f,+}\scM(f^{-\alpha})$ of sub-$\scD_X[\partial_tt]$-modules, the so-called \emph{kernel filtration}, by
\[K_lV^{\gamma}i_{f,+}\scM(f^{-\alpha}):=\{u\in V^{\gamma}i_{f,+}\scM(f^{-\alpha})\mid(\partial_tt-\gamma)^l\cdot u \in V^{>\gamma}i_{f,+}\scM(f^{-\alpha})\}.\]

We also define another filtration on $i_{f,+}\scM(f^{-\alpha})$, the so-called \emph{$t$-order filtration}, by seeing that there is an isomorphism $i_{f,+}\scM(f^{-\alpha})\simeq \scM(f^{-\alpha})[\partial_t]$, and letting
\[F_k^{t-\text{ord}}i_{f,+}\scM(f^{-\alpha}):=\sum_{i=0}^k\scM(f^{-\alpha})\partial_t^i.\]

Using these two filtrations, we prove a formula for the Hodge and weight filtrations on $\scM(f^{-\alpha})$, which will be very important for all further results in the paper.

\begin{thm}[Theorem \ref{thmmainformula} below]

For any $k,l \in \bZ_{\geq 0}$, we have an equality
\[F_k^HW_{n+l}\scM(f^{-\alpha}) = \psi_0\left(F_k^{t-\text{\emph{ord}}}K_lV^0i_{f,+}\scM(f^{-\alpha})\right).\footnotemark\]

\footnotetext{This theorem is an extension of formulas such as \cite{MP20}, Theorem A', \cite{Olano23}, Theorem A and \cite{CNS22}, Corollary 4.3.}

\noindent Here, $\psi_0:i_{f,+}\scM(f^{-\alpha})\to\scM(f^{-\alpha})$ is the map $\partial_t\mapsto 0$.

\label{thmmainformulaintro}
    
\end{thm}

We remark that Davis and Yang \cite{DY25} have recently obtained a stronger expression for the Hodge filtration, incorporating also the kernel of the map $\psi_0$ appearing above. A weighted analogue of their Corollary 1.2 may also be shown to hold, as a stronger version of the above Theorem \ref{thmmainformulaintro}.

In Section \ref{sectionmicrolocal}, we then use Theorem \ref{thmmainformulaintro} to prove some further general statements about the Hodge and weight filtrations. The Hodge filtration is a good filtration, so has some finite generating level, and the same holds when restricting to weight filtration steps. The weight filtration meanwhile is a finite filtration. So we may ask the following questions:

\begin{itemize}
    \item For $l \in \bZ_{\geq 0}$, what is the generating level of $(W_{n+l}\scM(f^{-\alpha}), F_{\bullet}^H)$?
    \item What is the smallest $l \geq 0$ such that $W_{n+l}\scM(f^{-\alpha})=\scM(f^{-\alpha})$?
\end{itemize}

We obtain (partial) answers to these questions, related to invariants we obtain from the so-called \emph{local Bernstein-Sato polynomial} $b_{f,\hspace{0.7pt}\fx}(s)\in\bC[s]$ of $f$ at a point $\fx\in X$. This is the unique monic polynomial of minimal degree satisfying the functional equation
\[b_{f,\hspace{0.7pt}\fx}(s)f_{\fx}^s\in\scD_{X,\hspace{1pt}\fx}[s]\cdot f_{\fx}^{s+1},\]
where $f_{\fx}\in\scO_{X,\hspace{0.7pt}\fx}$ is the germ of $f$ at $\fx$. The set of roots $\rho_{f,\hspace{0.7pt}\fx}$ and their multiplicities are intrinsically related to the singularity theory and Hodge theory of $f$. Note that all roots are negative rational numbers, and that $-1$ is always a root.

The perhaps most important invariant associated to $b_{f,\hspace{0.7pt}\fx}(s)$ is the \emph{minimal exponent} of $f$ at $\fx$. In general we also introduce the \emph{weighted minimal exponents} of $f$ at $\fx$, to help us in our investigation of the weighted setting.

\begin{defn}

Write $\widetilde{b}_{f,\hspace{0.7pt}\fx}(s):=b_{f,\hspace{0.7pt}\fx}(s)/(s+1) \in \bC[s]$ for the \emph{reduced Bernstein-Sato polynomial of $f$ at $\fx$}. The \emph{minimal exponent} $\widetilde{\alpha}_{f,\hspace{0.7pt}\fx}$ of $f$ at $\fx$ is the smallest root of $\widetilde{b}_{f,\hspace{0.7pt}\fx}(-s)$. For $l\in\bZ_{\geq 0}$, the \emph{$l$-th weighted minimal exponent} $\widetilde{\alpha}_{f,\hspace{0.7pt}\fx}^{(l)}$ of $f$ at $\fx$ is the smallest root of $\widetilde{b}_{f,\hspace{0.7pt}\fx}(-s)$ of multiplicity $\geq l+1$. These are all positive rational numbers.

\label{defnwminlexpintro}
    
\end{defn}

We use Saito's algebraic microlocalisation (\cite{MSai94}) as a technical tool, as a stepping stone between the invariants coming from the Bernstein-Sato polynomial and the mixed Hodge module structure on $\scM(f^{-\alpha})$. Of particular significance are the so-called \emph{microlocal multiplier ideals} (see \cite{MSai17}, (1.5.4), and Definition \ref{defnmicromultideals} below) associated to $f$ (also known as \emph{higher multiplier ideals}, see \cite{SY24}), and their weighted counterparts (Definition \ref{defnmicromultideals} below), as we define in Section \ref{sectionmicrolocal} below. Using microlocalisation, we obtain the following set of theorems related to our above questions.

\begin{prop}[Proposition \ref{prophodgepole2} below]

For $k \in \bZ_{\geq 0}$, $\alpha \in (0,1]$ and $l \in \bZ_{\geq 0}$, we have that
\[F_k^HW_{n+l+\lfloor\alpha\rfloor}\scM(f^{-\alpha})=\scO_Xf^{-k-\alpha} \text{ at }\fx\]
if and only if 
\[\text{either } k+\alpha < \widetilde{\alpha}_{f,\hspace{0.7pt}\fx} \text{ or } k+\alpha =\widetilde{\alpha}_{f,\hspace{0.7pt}\fx} \text{ and }\widetilde{\alpha}_{f,\hspace{0.7pt}\fx} \neq \widetilde{\alpha}_{f,\hspace{0.7pt}\fx}^{(l)}.\]

\label{prophodgepole2intro}
    
\end{prop}

This allows us, via Corollary \ref{corpairintro}, to deduce a relation between the weighted minimal exponents and the thresholds at which the pairs $(X,\alpha D)$ reach certain singularity types.

\begin{cor}[Corollary \ref{corlct} below]

\hfill

\begin{enumerate}[label=\roman*)]

\item The pair $(X,\alpha D)$ is log canonical if and only if $\alpha \leq \widetilde{\alpha}_{f,\hspace{0.7pt}\fx}$.

\item The pair $(X,\alpha D)$ is purely log terminal if and only if $\alpha < \widetilde{\alpha}_{f,\hspace{0.7pt}\fx}$ or $\alpha=\widetilde{\alpha}_{f,\hspace{0.7pt}\fx}$ and $\alpha \neq 1$, $\widetilde{\alpha}_{f,\hspace{0.7pt}\fx}\neq \widetilde{\alpha}_{f,\hspace{0.7pt}\fx}^{(1)}$.

\item The pair $(X,\alpha D)$ is Kawamata log terminal if and only if $\alpha \neq 1$ and $\alpha < \widetilde{\alpha}_{f,\hspace{0.7pt}\fx}$\footnotemark.

\end{enumerate} 

\footnotetext{i) and iii) are already well known. Corollary \ref{corlctintro} can be seen as giving a criterion for when $(X,\text{lct}(X,D)D)$ is a plt pair.}

\label{corlctintro}

\end{cor}

Next, we obtain upper and lower bounds for the highest weight of $\scM(f^{-\alpha})$:

\begin{thm}[Theorem \ref{thmlargestweight} below]

Let $\alpha\in(0,1]$. Write $w_{\text{\emph{max}}}$ for the highest weight of $\scM(f^{-\alpha})$ at $\fx$ (i.e. the largest integer $l$ such that $\text{\emph{gr}}^W_l\scM(f^{-\alpha})\neq 0$ at $\fx$). For $\beta\in\bQ$, write $\widetilde{m}_{f,\beta,\hspace{0.7pt}\fx}$ for the multiplicity of $\beta$ as a root of $\widetilde{b}_{f,\hspace{0.7pt}\fx}(s)$. Then
\[n+\text{\emph{max}}\{\widetilde{m}_{f,-\alpha-i,\hspace{0.7pt}\fx}\,|\,i \in \bZ\} +\lfloor\alpha\rfloor\leq w_{\text{\emph{max}}}\leq n+\sum_{i\geq 0}\widetilde{m}_{f,-\alpha-i,\hspace{0.7pt}\fx}+\lfloor\alpha\rfloor.\]

\label{thmlargestweightintro}
    
\end{thm}

\begin{note}

\vspace{-10pt}
$W_n\scO_X(*f)=\scO_X \neq \scO_X(*f)$ always, as we assume that $f$ is non-invertible, so the bound $n+\lfloor\alpha\rfloor\leq w_{\text{max}}$ is immediate.
    
\end{note}

\begin{rem}

Note also that these bounds of course also imply that in the case that $-\alpha-i$ is the only integer shift of $-\alpha$ that is a root of $\widetilde{b}_{f,\hspace{0.7pt}\fx}(s)$, we actually obtain a precise expression for the maximal weight, as
\[w_{\text{max}}=n+\widetilde{m}_{f,-\alpha-i,\hspace{0.7pt}\fx}+1.\]
This occurs for instance (see Example \ref{egweightbounds} below) when $\alpha=1$ and when $n=2$, when $\alpha=1$ and $f$ defines a \emph{strongly Koszul free divisor}, or when $\alpha=1$ and $f$ defines a (reduced, central) arrangement of hyperplanes.
    
\end{rem}

For the first of our two questions, we obtain an upper bound.

\begin{thm}[Theorems \ref{thmgenlevelW1} and \ref{thmgenlevelW2} below]

Assume $\alpha\in(0,1]$. The generating level of $(\text{\emph{gr}}^W_{n+l}\scM(f^{-\alpha}),F_{\bullet}^H)$ at a point $\fx\in X$ is less than or equal to 
\[n-l-\lceil\alpha+\widetilde{\alpha}_{f,\hspace{0.7pt}\fx}\rceil+1.\]

\label{thmgenlevelW1intro}
    
\end{thm}

From this we may then deduce the following.

\begin{cor}[Corollary \ref{corgenlevelW} below]

Assume $\alpha\in(0,1]$. The generating level of $(W_{n+l}\scM(f^{-\alpha}),F_{\bullet}^H)$ at $\fx$ is less than or equal to
$n-\lceil\alpha+\widetilde{\alpha}_{f,\hspace{0.7pt}\fx}\rceil+1-\lfloor\alpha\rfloor$.

\label{corgenlevelWintro}
    
\end{cor}

Finally, we also use the formula given in Theorem \ref{thmmainformulaintro} to obtain expressions for the Hodge and weight filtrations for several classes of examples.

In Section \ref{subsectionegs}, we first give a formula for both of these filtrations in the case that $f$ has simple normal crossing singularities (Theorem \ref{thmSNCFW2}), which is used essentially in Section \ref{sectionlogres} to obtain the expression given in Theorem \ref{thmpiplusintro}. In Section \ref{subsectionegs} we also give an expression for when $f$ has weighted homogeneous isolated singularities (Theorem \ref{thmWHom}), extending the work of \cite{MSai09} and \cite{Z21}.

We also in Section \ref{sectionPPD} attain a formula for the Hodge and weight filtrations for a much larger class of examples. These functions have in particular to be \emph{parametrically prime}. This is an algebraic condition without any geometric intuition, but many classes of examples have been shown to satisfy this hypothesis. This extends results proven in \cite{CNS22}, \cite{BD24} and \cite{D24}.

\begin{defn}\label{defnPPintro}

\hfill

\begin{enumerate}[label=\alph*)]
    \item We say that $f$ is \emph{Euler homogeneous at $\fx\in X$} if there exists some $E \in\text{Der}_{\bC}(\scO_{X,\hspace{0.7pt}\fx})$ such that $E(f_{\fx})=f_{\fx}$, where $f_{\fx}\in\scO_{X,\hspace{1pt}\fx}$ is the germ of $f$ at $\fx\in X$.
    \item We say that $f$ is \emph{parametrically prime at $\fx\in X$} if $\text{gr}^{\sharp}(\text{ann}_{\scD_{X,\hspace{1pt}\fx}[s]} f_{\fx}^{s-1}) \subseteq \text{gr}^{\sharp}\scD_{X,\hspace{1pt}\fx}[s]$ is prime, where $\text{gr}^{\sharp}(\text{ann}_{\scD_{X,\hspace{1pt}\fx}[s]} f_{\fx}^{s-1})$ is the ideal of symbols of elements of $\text{ann}_{\scD_{X,\hspace{1pt}\fx}[s]} f_{\fx}^{s-1}$ with respect to the \emph{total order filtration} $F_k^{\sharp}\scD_{X,\hspace{1pt}\fx}[s]:=\sum_{i=0}^k(F_{k-i}\scD_{X,\hspace{1pt}\fx})s^i$ on $\scD_{X,\hspace{1pt}\fx}[s]$. 
\end{enumerate}

\end{defn}

\noindent Now, write 
\[\beta_{f,-\alpha,\hspace{0.7pt}\fx}(s):=\prod_{\lambda \in \rho_{f,\hspace{0.7pt}\fx}\cap (-\alpha-1,-\alpha)}(s+\lambda+1)^{m_{f,\lambda,\hspace{0.7pt}\fx}},\]
where $m_{f,\lambda,\hspace{0.7pt}\fx}$ is the multiplicity of $\lambda$ as a root of $b_{f,\hspace{0.7pt}\fx}(s)$. We then define the following $\scD_{X,\hspace{1pt}\fx}[s]$-ideal:
\[\Gamma_{f,-\alpha,\hspace{0.7pt}\fx}:=\scD_{X,\hspace{1pt}\fx}[s]f_{\fx}+\scD_{X,\hspace{1pt}\fx}[s]\beta_{f,-\alpha,\hspace{0.7pt}\fx}(-s)+\text{ann}_{\scD_{X,\hspace{1pt}\fx}[s]}f_{\fx}^{s-1}\subseteq \scD_{X,\hspace{1pt}\fx}[s],\]
and equip it with a finite filtration by sub-ideals as follows:
\[W_0\Gamma_{f,-\alpha,\hspace{0.7pt}\fx}:=\Gamma_{f,-\alpha-\epsilon,\hspace{0.7pt}\fx}\subseteq \Gamma_{f,-\alpha,\hspace{0.7pt}\fx}, \,\,\,\,\,0<\epsilon<<1,\]
and, for $l \in\bZ_{\geq 0}$,
\[W_l\Gamma_{f,-\alpha,\hspace{0.7pt}\fx}:=\{\gamma \in \Gamma_{f,-\alpha,\hspace{0.7pt}\fx}\mid (s+\alpha)^l\gamma\in W_0\Gamma_{f,-\alpha,\hspace{0.7pt}\fx}\}\subseteq \Gamma_{f,-\alpha,\hspace{0.7pt}\fx}.\]
Then we prove the following:

\begin{thm}[Theorem \ref{thmgammaW} below]

Assume $\alpha\geq 0$. Let $\fx\in X$ and assume that $\rho_{f,\hspace{0.7pt}\fx}\subseteq(-2-\alpha,-\alpha)$. Let $l \in\bZ_{\geq 0}$. Then

\begin{enumerate}[label=\roman*)]

\item \, $W_{n+l}\scM(f^{-\alpha})_{\fx}=\phi_{-\alpha}(W_l\Gamma_{f,-\alpha,\hspace{0.7pt}\fx})\cdot f_{\fx}^{-1-\alpha}$, where $\phi_{-\alpha}:\scD_{X,\hspace{1pt}\fx}[s] \to \scD_{X,\hspace{1pt}\fx}\, ;\, P(s) \mapsto P(-\alpha)$.

\item \, $F_0^HW_{n+l}\scM(f^{-\alpha})_{\fx}=\left(W_l\Gamma_{f,-\alpha,\hspace{0.7pt}\fx}\cap\scO_{X,\hspace{0.7pt}\fx}\right)\cdot f_{\fx}^{-1-\alpha}$.

\end{enumerate}

\label{thmgammaWintro}
    
\end{thm}

At least for $\alpha=0$, there are many classes of examples satisfying the above hypothesis. All hyperplane arrangements satisfy this, as do \emph{strongly Koszul free divisors} for instance.

Under some somewhat stronger hypotheses, we obtain an analogous expression for every step of the Hodge filtration.

\begin{thm}[Theorem \ref{thmPPformulaW} below]

Assume that $\alpha\geq 0$. Assume that $f$ is Euler homogeneous and parametrically prime at a point $\fx\in X$, and assume that $\rho_{f,\hspace{0.7pt}\fx}\subseteq(-2-\alpha,-\alpha)$. Then, for all $k \in \bZ$ and $l\in\bZ_{\geq 0}$,
\[F_k^HW_{n+l}\scM(f^{-\alpha})_{\fx}=\phi_{-\alpha}(W_l\Gamma_{f,-\alpha,\hspace{0.7pt}\fx}\cap F_k^{\sharp}\scD_{X,\hspace{1pt}\fx}[s])\cdot f_{\fx}^{-1-\alpha}.\]

\label{thmPPformulaWintro}
    
\end{thm}

Again, there are many classes of examples satisfying these hypotheses. For instance, any function of \emph{linear Jacobian type} is parametrically prime and Euler homogeneous, and any function that is \emph{tame}, \emph{Saito holonomic} and \emph{strongly Euler homogeneous} is of linear Jacobian type. So strongly Koszul free divisors and tame hyperplane arrangements for instance satisfy all necessary hypotheses (with $\alpha=0$) and we may apply the above theorem to calculate their Hodge and weight filtrations. See \cite{BD24} and \cite{D24} for further details, as well as Section \ref{sectionPPD} below.

The structure of the paper is as follows. In Section 2 we record some of the basics of mixed Hodge modules, as well as the main theorems and constructions we will require in this paper. We then investigate the left $\scD_X$-module $\scM(f^{-\alpha})$ and explain how to equip it with a canonical complex mixed Hodge module structure. Then we choose a log resolution and investigate the relation between the mixed Hodge module structure on $\scM(f^{-\alpha})$ and on the pullback. In Section 3 we recall the notions of specialisability and quasi-unipotency, and discuss nearby and vanishing cycle functors. We then go on to use these fundamentals to prove the formula for the Hodge and weight filtrations given in Theorem \ref{thmmainformulaintro} above. In Section 4 we introduce microlocalisation, (weighted) minimal exponents and (weighted) microlocal multiplier ideals. We derive some useful formulae and results for each of these invariants and then use what we have learnt to prove Theorems \ref{thmlargestweightintro} and \ref{thmgenlevelW1intro} as given above, concerning highest weights and generating levels. We then conclude in Section 5 by calculating the Hodge and weight filtration for certain parametrically prime divisors as above described.

The author would like to thank his supervisor Christian Sevenheck for introducing him to the subject and for numerous helpful discussions regarding the material of this paper.

\newpage

\section{Hodge filtrations via log resolutions}\label{sectionlogres}

\noindent We begin with a global investigation of the canonical filtrations $F_{\bullet}^H$ and $W_{\bullet}$ associated to the left $\scD_X$-module $\scM(f^{-\alpha})$. We first explain how to construct these filtrations functorially, i.e. how to put a (complex) mixed Hodge module structure on $\scM(f^{-\alpha})$. We then describe how one may calculate these filtrations using log resolutions, and obtain some expressions for certain steps of these filtrations, all in the vein of the analysis carried out in papers such as \cite{MP19a}, \cite{MP19b}, \cite{Olano21}, \cite{Olano23}.

\vspace{10pt}

\subsection{Hodge module prerequisites}

\; \vspace{5pt} \\ We begin by recalling some key statements concerning complex mixed Hodge modules and twisted mixed Hodge modules that will be required in this article, for the convenience of the reader. The main sources are Saito's original two papers on mixed Hodge modules \cite{MSai88} and \cite{MSai90}, the (at the present incomplete) mixed Hodge module project \cite{SS24} of Sabbah and Schnell, which develops the theory in the complex setting, and the work of papers such as \cite{SV11}, \cite{DV23}, \cite{DV24} and \cite{SY24} which introduce and develop the theory of twisted mixed Hodge modules. Another useful source we will use throughout is Mochizuki's \cite{M15a}, which (in particular) develops the theory of mixed Hodge modules in the real setting and which we use to fill in some of the current gaps in \cite{SS24}.

Let $X$ be a complex manifold of dimension $n$. We will write $\text{MH}(X,\bC)$, or sometimes $\text{MH}(X)$ (or $\text{MH}(X,\bC,w)$, $\text{MH}(X,w)$ when we wish to fix the weight $w$), for the category of complex (pure) Hodge modules on the manifold $X$. We will write $\text{MHM}(X,\bC)$, or $\text{MHM}(X)$, for the category of complex mixed Hodge modules on $X$. $\text{MH}(X,\bQ)$ and $\text{MHM}(X,\bQ)$ will denote the categories of rational (pure) Hodge modules and rational mixed Hodge modules. Note that polarisations won't be required for the most part in this paper so polarised / polarisable mixed Hodge modules won't be mentioned. Note also that unless stated otherwise, $\scD$-modules are to be assumed to be left.

\begin{prop}[\cite{SS24}, Theorem 14.6.1]

Any polarisable variation of $\bC$-Hodge structure on $X$ of weight $m$ (see Section 4 of \cite{SS24} for the definition (4.1.4) and fundamental results) has a canonical structure as a complex pure Hodge module of weight $m+n$ on $X$. 

\label{propVHS}
    
\end{prop}

\begin{eg}

The most important example in this paper is as follows. Let $\beta\in\bQ$ and consider the vector bundle with (flat, holomorphic) connection $(\scO_{\bC^*},\nabla^{\beta})$ on the complex manifold $\bC^*$, with $\nabla^{\beta}$ defined by
\[\nabla^{\beta} : \scO_{\bC^*} \to \Omega_{\bC^*}^1 \,\, ;\, \, g \mapsto \left(\frac{\partial g}{\partial t} + \beta\frac{g}{t}\right)dt,\]
where $t$ is the natural coordinate function on $\bC^*$. 

We write $\bV^{\beta}$ for the local system associated to this vector bundle with connection, i.e. $\bV^{\beta}:=\ker\nabla^{\beta}$. This is the local system of rank 1 on the complex manifold $\bC^*$ with monodromy eigenvalue $e^{-2\pi i\beta}$, i.e. the local system defined by the monodromy representation
\[\rho^{\beta} : \pi_1(\bC^*) \simeq \bZ \to \text{Aut}(\bC) \simeq \bC^* \,\, ;\, \, m \mapsto e^{-2\pi i \beta m}.\]

Note firstly that $\bV^{\beta}$ is a \emph{rational} local system if and only if $\beta \in \bZ$, otherwise it has only the structure of a complex-valued local system. See secondly that this local system does in fact satisfy that its monodromy eigenvalue is a root of unity (in general these eigenvalues must only be of absolute value 1).

We write $\scO^{\beta}$ for the left $\scD_{\bC^*}$-module associated to the vector bundle with connection $(\scO_{\bC^*},\nabla^{\beta})$. We may write this as
\[\scO^{\beta} = \scO t^{\beta}, \text{ with action given by } \partial_t \cdot (gt^{\beta}) = \frac{\partial g}{\partial t}t^{\beta} +\beta g t^{\beta-1},\]
where $\scO:=\scO_{\bC^*}$.

$(\scO,\nabla^{\beta})$ induces a variation of $\bC$-Hodge structure on $\bC^*$ of weight 0, by giving $(\scO,\nabla^{\beta})$ the trivial Hodge filtration $\text{gr}^F_0\scO^{\beta} = \scO^{\beta}$. 

We also consider the $\scO\otimes_{\bC}\overline{\scO}$-linear morphism $S^{\beta}:\scO\otimes_{\bC}\overline{\scO} \to \scC^{\infty}_{\bC^*}$ given \emph{locally} by
\[S^{\beta}:\scO\otimes_{\bC}\overline{\scO} \to \scC^{\infty}_{\bC^*}\,;\, g\otimes \overline{g'}\mapsto g\overline{g'}t^{\beta},\]
where $g,g'\in\scO$. It is then easy to see (again, locally) that
\[S^{\beta}(\nabla^{\beta}g,\overline{g'})=d'S(g,\overline{g'}),\]
so that $S$ induces a polarisation on the aforementioned variation of $\bC$-Hodge structure (see \cite{SS24}, Lemma 4.1.3 and Remark 4.1.8).

Using Proposition \ref{propVHS}, we thus induce a complex pure Hodge module of weight 1 on $\bC^*$, which we will denote $\underline{\bC}^{\beta}[1]$. This overlies a \emph{rational} pure Hodge module if and only if $\beta \in \bZ$ (in which case we obtain the trivial Hodge module). Moreover, we have in general a natural isomorphism of Hodge modules $\underline{\bC}^{\beta}[1] \simeq \underline{\bC}^{\beta +1}[1]$.

\label{eglocsys}
    
\end{eg}

\begin{prop}

Let $f : X\to Y$ be a holomorphic map between complex manifolds. 

\begin{enumerate}[label =\roman*)]

\item There exist functors
\[f^*, f^! : D^{\text{\emph{b}}}\text{\emph{MHM}}(Y,\bC) \to D^{\text{\emph{b}}}\text{\emph{MHM}}(X,\bC)\]
lifting the analogous functors of constructible complexes. These functors respect composition of functions (i.e., given $g:Y \to Z$, $(g\circ f)^* = f^*\circ g^*$ and $(g\circ f)^! = f^!\circ g^!$).

\item \emph{(\cite{M15a}\footnote{Throughout, we will often be content with citing \cite{M15a} when referencing general results concerning complex mixed Hodge modules. This suffices due to the alternative definition of complex mixed Hodge modules given for instance in \cite{DS13} (Definition 3.2.1) as the summand of a real mixed Hodge module. Real mixed Hodge modules are related to mixed twistor $\scD$-modules through the functor described in Section 13.5 of \cite{M15a} which preserves all relevant functors. The essential image of this functor is described in \cite{M15b}, Theorem 3.35. \newline \indent Of course real mixed Hodge modules are somewhat simpler (being stronger) than twistor $\scD$-modules and completely analogous to Saito's rational mixed Hodge modules, and thus all results should follow through analogously. We cite Mochizuki only to be as rigorous as possible.}, Proposition 7.2.7).} If $f$ is projective, then there exists a functor
\[f_*= f_! : D^{\text{\emph{b}}}\text{\emph{MHM}}(X,\bC) \to D^{\text{\emph{b}}}\text{\emph{MHM}}(Y,\bC)\]
lifting the analogous functor of constructible complexes. This functor also respects composition of functions. Note also that the functor induces functors
\[R^kf_*=H^k\circ f_*:\text{\emph{MH}}(X,\bC,w)\to\text{\emph{MH}}(X,\bC,w+k).\]

\item \emph{(\cite{M15a}, Theorem 13.3.1).} There is an (exact) duality functor 
\[\bD_X : \text{\emph{MHM}}(X,\bC) \to \text{\emph{MHM}}(X,\bC)\] 
lifting the Verdier duality functor on perverse sheaves, and which exchanges $f_*$ with $f_!$ and $f^*$ with $f^!$ (when $f_*$ and $f_!$ exist).
    
\end{enumerate}

\label{propMHMfunctors}
    
\end{prop}

\begin{proof}

We make a quick comment as to one way to construct the functors in part a), as we were unable to find a source for this result. Assuming the existence of the duality functor, it suffices to construct $f^*$ say. By decomposing the function $f$ into the composition of the maps
\[i_f : X \to X\times Y ; x \mapsto (x,f(x)), \,\,\,\,\, p_2 : X\times Y \to Y ; (\fx,\fy)\mapsto \fy,\]
we see that it suffices to construct $f^*$ in the case that $f$ is a closed embedding and the case that $f$ is a projection. 

For a projection $f : X=Z\times Y \to Y$, we define 
\[f^*M:=\underline{\bC}_Z[n]\boxtimes M\in\text{MHM}(X),\]
where $\underline{\bC}_Z[n]$ is the pure Hodge module of weight $n$ on $Z$ induced by the trivial variation of $\bC$-Hodge structure on $Z$ of weight 0 (Proposition \ref{propVHS}), and $-\boxtimes -$ denotes the external tensor product, which preserves mixed Hodge modules (see \cite{M15a}, Lemma 11.4.12).

For $f=i$ a closed embedding, we may either assume the existence of projective pushforward and of left and right adjoints of projective pushforward, or we may construct the pullback functors manually as follows. Since the statement is local, by decomposing $i$ we may assume that it is of the form $i:X=\{g=0\} \to Y$ for $g:Y \to \bC$ some non-zero non-invertible holomorphic function with no critical values. Then, using Kashiwara's equivalence, we may define $i^*M$ and $i^!M$ (for $M\in \text{MHM}(Y)$) to be the unique objects of $D^{\text{b}}\text{MHM}(X)$ such that the following exact triangles exist:
\[i_*i^!M\rightarrow\phi_{g,1}M\xrightarrow{\text{var}} \psi_{g,1}M(-1)\xrightarrow{+1},\]
\[\psi_{g,1}M\xrightarrow{\text{can}} \phi_{g,1}M\rightarrow i_*i^*M\xrightarrow{+1}.\vspace{5pt}\]
\noindent The existence of $\phi_{g,1}M, \psi_{g,1}M \in \text{MHM}(Y)$ is a fundamental condition in the definition of mixed Hodge modules. We'll discuss these objects in further detail in Section 3 when we need them (see Theorem \ref{thmphipsiMHM}).
    
\end{proof}

\begin{rem}

In the algebraic category, the direct image functors $f_*$ and $f_!$ exist also when $f$ is not projective. The existence of $j_*$ and $j_!$ for an open embedding in the analytic category depends upon the mixed Hodge module one wishes to apply these functors to. However, we do have the following:
    
\end{rem}

\begin{prop}[\cite{M15a}, Proposition 11.2.2 and Corollary 11.2.10] 

Let $j:U \hookrightarrow X$ be an open embedding of complex manifolds, such that the complement $Z$ of $U$ is a hypersurface. Let $M$ be a complex mixed Hodge module on $X$. Then the functors $j_*$ and $j_!$ are well-defined on $j^*M$\footnote{See Remark \ref{remjjtriangles} below for the meaning of this statement.}. If $\scM$ is the underlying $\scD_X$-module of $M$, then the underlying $\scD_X$-modules of $j_*j^*M$ and $j_!j^*M$ are 
\[\scM\otimes_{\scO_X}\scO_X(*Z) \,\,\, \text{ and }\,\,\, \bD\left((\bD\scM)\otimes_{\scO_X}\scO_X(*Z)\right)\]
respectively, where $\scO_X(*Z)$ is the sheaf of meromorphic functions on $X$ which are holomorphic along $U$.

\label{proplocalisable}
    
\end{prop}

\begin{note}

Note that for an open embedding $j:U \to X$, $j^*=j^!$.
    
\end{note}

\begin{rem}

Note that in this paper, we will effectively define $j_*j^*M$ and $j_!j^*M$ to complete the following two exact triangles;
\[i_*i^!M\to M \to j_*j^*M \xrightarrow{+1}\]
and
\[j_!j^*M\to M \to i_*i^*M \xrightarrow{+1}.\]

\vspace{2pt}\noindent Here, $i:Z \hookrightarrow X$ is the natural closed embedding of $Z$ into $X$. $Z$ is of course not in general a manifold, but we may define the functors $i_*i^!$ and $i_*i^*$ nevertheless; choose locally a reduced function $f:X\to \bC$ whose zero locus equals $Z$, and consider the Cartesian diagram

\begin{center}
\begin{tikzcd}
Z \arrow[r, "i"] \arrow[d, "i"] & X \arrow[d,"i_f"]\\
X \arrow[r, "i'"] & X\times\bC, \\   
\end{tikzcd}\vspace{-17pt}
\end{center}

\noindent where $i'(\fx)=(\fx,0)$ and $i_f(\fx)=(\fx,f(\fx))$. It is then a simple exercise to show that 
\[\psi_{t,1}i_{f,*}M\simeq i_*'\psi_{f,1}M \,\,\,\text{ and }\,\,\,\phi_{t,1}i_{f,*}M\simeq i_*'\phi_{f,1}M,\]
implying by Kashiwara's equivalence and the exact triangles appearing in the proof of Proposition \ref{propMHMfunctors} that we have exact triangles
\[i'^!i_{f,*}M\rightarrow\phi_{f,1}M\xrightarrow{\text{var}} \psi_{f,1}M(-1)\xrightarrow{+1},\]
\[\psi_{f,1}M\xrightarrow{\text{can}} \phi_{f,1}M\rightarrow i'^*i_{f,*}M\xrightarrow{+1}.\vspace{5pt}\]
Motivated by these triangles, we then define (locally)
\[i_*i^!:=i'^!i_{f,*}\,\,\,\text{ and }\,\,\,i_*i^*:=i'^*i_{f,*}.\]

\label{remjjtriangles}

\end{rem}

\begin{note}

As a final comment, see that the setting of Proposition \ref{proplocalisable} is indeed the only situation in which a mixed Hodge module on $U$ is ``extendable" over the open embedding $j$, i.e. for a complex mixed Hodge module $N$ on $U$, $j_*N, j_!N \in \text{MHM}(X)$ ``exist" if any only if there exists some $M\in\text{MHM}(X)$ such that $N=j^*M$. 
    
\end{note}

\noindent We may also always extend variations of $\bC$-Hodge structure over open embeddings.

\begin{prop}[\cite{SS24}, Theorem 16.2.1]

Let $j :U \to X$ be an open embedding of complex manifolds such that the complement $Z$ of $U$ in $X$ is a hypersurface. Let $M$ be the complex pure Hodge module of weight $m$ overlying a polarisable variation of $\bC$-Hodge structure on $U$. Then there is a unique complex pure Hodge module of weight $m$ on $X$, which we will denote $j_{!*}M$, satisfying that $j^*j_{!*}M \simeq M$ and that $j_{!*}M$ has no submodules or quotients supported on $Z$\footnote{This is equivalent to requiring that $j_{!*}M$ has \emph{strict support} $X$.}. This will be called the \emph{minimal extension} of $M$ over $j$, and its underlying $\scD_X$-module is indeed the minimal extension over $j$ of the $\scD_U$-module underlying $M$. 

\label{propextension}
    
\end{prop}

\begin{rem}

Under the assumptions of Proposition \ref{propextension}, note that we must have that $j_*$ and $j_!$ thus exist on $M$. Note also that since the complement of $U$ is a hypersurface, then $j_*$ and $j_!$ are exact and thus $j_*M, j_!M \in \text{MHM}(X,\bC)$ (i.e. they're concentrated in degree zero as elements of the derived category). Moreover, since $j^*W_{\text{min}}j_*M=M$, and since all non-zero $\scD_X$-submodules of the $\scD_X$-module underlying $j_*M$ may be shown (using Proposition \ref{proplocalisable}) to have support $X$,  the uniqueness of the minimal extension implies in this case that the lowest non-zero weight filtration step of $j_*M$ is $W_mj_*M$, and that
\[W_mj_*M \simeq j_{!*}M.\]

\label{remWmin}
    
\end{rem} 

\vspace{10pt}

\noindent We will discuss further details concerning complex mixed Hodge modules, in particular the existence and properties of nearby and vanishing cycles for complex mixed Hodge modules, in Section \ref{sectionVfilt}. In Section \ref{sectionlogres} we will only need the above facts, as in this section we're only interested in putting a mixed Hodge module structure on $\scM(f^{-\alpha})$.

The only further topic of discussion is twisted mixed Hodge modules, which we will use to exhibit a more global construction of the mixed Hodge module structure on $\scM(f^{-\alpha})$ we will soon define. Twisted $\scD$-modules were first investigated by Beilinson and Bernstein in \cite{BB93}. See also Chapter 11 of \cite{HTT08}. The following short summary is cited from \cite{SY24}.

\begin{defn}

Let $X$ be a complex manifold and $\scL$ a holomorphic line bundle on $X$. Write $p:L\to X$ for the total space of the line bundle $\scL$. Write $\scI_X \subseteq \scO_L$ for the ideal sheaf of the zero section in $L$. Define a decreasing filtration of $\scO_L$-modules 
\[V^j\scD_L = \{P\in\scD_L\,:\, P\cdot \scI_X^i \subseteq \scI_X^{i+j} \,\,\,\forall i \in\bZ_{\geq 0}\}.\]
Write $\theta$ for the vector field tangent to the natural $\bC^*$ action on $L$ (mapping to $t\partial_t$ in any trivialisation). Then $\theta \in V^0\scD_L$.

Let $\beta \in \bQ$. The sheaf of \emph{$\beta\scL$-twisted differential operators} on $X$ is then defined to be the sheaf of rings
\[\scD_{X,\hspace{0.5pt}\beta\scL}:=\text{gr}_V^0\scD_L / (\theta-\beta)\text{gr}_V^0\scD_L.\]

Note that in any trivialisation $L|_U \simeq U \times\bC_t$ of the line bundle $\scL$, 
\[\scD_{X,\hspace{0.5pt}\beta\scL}|_U\simeq\scD_U[t\partial_t] / (t\partial_t-\beta)\scD_U[t\partial_t] \simeq \scD_U.\]
Moreover, in a change of trivialisation given by $g\in\scO_U^*$, the differential operators change via 
\begin{equation}\label{eqtrivtran}\partial_i \mapsto \partial_i + \beta g^{-1}\frac{\partial g}{\partial x_i}.\end{equation}

A \emph{$\beta\scL$-twisted $\scD$-module} on $X$ is a module over the ring $\scD_{X,\hspace{0.5pt}\beta\scL}$. If $\scM$ is such a module, then on any trivialisation $L|_U \simeq U \times\bC_t$, $\scM|_U$ is naturally a $\scD_U$-module under the above isomorphism. Note also that an equivalent definition for $\beta\scL$-twisted $\scD$-module is as a collection of $\scD$-modules indexed by trivialisations of $\scL$, with appropriate gluing data. Finally, see that if $\scM$ is a $\beta\scL$-twisted $\scD$-module on $X$, then $\scL \otimes \scM$ is a $(\beta-1)\scL$-twisted $\scD$-module on $X$.

\label{defntwDmod}
    
\end{defn}

\begin{defn}

Take $X$, $\scL$ and $\beta$ as above. A \emph{$\beta\scL$-twisted mixed Hodge module} on $X$ is a collection of complex mixed Hodge modules indexed by the trivialisations of the line bundle $\scL$, transforming on intersections (by which we mean the underlying bi-filtered $\scD_U$-module transforms) as according to Equation \eqref{eqtrivtran} above. Such a twisted mixed Hodge module is \emph{pure} if it is pure on every trivialisation. We write $\text{MHM}(X,\beta\scL)$ and $\text{MH}(X,\beta\scL)$ for the categories of $\beta\scL$-twisted mixed and pure Hodge modules respectively. Note that when $\beta=0$ we recover $\text{MHM}(X,\bC)$ and $\text{MH}(X,\bC)$.

\label{defntwMHM}
    
\end{defn}

\begin{prop}

Let $f:X \to Y$ be a projective map between complex manifolds and $\beta \in \bQ$, and let $\scL$ be a holomorphic line bundle on $Y$. Write $\scL_X:=f^*\scL$. Then there exists a functor
\[f_*=f_! : D^{\text{\emph{b}}}\text{\emph{MHM}}(X,\beta\scL_X) \to D^{\text{\emph{b}}}\text{\emph{MHM}}(Y,\beta\scL)\]
restricting to the corresponding functor on trivialisations. These functors again respect composition of functions.

\label{proptwMHMfunctors}

\end{prop}

\begin{proof}

This follows immediately from \ref{propMHMfunctors}.ii). See \cite{SY24}, Theorem 3.8 for the statement of this result in the pure case, as well as the statement of the decomposition theorem for filtered twisted $\scD$-modules underlying twisted pure Hodge modules along projective morphisms.
    
\end{proof}

\noindent See \cite{DV24} and \cite{SY24} especially for further detail about twisted mixed Hodge modules; they won't be used at all extensively in this paper.

\vspace{10pt}

\subsection{The complex mixed Hodge module structure on $\scM(f^{-\alpha})$}

\; \vspace{5pt} \\ Now we exhibit how to endow the left $\scD_X$-module $\scM(f^{-\alpha})$ with a canonical complex mixed Hodge module structure. In fact using twisted mixed Hodge modules we are able to globalise this construction. This mixed Hodge module only has a rational structure in the case that $\alpha$ is an integer, although, as we will see, it is always a direct summand of a rational mixed Hodge module, which implies that $\scM(f^{-\alpha})$ has quasi-unipotent monodromy and thus is still in fact $\bQ$-specialisable. 

Again $X$ is a complex manifold of dimension $n$ and $f\in\scO_X$ is a holomorphic function on $X$ such that the zero locus $Z:=f^{-1}(0)\subseteq X$ is a hypersurface, and $U:=X\backslash Z$. We take also a rational number $\alpha \in\bQ$ and write $\overline{\alpha}:=-\alpha +\lceil\alpha\rceil\in[0,1)$.

\begin{defn}

We define two left $\scD_X$-modules as follows. 
\[\scM(f^{-\alpha}):=\scO_X[f^{-1}]f^{-\alpha} \,\,\text{ with action given by } \delta \cdot (uf^{-\alpha}) := \delta(u)f^{-\alpha}-\alpha u\delta(f) f^{-\alpha-1},\]
where $\delta \in\text{Der}_{\bC}(\scO_X)$ and $u \in\scO_X[f^{-1}]$.
\[\scO_{X,-\alpha f}:= \scD_X \cdot f^{\,\overline{\alpha}} \subseteq \scM(f^{-\alpha}).\] We will write $\scO_X(*f)$ for $\scM(f^{\,0}) \simeq \scO_X[f^{-1}]$.

\vspace{5pt}

\noindent Now, write $j : U \hookrightarrow X$ for the natural open embedding and consider the Cartesian diagram
\begin{center}
\begin{tikzcd}
U \arrow[r, "j"] \arrow[d, "\widetilde{f}"] & X \arrow[d, "f"]\\
\bC^* \arrow[r, "\widetilde{j}"] & \bC.
\end{tikzcd}
\end{center}

\noindent Recalling the left $\scD_{\bC^*}$-module $\scO^{-\alpha}$ defined in the previous subsection, we define 
\[\scO_{U, -\alpha f} := \widetilde{f}^*\scO^{-\alpha}.\]
Note that this only depends upon the fractional part of $\alpha$ (i.e. on $\overline{\alpha}$).

See also that the left $\scD_U$-module $\scO_{U,-\alpha f}$ may be written 
\[\scO_{U,-\alpha f} = \scO_Uf^{-\alpha}, \,\,\text{ with action }\, \delta \cdot (uf^{-\alpha}) = \delta(u)f^{-\alpha} -\alpha u\delta(f)f^{-\alpha-1}.\]

\label{defnDmods}

\end{defn}

\begin{lem}

These two left $\scD_X$-modules satisfy the following elementary properties.

\begin{enumerate}[label=\roman*)]

\item $\scM(f^{-\alpha}) \simeq \scO_{X,-\alpha f}\otimes_{\scO_X}\scO_X(*f)$ as left $\scD_X$-modules. 

\item $\scO_{X,-\alpha f}$ is the minimal extension of $\scO_{U,-\alpha f}$ along the open embedding $j$.

\item $\scO_{X,-\alpha f} = \scD_X \cdot f^{k+\overline{\alpha}}$ for any $k \in\bZ_{\geq 0}$.

\item Let $\fx\in X$ such that $f(\fx)=0$. Then $\scO_{X,-\alpha f,\hspace{0.7pt}\fx} = \scM(f^{-\alpha})_{\fx}$ if and only if $b_{f,\hspace{0.7pt}\fx}(l-\alpha) \neq 0$ for all $l \in\bZ$, where $b_{f,\hspace{0.7pt}\fx}(s)$ is the Bernstein-Sato polynomial of $f$ at $\fx$ (see Section \ref{sectionmicrolocal} for further details).
    
\end{enumerate}

\label{lemDmodprops}
    
\end{lem}

\begin{proof}

\begin{enumerate}[label=\roman*)]

\item This is elementary, an isomorphism is given by 
\[\scO_{X,-\alpha f}\otimes_{\scO_X}\scO_X(*f) \to \scM(f^{-\alpha}) \,\,;\,\, P\cdot f^{\,\overline{\alpha}}\otimes u \mapsto uP\cdot f^{\,\overline{\alpha}}.\]

\item One definition of the minimal extension of $\scO_{U,-\alpha f}$ along the open embedding $j$ is as the smallest (with respect to inclusion) left $\scD_X$-submodule of (the non-coherent) $j_+\scO_{U,-\alpha f}$ which pulls back along $j$ to $\scO_{U,-\alpha f}$.

Write $\scN\subseteq\scM(f^{-\alpha}) \subseteq j_+\scO_{U,-\alpha f}$ for the minimal extension of $\scO_{U,-\alpha f}$ along $j$. Then, since $\scN$ restricts to $\scO_{U,-\alpha f}$ on $U$, $f^{k-\alpha} \in \scN$ for some $k \in\bZ$. 

Now, for each $\fx\in X$, write $b_{f,\hspace{0.7pt}\fx}(s)$ for the Bernstein-Sato polynomial of $f$ at $\fx$, i.e. the monic polynomial of minimal degree such that there exists $P_{\fx}(s)\in\scD_{X,\hspace{1pt}\fx}[s]$ with
\[P_{\fx}(s)\cdot f_{\fx}^{s+1}=b_{f,\hspace{0.7pt}\fx}(s)f_{\fx}^s,\]
where $f_{\fx}\in\scO_{X,\hspace{0.7pt}\fx}$ is the germ of $f$ at $\fx$.

Then it is well-known that the roots of the Bernstein-Sato polynomial are all negative rational numbers. Therefore, for $k \in\bZ_{\geq 1}$, we may write
\[f^{\overline{\,\alpha}}_{\fx} = \left(\prod_{l=0}^{k-1}b_{f,\hspace{0.7pt}\fx}(l+\overline{\alpha})^{-1}P_{\fx}(l+\overline{\alpha})\right) \cdot f_{\fx}^{k+\overline{\alpha}},\]
implying that
\[\scD_{X,\hspace{1pt}\fx}\cdot f_{\fx}^{\,\overline{\alpha}} = \scD_{X,\hspace{1pt}\fx}\cdot f_{\fx}^{k+\overline{\alpha}},\]
which in turn gives the global statement
\[\scD_X\cdot f^{\,\overline{\alpha}} = \scD_X\cdot f^{k+\overline{\alpha}}.\]
Therefore we see that the smallest $\scD_X$-submodule of $\scM(f^{-\alpha})$ restricting to $\scO_{U,-\alpha f}$ is 
\[\scN=\scD_X\cdot f^{k+\overline{\alpha}} = \scD_X\cdot f^{\,\overline{\alpha}}\]
for any $k \in \bZ_{\geq 0}$.

\item The above proof also proves part iii).

\item It is easy to show (see for instance Lemma 6.21 of \cite{Kash03}) that, for $k \in\bZ$,
\[\scO_{X,\hspace{0.7pt}\fx}[f_{\fx}^{-1}]f_{\fx}^{-\alpha} = \scD_{X,\hspace{1pt}\fx}\cdot f_{\fx}^{-\alpha +k} \,\,\Leftrightarrow\,\, b_{f,\hspace{0.7pt}\fx}(l-\alpha)\neq 0 \,\,\,\text{for all }\, l \in\bZ_{<k}.\]
Therefore we see that $\scM(f^{-\alpha})_{\fx}=\scO_{X,-\alpha f,\hspace{0.7pt}\fx} = \scD_{X,\hspace{1pt}\fx}\cdot f_{\fx}^{\,\overline{\alpha}}$ if and only if $b_{f,\hspace{0.7pt}\fx}(l-\alpha) \neq 0$ for all $l \in \bZ_{< \lceil\alpha\rceil}$. Since as stated above we know that all of the roots of $b_{f,\hspace{0.7pt}\fx}(s)$ are negative, this is equivalent to $b_{f,\hspace{0.7pt}\fx}(l-\alpha)\neq 0$ for all $l \in\bZ$.
    
\end{enumerate}

\end{proof}

\begin{defn}

Recalling also the complex pure Hodge module $\underline{\bC}^{\beta}[1]\in\text{MH}(\bC^*,\bC)$ of weight 1 defined in the previous subsection, we define 
\[\underline{\bC}_{U,-\alpha f}[n] := \widetilde{f}^*\underline{\bC}^{-\alpha}[1] \in \text{MH}(U,\bC).\]
Note that this also only depends upon the fractional part of $\alpha$.

$\underline{\bC}_{\,U,-\alpha f}[n]$ is a complex pure Hodge module of weight $n$, and is induced by a variation of $\bC$-Hodge structure. Thus we may apply Proposition \ref{propextension} to make our most important definitions of this paper.
    
\end{defn}

\begin{defn}

\[M(f^{-\alpha}):=j_*\underline{\bC}_{U,-\alpha f}[n] \in \text{MHM}(X,\bC), \,\,\,\,\,\,\, \underline{\bC}_{X,-\alpha f}[n]:=j_{!*}\underline{\bC}_{U,-\alpha f}[n]\in\text{MH}(X,\bC).\]

\label{defnMHMs}
    
\end{defn}

\noindent By Lemma \ref{lemDmodprops}.i) and Proposition \ref{proplocalisable} the left $\scD_X$-module underlying $M(f^{-\alpha})$ is $\scM(f^{-\alpha})$. By Lemma \ref{lemDmodprops}.ii), the underlying left $\scD_X$-module of $\underline{\bC}_{X,-\alpha f}[n]$ is $\scO_{X,-\alpha f}$.

Also, Remark \ref{remWmin} implies that the smallest non-zero weight step of $M(f^{-\alpha})$ has index $n$, and that
\[W_nM(f^{-\alpha}) = \underline{\bC}_{X,-\alpha f}[n].\]
Therefore we see that the underlying left $\scD_X$-module of $W_nM(f^{-\alpha})$ is $\scO_{X,-\alpha f}$. In particular, by Lemma \ref{lemDmodprops}.iv), $M(f^{-\alpha})$ is pure as a complex Hodge module if and only if $b_{f,\hspace{0.7pt}\fx}(-\alpha+k)\neq 0$ for all $k\in\bZ$ and for all $\fx\in X$ with $f(\fx)=0$.

The following is analogous to Lemma 2.6 of \cite{MP19b}.

\begin{lem}

Take $l \in \bZ_{\geq 1}$ such that $\alpha l\in\bZ$. Consider the map
\[p : \widetilde{U}:=\{(x,z) \in U\times\bC\,|\, f(x)^{-l\alpha}=z^l\} \to U ; (x,z) \mapsto x.\]
$p$ is a holomorphic map between complex manifolds, and is clearly a local biholomorphism, and projective (since $\widetilde{U}$ is closed in $U\times\bP^1_{\bC}$).

Then we have a canonical isomorphism of left $\scD_X$-modules
\[p_+\scO_{\widetilde{U}} \simeq \bigoplus_{i=0}^{l-1}\scO_{U,-i\alpha f},\]
inducing isomorphisms of complex mixed Hodge modules
\[p_*\underline{\bC}_{\widetilde{U}}[n] \simeq \bigoplus_{i=0}^{l-1}\underline{\bC}_{U,-i\alpha f}[n]\,\,\,\text{ and }\,\,\, j_*p_*\underline{\bC}_{\widetilde{U}}[n] \simeq \bigoplus_{i=0}^{l-1}M(f^{-i\alpha}).\]

\label{lemsummand}
    
\end{lem}

\begin{proof}

Since $p$ is a local isomorphism, $p^*\scD_U \simeq \scD_{\widetilde{U}}$, and thus $p_+\scO_{\widetilde{U}}\simeq p_*\scO_{\widetilde{U}}$, where $p_*$ denotes sheaf push-forward.

Now, an arbitrary holomorphic function on $\widetilde{U}$ may be written $\sum_{i=0}^{l-1}g_iz^i$, where $g_i \in \scO_U$. Thus we may define an $\scO_U$-linear map
\[\tau : p_*\scO_{\widetilde{U}} \to \bigoplus_{i=0}^{l-1}\scO_{U,-i\alpha f}\,\,\, ;\,\,\, \sum_{i=0}^{l-1}g_iz^i \mapsto (g_if^{-i\alpha})_{i=0}^{l-1}.\]
It is clear that this is an isomorphism of $\scO_U$-modules. 

Now, given any derivation $\delta \in\text{Der}_{\bC}(\scO_{\widetilde{U}})$, 
\[0=\delta(f^{-l\alpha}-z^l)=-l\alpha \delta(f)f^{-l\alpha-1}-l\delta(z)z^{l-1} \,\,\,\Rightarrow\,\,\, f\delta(z) = -\alpha \delta(f)z,\]
implying that
\[\delta(g_iz^i) = \delta(g_i)z^i-i\alpha g_i\frac{\delta(f)}{f}z^i\]
for any $g_i \in\scO_U$.

This thus proves that $\tau$ is a $\scD_U$-linear map; if $\delta \in\text{Der}_{\bC}(\scO_U)$ and $g_i \in \scO_U$, then 
\[\delta(g_if^{-i\alpha}) = \delta(g_i)f^{-i\alpha} -i\alpha g_i \frac{\delta(f)}{f}f^{-i\alpha}.\]

Finally, the first isomorphism of complex Hodge modules follows because all of the Hodge modules involved on both sides have trivial Hodge and weight filtrations, with matching indexing. The second isomorphism just follows by applying the (well-defined) functor $j_*$.
    
\end{proof}

\vspace{10pt}

\noindent Finally, we also globalise the above construction, using twisted mixed Hodge modules. 

Assume that we are given a complex manifold $X$ and an effective divisor $D$ on $X$, and a rational number $\alpha \in \bQ$. Write $Z$ for the support of $D$ and $U$ for the complement of $Z$ in $X$, with $j:U\hookrightarrow X$ for the natural open embedding.

Write $\scL=\scO_X(D)$. Assume that we have a global section $s\in\Gamma(X,\scL)$ such that $D=\text{div}(s)$. 

\begin{defn}

Given a local trivialisation of the line bundle $\scL$ on an open subset $V\subseteq X$, we obtain some $f\in\scO_V$ such that $f=s|_V$ and thus $D|_V = \text{div}(f)$. We then consider the complex mixed Hodge module $M(f^{-\alpha})$ defined above. It is a simple exercise to check that the underlying $\scD_V$-module does indeed change according to Equation \eqref{eqtrivtran} under change of trivialisation. Thus the definition of twisted mixed Hodge module gives us a $-\alpha\scL$-twisted mixed Hodge module on $X$, which we will write as $M_{X,-\alpha \scL}$, which restricts to $M(f^{-\alpha})$ on a trivialisation such as the one above.

\label{defntwMHMs}
    
\end{defn}

\begin{rem}

\hfill

\begin{enumerate}[label=\roman*)]

\item Since $j^*\scL \simeq \scO_U$, we have functors 
\[j_*, j_! : \text{MHM}_{\text{ex}}(U,\bC) \to \text{MHM}(X,-\alpha\scL),\]
where $\text{MHM}_{\text{ex}}(U,\bC)$ is the full subcategory of complex mixed Hodge modules on $U$ for which $j_*$ and $j_!$ (as functors into $\text{MHM}(X,\bC)$) are well-defined.

We use the same notation as in the untwisted case but hope that context will help differentiate.

Then $M_{X,-\alpha\scL} = j_*\underline{\bC}_{U}[n]$.

\item Consider the left $\scD_L$-module $s_+\scO_X$. We write
\[V^j_{\text{ind}}s_+\scO_X := V^j\scD_L \cdot 1 \subseteq s_+\scO_X.\]
Let $k\in\bZ$ be such that $b_{f,\hspace{0.7pt}\fx}(l-\alpha)\neq 0$ for all $l \in\bZ_{<k}$ (so that $\scD_{X,\hspace{1pt}\fx}\cdot f_{\fx}^{k-\alpha} = \scM(f^{-\alpha})_{\fx}$), for all local trivialisations and points $\fx$ (i.e. ``take $k$ sufficiently small"). The $-\alpha\scL$-twisted $\scD$-module on $X$ underlying the $-\alpha\scL$-twisted mixed Hodge module $M_{X,-\alpha \scL}$ is then given by
\[\scM_{X,-\alpha\scL}:=\scL^k\otimes_{\scO_X}\frac{\text{gr}_{V_{\text{ind}}}^0s_+\scO_X}{(\theta+\alpha-k)\cdot\text{gr}_{V_{\text{ind}}}^0s_+\scO_X}.\]
(This is because the above is locally isomorphic to $\frac{\scD_X[s]\cdot f^s}{(s+\alpha-k)\scD_X[s]\cdot f^s} \simeq \scD_X\cdot f^{k-\alpha} = \scM(f^{-\alpha})$.)

\end{enumerate}
    
\end{rem}

\vspace{10pt}

\subsection{Hodge and weight filtrations via log resolutions}

\; \vspace{5pt} \\ We will now use log resolutions to obtain an expression for the mixed Hodge module structure on $\scM(f^{-\alpha})$, mirroring the strategy from \cite{MP19a}, \cite{MP19b}, \cite{Olano23}, \cite{MP22} etc. The first step in this direction is to understand what happens in the case that our divisor is simple normal crossing. It will turn out to be very necessary to understand what happens in the non-reduced setting; when we return to the general situation later, pulling our divisor back along the chosen log resolution will give us a non-reduced simple normal crossing divisor, even if we start with a reduced divisor.

Let $Y$ be a complex manifold of dimension $n$ and let $E$ be a simple normal crossing divisor on $Y$, so that we may write
\[E=\sum_{i=1}^ma_iE_i\]
for some smooth hypersurfaces $E_i$ and $a_i \in \bZ_{\geq 0}$. Again, let $\alpha \in \bQ$. Write
\[I:=\{i \in [m] \; |\; a_i\neq 0\} \;\text{ and }\; I_{\alpha}:=\{i \in I\; |\; \alpha a_i\in\bZ\}.\]
Assume that $E$ is defined by some global function $g \in\cO_Y$. Throughout, we interchange between working locally and globally. Whenever we work locally at a point $\fy \in Y$, we choose local coordinates $y_1,\ldots,y_n$ on $Y$ centred at $\fy$ such that, without loss of generality (via reordering), $E_i$ is locally defined by the function $y_i$, for $i=1,\ldots,n$, and such that, if we write
\[I_{\fy}:=\{i \in I \; |\; \fy\in E_i\} \subseteq [n] \;\;\text{ and }\;\; I_{\alpha,\fy}:=\{i \in I_{\alpha}\; |\; \fy\in E_i\}\subseteq I_{\fy},\]
then $g$ is given locally at $\fy$ by 
\[g = \prod_{i\in I_{\fy}}y_i^{a_i}.\]

\begin{thm}


Still working locally, we write $m_{\alpha,\fy}:=|I_{\alpha,\fy}|$.

\begin{enumerate}[label = \roman*)]

\item Then
\[W_{n+m_{\alpha,\fy}}\scM(g^{-\alpha})_{\fy} = \scM(g^{-\alpha})_{\fy}.\footnotemark\]

\footnotetext{So in particular, $W_{n+m_{\alpha}}\scM(g^{-\alpha}) = \scM(g^{-\alpha})$, where $m_{\alpha}:=|I_{\alpha}|$. Moreover, $W_{2n}\scM(g^{-\alpha}) = \scM(g^{-\alpha})$.}

\item 
Let $l\in\bZ$ such that $0 \leq l \leq m_{\alpha,\fy}$. Then
\[F_0^HW_{n+l}\scM(g^{-\alpha})_{\fy} = \sum_{J\subseteq I_{\alpha,\fy},\, |J|=l}\scO_{Y,\fy}\prod_{i \in I_{\alpha,\fy}\backslash J}y_i^{\lceil \alpha a_i\rceil}\prod_{i \in (I_{\fy}\backslash I_{\alpha,\fy})\cup J}y_i^{\lceil \alpha a_i\rceil-1} g^{-\alpha}\]
and, for $k \in \bZ_{\geq 0}$, 
\[F_k^HW_{n+l}\scM(g^{-\alpha})=F_k\scD_Y\cdot F_0^HW_{n+l}\scM(g^{-\alpha}).\]
    
\end{enumerate}

\label{thmSNCFW}
    
\end{thm}

\begin{proof}

The only way to prove this result is using the relationship between Hodge and weight filtrations and $V$-filtrations. We will prove this theorem in Section 3 (Theorem \ref{thmSNCFW2}) using our main formula for Hodge and weight filtrations that we will also prove in Section 3 (Theorem \ref{thmmainformula}).
    
\end{proof}

\begin{note}

It is easy to check that this indeed agrees with the expression $W_n\scM(g^{-\alpha}) = \scD_Y\cdot g^{\,\overline{\alpha}}$ obtained earlier in the chapter, but the Hodge filtration is not the order filtration induced by the generator $g^{\,\overline{\alpha}}$, but rather the order filtration induced by the (globally defined) generator 
\[\prod_{i \in I_{\alpha,\fy}}y_i^{\lceil \alpha a_i\rceil}\prod_{i \in I_{\fy}\backslash I_{\alpha,\fy}}y_i^{\lceil \alpha a_i\rceil-1} g^{-\alpha} = \prod_{i \in I_{\fy}}y_i^{\lfloor \alpha a_i\rfloor}\,g^{-\alpha}.\]
    
\end{note}

\vspace{5pt}

\noindent Next we use this result to find filtered resolutions (by induced $\scD$-modules) for the right $\scD_X$-module $\scM_r(g^{-\alpha}):=\omega_Y\otimes_{\scO_Y}\scM(g^{-\alpha})$ and for its weight filtration steps. We write $E_{\text{red}}$ for the reduction of the divisor $E$, given by
\[E_{\text{red}}=\sum_{i\in I}E_i,\] 
and define a logarithmic connection
\[\nabla:\scO_Y(-\lceil \alpha E\rceil)\to \scO_Y(-\lceil \alpha E\rceil)\otimes_{\scO_Y}\Omega_Y^1(\log E_{\text{red}}) \; ;\; s \mapsto s\otimes \sum_{i\in I_{\fy}}\lceil \alpha a_i\rceil \frac{dy_i}{y_i}.\]
Now we define the following filtered complex $(C_{-\alpha}^{\bullet},F_{\bullet-n})$ of right $\scD_Y$-modules.

The objects of $C_{-\alpha}^{\bullet}$ and $F_{k-n}C_{-\alpha}^{\bullet}$ (for $k \in \bZ$) are given by 
\[C_{-\alpha}^r:=\scO_Y(-\lceil \alpha E\rceil)\otimes_{\scO_Y}\Omega_Y^{n+r}(\log E_{\text{red}})\otimes_{\scO_Y}\scD_Y,\]
\[F_{k-n}C_{-\alpha}^r:=\scO_Y(-\lceil \alpha E\rceil)\otimes_{\scO_Y}\Omega_Y^{n+r}(\log E_{\text{red}})\otimes_{\scO_Y}F_{r+k}\scD_Y,\]
\vspace{3pt}while the differential is a ``twist by $-\alpha \log(g)$" of the differential induced by $\nabla$, i.e.
\[D:C_{-\alpha}^r \to C_{-\alpha}^{r+1}\;\; ; \;\; s \otimes P \mapsto \nabla s\otimes P+\sum_{i=1}^ndy_i\wedge s\otimes\partial_{y_i}P-\alpha\frac{dg}{g}\wedge s\otimes P\]
in local coordinates, where $s \in \scO_Y(-\lceil \alpha E\rceil)\otimes_{\scO_Y}\Omega_Y^{n+r}(\log E_{\text{red}})$ and $P\in \scD_X$. Note that it is easy to see that this differential is indeed well-defined when restricted to $F_{k-n}C_{-\alpha}^{\bullet}$. 

\begin{prop}

The filtered complex $(C_{-\alpha}^{\bullet},F_{\bullet-n})$ is quasi-isomorphic to the filtered right $\scD_Y$-module $(\scM_r(g^{-\alpha}),F_{\bullet-n}^H)$. 

\label{propSNCresF}
    
\end{prop}

\begin{proof}

This is proven in \cite{MP19b} (Proposition 6.1). An explicit quasi-isomorphism is given by 
\[\Phi : \scO_Y(-\lceil\alpha E\rceil)\otimes_{\scO_Y}\omega_Y(E_{\text{red}})\otimes_{\scO_Y}\scD_Y \to \scM_r(g^{-\alpha}),\]
given in local coordinates by
\[h\omega \otimes P \mapsto (g^{-\alpha}h\omega)\cdot P,\]
where $h=\prod_{i\in I_{\fy}}y_i^{\lceil\alpha a_i\rceil}$, $\omega \in\omega_Y(E_{\text{red}})$ and $P\in\scD_Y$.
    
\end{proof}

\begin{rem}

It is important to note that if $s \in \scO_Y(-\lceil \alpha E\rceil)$ then its image under this twisted differential map has poles only along the $E_i$ for which $i \notin I_{\alpha}$. Indeed, it is easy to see, for $P\in \scD_Y$, that
\[D(s\otimes P) = s\otimes \sum_{i \in I_{\fy}}(\lceil \alpha a_i\rceil-\alpha a_i)\frac{dy_i}{y_i}\otimes P+ s\otimes \sum_{i=1}^n dy_i\otimes \partial_{y_i}P.\]
    
\end{rem}

For $l \in \bZ$ with $0 \leq l \leq m_{\alpha}$, the above remark now allows us to define filtered subcomplexes $(W_{n+l}C_{-\alpha}^{\bullet},F_{\bullet-n})$ of $(C_{-\alpha}^{\bullet},F_{\bullet-n})$, whose objects are defined by
\[W_{n+l}C_{-\alpha}^r:=\scO_Y(-\lceil \alpha E\rceil)\otimes_{\scO_Y}W_l^{\alpha}\Omega_Y^{n+r}(\log E_{\text{red}})\otimes_{\scO_Y}\scD_Y,\]
\[F_{k-n}W_{n+l}C_{-\alpha}^r:=\scO_Y(-\lceil \alpha E\rceil)\otimes_{\scO_Y}W_l^{\alpha}\Omega_Y^{n+r}(\log E_{\text{red}})\otimes_{\scO_Y}F_{r+k}\scD_Y.\]
\vspace{3pt}Here, $W_l^{\alpha}\Omega_Y^{n+r}(\log E_{\text{red}})$ is defined as the $\scO_X$-submodule of $\Omega_Y^{n+r}(\log E_{\text{red}})$ generated locally by elements of the form
\[\bigwedge_{i \in J_1}dy_i\wedge\bigwedge_{i \in J_2}\frac{dy_i}{y_i},\]
where $J_1 \subseteq I_{\alpha,\fy}\cup ([n]\backslash I_{\fy})$ and $J_2 \subseteq I_{\fy}$ are disjoint such that $J_2\cap I_{\alpha,\fy}$ has size $l$ and such that $J_1\cup J_2$ has size $n+r$. (i.e. we filter the logarithmic forms along $E_{\text{red}}$ by the order of pole along $\sum_{i\in I_{\alpha}}E_i$.)

Note that if $I_{\alpha}=I$ then this weight filtration agrees with the standard weight filtration on $\Omega_Y^{n+r}(\log E_{\text{red}})$ as defined by Deligne, \cite{Del71}, (3.1.5).

The above remark ensures that the differential maps are well-defined when restricted to these objects. 

\vspace{5pt}

\noindent As an extension of Proposition 4.1 of \cite{Olano23} (to the twisted setting) we have:

\begin{prop}

The filtered complex $(W_{n+l}C_{-\alpha}^{\bullet},F_{\bullet-n})$ is quasi-isomorphic to the filtered right $\scD_Y$-module $(W_{n+l}\scM_r(g^{-\alpha}),F_{\bullet-n}^H)$. 

\label{propSNCresW}
    
\end{prop}

\begin{proof}

By Proposition \ref{propSNCresF}, it suffices to show that the complex $W_{n+l}C_{-\alpha}^{\bullet}$ is quasi-isomorphic to $W_{n+l}\scM_r(g^{-\alpha})$. 

First we show that this complex is exact away from zero. Take $r \in \{-n,\ldots,-1\}$ and $u \in W_{n+l}C_{-\alpha}^r=\scO_Y(-\lceil \alpha E\rceil)\otimes_{\scO_Y}W_l^{\alpha}\Omega_Y^{n+r}(\log E_{\text{red}})\otimes_{\scO_Y}\scD_Y$, and assume that $D(u) = 0$. Then, since $H^r(C_{-\alpha}^{\bullet})=0$, there exists some $v \in \scO_Y(-\lceil \alpha E\rceil)\otimes_{\scO_Y}\Omega_Y^{n+r-1}(\log E_{\text{red}})\otimes_{\scO_Y}\scD_Y$ such that $D(v)=u$.

In local coordinates we may write $v = h\omega \otimes P$ where $h = \prod_{i \in I_{\fy}}y_i^{\lceil \alpha a_i\rceil}$, $\omega\in\Omega_Y^{n+r-1}(\log E_{\text{red}})$ and $P\in\scD_Y$. Then
\begin{align*}u =D(v) &=h\left(\sum_{i\in I_{\fy}}\lceil\alpha a_i\rceil\frac{dy_i}{y_i}\wedge\omega\otimes P + \sum_{i \in[n]}dy_i\wedge\omega\otimes\partial_{y_i}P -\alpha\frac{dg}{g}\wedge\omega\otimes P\right)\\
&=h\left(\sum_{i\in I_{\fy}\backslash I_{\alpha,\fy}}\left(\lceil\alpha a_i\rceil-\alpha a_i\right)\frac{dy_i}{y_i}\wedge\omega\otimes P + \sum_{i \in[n]}dy_i\wedge\omega\otimes\partial_{y_i}P\right).\end{align*}

In particular, $dy_i\wedge\omega \in W_l^{\alpha}\Omega_Y^{n+r}(\log E_{\text{red}})$ for all $i$, implying that $\omega \in W_l^{\alpha}\Omega_Y^{n+r-1}(\log E_{\text{red}})$ as required.

Next we consider the restriction of the map $\Phi$, 
\[\Phi_{n+l} : \scO_Y(-\lceil\alpha E\rceil)\otimes_{\scO_Y}W_l^{\alpha}\omega_Y(E_{\text{red}})\otimes_{\scO_Y}\scD_Y \to W_{n+l}\scM_r(g^{-\alpha}),\]
which is well-defined and surjective by the expression for $W_{n+l}\scM_r(g^{-\alpha})$ given in Theorem \ref{thmSNCFW}. Given that the kernel of $\Phi$ equals the image of 
\[\scO_Y(-\lceil\alpha E\rceil)\otimes_{\scO_Y}\Omega_Y^{n-1}(\log E_{\text{red}})\otimes_{\scO_Y}\scD_Y \to \scO_Y(-\lceil\alpha E\rceil)\otimes_{\scO_Y}\omega_Y(E_{\text{red}})\otimes_{\scO_Y}\scD_Y,\]
we may show identically to above that the kernel of $\Phi_{n+l}$ equals the image of 
\[\scO_Y(-\lceil\alpha E\rceil)\otimes_{\scO_Y}W_l^{\alpha}\Omega_Y^{n-1}(\log E_{\text{red}})\otimes_{\scO_Y}\scD_Y \to \scO_Y(-\lceil\alpha E\rceil)\otimes_{\scO_Y}W_l^{\alpha}\omega_Y(E_{\text{red}})\otimes_{\scO_Y}\scD_Y,\]
finishing the proof.

\end{proof}

\begin{cor}

There is a quasi-isomorphism of filtered complexes
\[\text{\emph{DR}}_Y(W_{n+l}M(g^{-\alpha})) \simeq (\scO_Y(-\lceil\alpha E\rceil)\otimes_{\scO_Y}W^{\alpha}_l\Omega_Y^{\bullet}(\log E_{\text{\emph{red}}}),\sigma^{\geq -\bullet-n}),\]
where $\sigma^{\geq\bullet}$ is the truncation filtration on the complex (concentrated in degrees $-n,\ldots,0$) $W^{\alpha}_l\Omega_Y^{\bullet}(\log E_{\text{\emph{red}}})$.

In particular,
\[\text{\emph{gr}}^F_{i-n}\text{\emph{DR}}_Y(W_{n+l}M(g^{-\alpha})) \simeq \scO_Y(-\lceil\alpha E\rceil)\otimes_{\scO_Y}W^{\alpha}_l\Omega_Y^{n-i}(\log E_{\text{\emph{red}}})[i].\]

\label{corgrDRSNC}
    
\end{cor}

\vspace{15pt}

Now we use what we have shown in the case of simple normal crossing divisors to obtain an expression for the filtrations on $\scM(f^{-\alpha})$ in the general setting, using Hironaka's existence of log resolutions.

As before, let $X$ be a complex manifold of dimension $n$ and let $D$ be an effective divisor on $X$. Let $\alpha \in \bQ$ and write $\overline{\alpha}=-\alpha +\lceil \alpha \rceil$. Write also $\scL=\scO_X(D)$.

\begin{defn}

A \emph{strong log resolution} for the pair $(X,D)$ is a proper birational morphism $\pi : Y \to X$ such that the pullback $E=\pi^*D$ along $\pi$ of the divisor $D$ is a simple normal crossing divisor, and such that the restriction of $\pi$ to $Y\backslash \text{Supp}(\pi^*D)$ is an isomorphism.
    
\end{defn}

\noindent It is a theorem of Hironaka (\cite{H64}, Main Theorem II\footnote{The theorem and proof generalise to complex analytic spaces. See for instance \cite{W08}, Theorem 2.0.2 (where a simplified proof is also given).}) that such a resolution always exists. Moreover, we may ensure that $\pi$ is locally given by a composition of blow ups with smooth centres (so that $\pi$ is locally projective in particular). Write 
\[E=\pi^*D \;\;\text{ and }\;\; \scL_Y=\pi^*\scL = \scO_Y(E).\]
Then by Proposition \ref{proptwMHMfunctors} above we have a functor
\[\pi_*:D^{\text{b}}\text{MHM}(Y,\overline{\alpha}\scL_Y) \to D^{\text{b}}\text{MHM}(X,\overline{\alpha}\scL),\]
and thus (using functoriality of direct image)
\[\pi_*M_{Y,\overline{\alpha}\scL_Y}= \pi_*j_{Y,*}\underline{\bC}_U[n]=(\pi\circ j_Y)_*\underline{\bC}_U[n]=j_*\underline{\bC}_U[n]=M_{X,\overline{\alpha}\scL},\]
where $j_Y:U \to Y$ is the open embedding of $U$ into $Y$. Therefore we can attempt to understand the $\overline{\alpha}\scL$-twisted mixed Hodge module structure on $\scM_{X,\overline{\alpha}\scL}$ by pushing forward the $\overline{\alpha}\scL_Y$-twisted mixed Hodge module structure on $\scM_{Y,\overline{\alpha}\scL_Y}$, which we now understand well due to the above (local) computations.

We now assume that there exists a global defining equation $f$ for the divisor $D$ and write $g$ for its pullback to $Y$. 

We consider the same filtered complex of right $\scD_Y$-modules $(C_{-\alpha}^{\bullet},F_{\bullet-n})$ as above, associated to the divisor $E=\pi^*D$, and its sub-filtered complexes $(W_{n+l}C_{-\alpha}^{\bullet},F_{\bullet-n})$. 

\begin{thm}

Let $l \in \bZ$ with $0 \leq l \leq m_{\alpha}$.

\begin{enumerate}[label = \roman*)]

\item For any $p \in \bZ\backslash \{0\}$ and $k \in \bZ$, 
\[R^p\pi_*(C_{-\alpha}^{\bullet}\otimes_{\scD_Y}\scD_{Y\to X})=0\]
and
\[R^p\pi_*F_{k-n}(C_{-\alpha}^{\bullet}\otimes_{\scD_Y}\scD_{Y\to X})=0.\]
Moreover, there are canonical isomorphisms
\[R^0\pi_*(C_{-\alpha}^{\bullet}\otimes_{\scD_Y}\scD_{Y\to X}) \simeq \scM_r(f^{-\alpha})\]
and
\[R^0\pi_*F_{k-n}(C_{-\alpha}^{\bullet}\otimes_{\scD_Y}\scD_{Y\to X}) \simeq F_{k-n}^H\scM_r(f^{-\alpha}).\]

\item For any $p \in \bZ\backslash \{0\}$ and $k \in \bZ$, 
\[R^p\pi_*(W_nC_{-\alpha}^{\bullet}\otimes_{\scD_Y}\scD_{Y\to X})=0\]
and
\[R^p\pi_*F_{k-n}(W_nC_{-\alpha}^{\bullet}\otimes_{\scD_Y}\scD_{Y\to X})=0.\]
Moreover, there are canonical isomorphisms
\[R^0\pi_*(W_nC_{-\alpha}^{\bullet}\otimes_{\scD_Y}\scD_{Y\to X}) \simeq \omega_{X,-\alpha f}:=\left(\scO_{X,-\alpha f}\right)_r\]
and
\[R^0\pi_*F_{k-n}(W_nC_{-\alpha}^{\bullet}\otimes_{\scD_Y}\scD_{Y\to X}) \simeq F_{k-n}^H\omega_{X,-\alpha f}.\]

\item We have equalities
\[\hspace{5pt}W_{n+l}\scM_r(f^{-\alpha}) = \text{\emph{Im}} \left(R^0\pi_*(W_{n+l}C_{-\alpha}^{\bullet}\otimes_{\scD_Y}\scD_{Y\to X}) \to R^0\pi_*(C_{-\alpha}^{\bullet}\otimes_{\scD_Y}\scD_{Y\to X}) \isommap \scM_r(f^{-\alpha})\right)\]
and, for $k\in\bZ$,
\[F_{k-n}^HW_{n+l}\scM_r(f^{-\alpha}) = \text{\emph{Im}} \left(R^0\pi_*F_{k-n}(W_{n+l}C_{-\alpha}^{\bullet}\otimes_{\scD_Y}\scD_{Y\to X}) \to \scM_r(f^{-\alpha})\right).\]
    
\end{enumerate}

\label{thmpiplus}
    
\end{thm}

\begin{proof}

\begin{enumerate}[label=\roman*)]

\item This first part is proven by Mustață and Popa in \cite{MP19b}, and holds by Proposition \ref{propSNCresW}, since $M(f^{-\alpha})=\pi_*M(g^{-\alpha})$, where we use the inherent strictness of the direct image of mixed Hodge modules (with respect to the Hodge filtration), which essentially just says that $R^p\pi_*$ and $F_{\bullet}$ commute (see \cite{SS24}, Theorem 14.3.2 and Proposition 8.8.23).

\item By the uniqueness property of minimal extension (see Proposition \ref{propextension}) we have necessarily that $\pi_*\underline{\bC}_{Y,-\alpha g}[n] = \underline{\bC}_{X,-\alpha f}[n]$, and then the statements follow exactly as in part i).

\item The first isomorphism holds by the definition of the weight filtration of the direct image; see for instance Section 7.1.2 of \cite{M15a}. The second isomorphism holds since $R^0\pi_*W_{n+l}M(g^{-\alpha}) \to R^0\pi_*M(g^{-\alpha})\simeq M(f^{-\alpha})$, being a morphism of mixed Hodge modules, is strict with respect to the Hodge filtrations (since, for instance, the category of mixed Hodge modules is abelian; see e.g. \cite{M15a}, Lemma 7.1.2).

\end{enumerate}
    
\end{proof}

\begin{note}

We see in particular that $W_{n+m_{\alpha}}\scM(f^{-\alpha})=\scM(f^{-\alpha})$, and that $W_{2n}\scM(f^{-\alpha})=\scM(f^{-\alpha})$.

\end{note} 

As a corollary, we find that certain steps of the Hodge filtration on weight filtration steps of $\scM(f^{-\alpha})$ coincide with familiar birational invariants.

\begin{defn}

Let $X'$ be a complex manifold and $D'$ an effective $\bQ$-divisor on $X'$. Choose a log resolution $\mu : Y' \to X'$ of the pair $(X',D')$.

\begin{enumerate}[label=\roman*)]

\item (\cite{L04}, Definition 9.2.1). The \emph{multiplier ideal sheaf} associated to $D'$ is defined to be
\[\scJ(X',D'):=\mu_*\scO_{Y'}(K_{Y'/X'}-\lfloor \mu^*D'\rfloor)\subseteq \scO_{X'}.\]

\item (\cite{T13}, Definition 1.6). Assume in addition that $D'$ is a \emph{boundary divisor}, i.e. that all coefficients of $D'$ lie in the closed interval $[0,1]$. Assume moreover that the strict transform $\widetilde{\lfloor D'\rfloor}$ of $\lfloor D'\rfloor$ along $\mu$ is smooth\footnote{i.e. the strict transforms of the components of $\lfloor D'\rfloor$ don't intersect.}\footnote{Such a log resolution always exists.}. Then the \emph{adjoint ideal sheaf} associated to $D'$ is defined to be 
\[\text{adj}(X',D'):=\mu_*\scO_{Y'}(K_{Y'/X'}+\widetilde{\lfloor D'\rfloor}-\lfloor\mu^*D'\rfloor)\subseteq\scO_{X'}.\]

\end{enumerate}

Both of these objects are indeed independent of the choice of log resolution.
    
\end{defn}

\begin{cor}

Assume $\alpha\geq 0$.

\begin{enumerate}[label=\roman*)]

\item $F_{-n}^HW_{n+l}\scM_r(f^{-\alpha}) = f^{-\alpha}\pi_*(\scO_Y(-\lceil\alpha E\rceil)\otimes W_l^{\alpha}\omega_Y(E_{\text{\emph{red}}}))$. \vspace{10pt}

\item $F_0^HW_n\scM(f^{-\alpha})=\scJ(X,\alpha D)f^{-\alpha}$. \vspace{10pt}

\item $F_0^H\scM(f^{-\alpha})=\scJ(X,(\alpha-\epsilon) D)f^{-\alpha}$, for any rational $0<\epsilon<<1$.\vspace{10pt}

\item Write $\left\{\alpha D\right\}:=\alpha D-\lceil\alpha D\rceil +D_{\text{\emph{red}}}$ for the boundary divisor associated to $\alpha D$. Then
\[F_0^HW_{n+1}\scM(f^{-\alpha})=\text{\emph{adj}}(X,\{\alpha D\})\scO_X(\left\{\alpha D\right\}-\alpha D)f^{-\alpha}.\footnotemark\]

\footnotetext{In particular, if $\alpha D$ is itself a boundary divisor, then $\alpha D=\left\{\alpha D\right\}$, and $F_0^HW_{n+1}\scM(f^{-\alpha})=\text{adj}(X,\{\alpha D\})f^{-\alpha}$, as observed in \cite{Olano21} (Theorem A) in the case that $D$ is reduced and $\alpha =1$.}

\end{enumerate}

\label{corpiplus}
    
\end{cor}

\begin{proof}

\begin{enumerate}[label=\roman*)]

\item By Theorem \ref{thmpiplus}, 
\begin{align*}
F_{-n}^HW_{n+l}\scM_r(f^{-\alpha}) =& \text{Im} \left(R^0\pi_*F_{-n}(W_{n+l}C_{-\alpha}^{\bullet}\otimes_{\scD_Y}\scD_{Y\to X}) \to \scM_r(f^{-\alpha})\right) \\
=& f^{-\alpha}\pi_*(\scO_Y(-\lceil\alpha E\rceil)\otimes W_l^{\alpha}\omega_Y(E_{\text{red}})). 
\end{align*}

\item It is easy to see that we have the equality
\[W_0^{\alpha}\omega_Y(E_{\text{red}})=\omega_Y(E_{I\backslash I_{\alpha}}),\]
where $E_{I\backslash I_{\alpha}}:=\sum_{i\in I\backslash I_{\alpha}} E_i$, so that
\[\scO_Y(-\lceil\alpha E\rceil)\otimes W_0^{\alpha}\omega_Y(E_{\text{red}}) = \omega_Y(-\lfloor\alpha E\rfloor).\]
Part i) then of course implies that
\[F_0^HW_n\scM(f^{-\alpha})=\pi_*\scO_Y(K_{Y/X}-\lfloor\pi^*\alpha D\rfloor)f^{-\alpha}=\scJ(X,\alpha D)f^{-\alpha}.\] 

\item This is \cite{MP19b}, Theorem A. $W_{m_{\alpha}}^{\alpha}\omega_Y(\log E_{\text{red}})=\omega_Y(E_{\text{red}})$, so
\[F_0^H\scM(f^{-\alpha})=\pi_*\scO_Y(K_{Y/X}+E_{\text{red}}-\lceil\alpha E\rceil)f^{-\alpha} = \pi_*\scO_Y(K_{Y/X}-\lfloor\pi^*(\alpha-\epsilon)D\rfloor)f^{-\alpha}.\]

\item Note that the proof here is similar to the one given in \cite{Olano21}, Proposition 5.5. Choose a log resolution such that $\widetilde{D_{\text{red}}}$ is smooth. We work locally, proving the result at the point $\fy$, which in turn proves the result globally.

Since (locally)
\[W_1^{\alpha}\omega_Y(E_{\text{red}}) = \sum_{i\in I_{\alpha,\fy}}y_i^{-1}\omega_Y(E_{I\backslash I_{\alpha}}),\]
we have a short exact sequence
\[0 \to \omega_Y(-\lfloor\alpha E\rfloor) \to \scO_Y(-\lceil\alpha E\rceil)\otimes W_1^{\alpha}\omega_Y(E_{\text{red}}) \to \omega_{E_{\alpha}(1)} \to 0,\]
where $E_{\alpha}(1):=\coprod_{i \in I_{\alpha}}E_i$.

Now, it is a well-known fact (see for instance the proof of \cite{Olano21}, Proposition 5.5) that $\pi_*E_j=0$ for any exceptional component $E_j$. 

Thus, if no component of $\widetilde{D_{\text{red}}}$ contains $\fy$, then $\pi_*\omega_{E_{\alpha}(1)}=0$ at $\fy$. Otherwise, there exists some $i \in I_{\fy}$ such that $E_i=\widetilde{D_{\text{red}}}$ at $\fy$. If $i \notin I_{\alpha,\fy}$, we again have that $\pi_*\omega_{E_{\alpha}(1)}=0$ at $\fy$. So we see that in either of these cases that
\[F_0^HW_{n+1}\scM(f^{-\alpha})=\pi_*\scO_Y(K_{Y\backslash X}-\lfloor\pi^*\alpha D\rfloor)f^{-\alpha}\]
at the point $\fx:=\pi(\fy)\in X$. 

If instead $i\in I_{\alpha,\fy}$, consider the commutative diagram with exact rows
\begin{center}
\begin{tikzcd}
    0 \arrow[r] & \omega_Y(-\lfloor\alpha E\rfloor) \arrow[-,double line with arrow={-,-}]{d}  \arrow[r] & \scO_Y(-\lfloor\alpha E\rfloor)\otimes\omega_Y(E_i) \arrow[r] \arrow[hookrightarrow, d] & \omega_{E_i} \arrow[r] \arrow[d, hookrightarrow] & 0\\
    0 \arrow[r] & \omega_Y(-\lfloor\alpha E\rfloor) \arrow[r] & \scO_Y(-\lceil\alpha E\rceil)\otimes W_1^{\alpha}\omega_Y(E_{\text{red}}) \arrow[r] & \omega_{E_{\alpha}(1)} \arrow[r] & 0.\\
\end{tikzcd}\vspace{-15pt}
\end{center}
In this case, $\pi_*\omega_{E_{\alpha}(1)} = \pi_*\omega_{E_i}$ at the point $\fy$. Thus
\[\pi_*(\scO_Y(-\lceil\alpha E\rceil)\otimes W_1^{\alpha}\omega_Y(E_{\text{red}}))=\pi_*(\scO_Y(-\lfloor\alpha E\rfloor)\otimes\omega_Y(E_i))\]
at $\fx$, so that 
\[F_0^HW_{n+1}\scM(f^{-\alpha})=\pi_*\scO_Y(K_{Y\backslash X}+E_i-\lfloor\pi^*\alpha D\rfloor)f^{-\alpha}\]
at $\fx$. Globalising, we thus obtain the expression 
\[F_0^HW_{n+1}\scM(f^{-\alpha})=\pi_*\scO_Y(K_{Y\backslash X}+\sum_{i\in I_{\alpha}\cap K}E_i-\lfloor\pi^*\alpha D\rfloor)f^{-\alpha},\]
where $K:=\{i\in I\mid E_i \text{ not exceptional}\}$. 

But it is easy to see that $\sum_{i\in I_{\alpha}\cap K}E_i = \widetilde{\lfloor\left\{\alpha D\right\}\rfloor}$. From this it is then easy to conclude that 
\[F_0^HW_{n+1}\scM(f^{-\alpha})=\text{adj}(X,\{\alpha D\})\scO_X(\left\{\alpha D\right\}-\alpha D)f^{-\alpha}\]
as required.

\end{enumerate}
    
\end{proof}

\begin{cor} 

Assume $\alpha\geq 0$.

\begin{enumerate}[label=\roman*)]

\item $F_0^H\scM(f^{-\alpha})=\scO_Xf^{-\alpha}$ if and only if the pair $(X,\alpha D)$ is log canonical.

\item $F_0^HW_{n+1}\scM(f^{-\alpha})=\scO_Xf^{-\alpha}$ if and only if the pair $(X,\alpha D)$ is purely log terminal. 

\item $F_0^HW_n\scM(f^{-\alpha})=\scO_Xf^{-\alpha}$ if and only if the pair $(X,\alpha D)$ is Kawamata log terminal. 

\end{enumerate}

\label{corpair}

\end{cor}

\begin{note}

This follows because in particular $(X,\alpha D)$ is log canonical $\Rightarrow$ $\alpha D$ is a boundary divisor. See \cite{KM98}, 2.34, as well as for instance \cite{ST07}, Definition 2.15-2.16 and Remark 2.17 for definitions and further details about the singularity theory of pairs.

\end{note}
    
\noindent Finally, we also find a necessary and sufficient criterion for determining the generating level of $(\scO_{X,-\alpha f}, F_{\bullet}^H)$, analogous to \cite{MP19b}, Theorem 10.1.

\begin{defn}

Let $Z$ be an arbitrary complex manifold, and $(\scM,F_{\bullet})$ a (integrally-indexed) filtered $(\scD_Z,F_{\bullet}^{\text{ord}})$-module. The \emph{generating level} of $(\scM,F_{\bullet})$ is the smallest integer\footnote{(Taken to be $\infty$ if no such integer exists.)} $k \in\bZ$ such that
\[F_i\scD_Z\cdot F_k\scM =F_{k+i}\scM \text{ for all }i \in\bZ_{\geq 0}.\]
    
\end{defn}

\begin{cor}

The generating level of $(\scO_{X,-\alpha f}, F_{\bullet}^H)$ is $\leq k$ if and only if 
\[R^i\pi_*\left(\scO_Y(-\lceil\alpha E\rceil)\otimes\Omega_Y^{n-i}(\log E_{I\backslash I_{\alpha}})\right)=0 \,\,\,\text{ for all }\,\, i >k.\]
In particular, the generating level is always $\leq n-1$.

\label{corgenlevWn}
    
\end{cor}

\begin{proof}

It is easy to see (see for instance \cite{MP22}, Lemma 10.1) that the generating level of $(\scO_{X,-\alpha f}, F_{\bullet}^H)$ is $\leq k$ if and only if
\[\scH^0\text{gr}_{i-n}^F\text{DR}_X(\scO_{X,-\alpha f}, F_{\bullet}^H)=0 \,\,\, \text{ for all }\,\, i >k.\]
But 
\begin{align*}
\scH^0\text{gr}_{i-n}^F\text{DR}_XW_nM(f^{-\alpha}) &=\scH^0\text{gr}_{i-n}^F\text{DR}_X\pi_*W_nM(g^{-\alpha})\\
&=\scH^0\bR\pi_*\text{gr}_{i-n}^F\text{DR}_YW_nM(g^{-\alpha})\\
&=\scH^0\bR\pi_*(\scO_Y(-\lceil\alpha E\rceil)\otimes_{\scO_Y}\Omega_Y^{n-i}(\log E_{I\backslash I_{\alpha}})[i])\\
&=R^i\pi_*\left(\scO_Y(-\lceil\alpha E\rceil)\otimes\Omega_Y^{n-i}(\log E_{I\backslash I_{\alpha}})\right).\\
\end{align*}

\vspace{-20pt}
    
\end{proof}

\begin{rem}

In general, the $E_2$-degeneration of the weight spectral sequence for mixed Hodge modules implies that $\text{gr}^W_{n+l}M(f^{-\alpha})$ is isomorphic to the cohomology of the complex
\[H^{-1}\pi_*\text{gr}_{n+l+1}^WM(g^{-\alpha}) \to H^0\pi_*\text{gr}_{n+l}^WM(g^{-\alpha}) \to H^1\pi_*\text{gr}_{n+l-1}^WM(g^{-\alpha}).\]
It is not clear to us whether it is possible to use this to obtain some generating level bound for $(\text{gr}_{n+l}^W\scM(f^{-\alpha}),F_{\bullet}^H)$ for general $l$.
    
\end{rem}
    
\vspace{10pt}

\vspace{10pt}

\newpage

\section{Hodge filtrations via canonical $V$-filtrations}\label{sectionVfilt}

\noindent We have now obtained expressions for the Hodge and weight filtrations on $\scM(f^{-\alpha})$ through the use of log resolution. In this section we obtain purely algebraic formulae for these filtrations, using results of Morihiko Saito concerning specialisability of $\scD$-modules. In most cases this expression is somewhat easier to calculate than the expression in terms of log resolutions, as we will display with multiple examples throughout the rest of the paper. (See Section \ref{subsectionegs} and Section \ref{sectionPPD}.)

\vspace{5pt}

\subsection{Specialisability and quasi-unipotency}

\; \vspace{5pt} \\ We recall for the convenience of the reader the notions of specialisability, quasi-unipotency, and regularity (of filtered $\scD$-modules), along a fixed hypersurface. Very important to our study of the filtered $\scD$-modules $(\scM(f^{-\alpha}),F_{\bullet}^H)$ is the fact that they satisfy all of these properties along the hypersurface $f^{-1}(0)$. $V$-filtrations were first introduced by Malgrange and Kashiwara (\cite{Mal83}, Lemme 3.3 and Théorème 3.4 and \cite{Kash83}, Theorem 1). The exposition here is based on Saito's treatment (\cite{MSai88}, Section 3.1-3.2). See also \cite{Sch14}, Section 8-9, \cite{PS08}, Section 14.2 and \cite{CDM24}, as well as \cite{SS24}, Chapter 10.

\begin{defn}[\cite{MSai88}, Définition 3.1.1 and Lemme 3.1.2]

Let $X$ be a complex manifold and $Y\subseteq X$ a complex submanifold of codimension 1, the zero locus of some holomorphic function $t:X\to\bC$ (with no critical values) say. Let $\scM$ be a regular holonomic left $\scD_X$-module.

Then $\scM$ is $\bQ$-\emph{specialisable} along the hypersurface $Y$ if there exists a rational, decreasing, exhaustive, discrete, left-continuous filtration, denoted $V^{\bullet}\scM$ (which we shorten to $V^{\bullet}$ below for ease of notation), satisfying 
\begin{enumerate}[label=\roman*)]
\item For each $\gamma\in\bQ$, $V^{\gamma}$ is a coherent module over the subring \[V^0\scD_X=\{P\in\scD_X\;|\;P\cdot t^i\in(t^i) \;\forall\; i\geq 0\} \subseteq \scD_X,\]
\item For each $\gamma\in\bQ$, $t\cdot V^{\gamma}\subseteq V^{\gamma+1}$, with equality if $\gamma>0$,
\item For each $\gamma\in\bQ$, $\partial_t\cdot V^{\gamma}\subseteq V^{\gamma-1}$,
\item For each $\gamma\in\bQ$, $\partial_tt-\gamma$ acts nilpotently on $\text{gr}_V^{\gamma}:=V^{\gamma}/ V^{>\gamma}$.
\end{enumerate}
If this filtration exists, it is unique, and is called the \emph{Kashiwara-Malgrange $V$-filtration} for $\scM$ along $Y$.

\label{defnspecial}
    
\end{defn}

\begin{defn}

Let $X$ be a complex manifold and $f:X\to\bC$ holomorphic, such that the zero locus $f^{-1}(0)$ is a hypersurface. Let $\scM$ be a regular holomorphic left $\scD_X$-module.

Then $\scM$ is $\bQ$-\emph{specialisable} along the hypersurface $f^{-1}(0)$ if the regular holonomic left $\scD_{X\times\bC}$-module $i_{f,+}\scM$ is $\bQ$-specialisable along $X\times\{0\}$, where $i_f$ is the graph embedding 
\[i_f:X\to X\times\bC\,;\,\fx \mapsto (\fx,f(\fx)).\]

\label{defnspecial2}
    
\end{defn}

\begin{rem}

Note that one can show that if $f$ is smooth, then $\scM$ is $\bQ$-specialisable along $f^{-1}(0)$ in the sense of Definition \ref{defnspecial} if and only if it is $\bQ$-specialisable along $f^{-1}(0)$ in the sense of Definition \ref{defnspecial2}, so that the two definitions are indeed consistent with one another. 
    
\end{rem}

\vspace{5pt}

\noindent The significance of specialisability is that it allows us to understand the $\scD$-modules corresponding to nearby and vanishing cycle functors under the Riemann-Hilbert correspondence. In fact, any regular holonomic $\scD$-module that has quasi-unipotent monodromy along the hypersurface in question under the de Rham functor is $\bQ$-specialisable along that hypersurface.

\begin{defn}

Let $X$ be a complex manifold and $f:X\to\bC$ a holomorphic function, such that the zero locus is a hypersurface. Let $\scM$ be a holonomic left $\scD_X$-module.

Then $\scM$ has \emph{quasi-unipotent monodromy} (or is \emph{quasi-unipotent}) along the hypersurface $f^{-1}(0)$ if the monodromy action on the nearby cycles $\psi_f\text{DR}_X\scM$ of the perverse sheaf $\text{DR}_X\scM$ is quasi-unipotent.

\label{defnquni}
    
\end{defn}

\begin{thm}[\cite{Kash83}, Theorem 1. See also \cite{PS08}, Theorem 14.23]

Let $X$ be a complex manifold and $f:X\to\bC$ a holomorphic function, such that the zero locus is a hypersurface. Let $\scM$ be a regular holonomic left $\scD_X$-module.

If $\scM$ has quasi-unipotent monodromy along $f^{-1}(0)$, then $\scM$ is $\bQ$-specialisable along $f^{-1}(0)$.

\label{thmKM}
    
\end{thm}

We now make the following definition, motivated by the Riemann-Hilbert correspondence.

\begin{defn}

Let $X$ be a complex manifold and $f:X\to\bC$ a holomorphic function,  such that the zero locus is a hypersurface. Let $\scM$ be a regular holonomic left $\scD_X$-module, $\bQ$-specialisable along $f^{-1}(0)$. Write 
\[\scM_f:=i_{f,+}\scM.\]
Now we define, for each $\gamma \in \bQ$,
\[\psi_{f,e(\gamma)}\scM:=\text{gr}_V^{\gamma-\lceil\gamma\rceil+1}\scM_f\,\,\,\text{ and } \,\,\,\phi_{f,e(\gamma)}\scM:=\text{gr}_V^{\gamma-\lfloor\gamma\rfloor}\scM_f,\]
where $e(\gamma):=e^{2\pi i \gamma}\in\bC$. These are then regular holonomic $\scD_X$-modules (see \cite{Kash83}, Theorem 2).

\label{defnphipsi}
    
\end{defn}

\begin{thm}[\cite{MSai88}, Proposition 3.4.12]

Let $X$ be a complex manifold and $f:X\to\bC$ a holomorphic function such that the zero locus $Z:=f^{-1}(0)$ is a hypersurface. Let $\scM$ be a regular holonomic left $\scD_X$-module, $\bQ$-specialisable along $Z$. Then, for $\gamma\in\bQ$ we have canonical isomorphisms
\[\text{\emph{DR}}_X\psi_{f,e(\gamma)}\scM \simeq \psi_{f,e(\gamma)}\text{\emph{DR}}_X\scM\,\,\,\text{ and }\,\,\,\text{\emph{DR}}_X\psi_{f,e(\gamma)}\scM \simeq \psi_{f,e(\gamma)}\text{\emph{DR}}_X\scM.\]
Moreover, under the de Rham functor and the above isomorphisms, 
\begin{enumerate}[label=\roman*)]

\item For each $\gamma\in\,[0,1]\,\cap\bQ$, $\partial_tt-\gamma: \text{\emph{gr}}_V^{\gamma}\scM_f\to\text{\emph{gr}}_V^{\gamma}\scM_f$ is mapped to $N=\frac{1}{2\pi i}\log T_u$, where $T_u$ is the unipotent component of the monodromy operator $T$ (so that $T=T_sT_u$).

\item $t:\text{\emph{gr}}_V^0\scM_f\to\text{\emph{gr}}_V^1\scM_f$ is mapped to $\text{\emph{var}}:\phi_{f,1}\text{\emph{DR}}_X\scM_f\to\psi_{f,1}\text{\emph{DR}}_X\scM_f$.

\item $\partial_t:\text{\emph{gr}}_V^1\scM_f\to\text{\emph{gr}}_V^0\scM_f$ is mapped to $\text{\emph{can}}:\psi_{f,1}\text{\emph{DR}}_X\scM_f\to\phi_{f,1}\text{\emph{DR}}_X\scM_f$.
    
\end{enumerate}

\label{thmphipsiRH}

\end{thm}

\noindent Since the operator $\partial_tt-\gamma$ acts nilpotently on $\text{gr}_V^{\gamma}$, we obtain, for each $r\in\bZ$, a unique increasing filtration $M_{\bullet}$ of coherent left $\scD_X$-modules on $\text{gr}_V^{\gamma}\scM_f$, the \emph{monodromy weight filtration centred at $r$}, satisfying that
\begin{enumerate}[label=\roman*)]
    \item For all $l$, $(\partial_tt-\gamma)\cdot M_l\text{gr}_V^{\gamma}\scM_f \subseteq M_{l-2}\text{gr}_V^{\gamma}\scM_f,$
    \item For all $l \geq 0$, $(\partial_tt-\gamma)^l\cdot -: \text{gr}^M_{r+l}\text{gr}_V^{\gamma}\scM_f \isommap \text{gr}^M_{r-l}\text{gr}_V^{\gamma}\scM_f.$
\end{enumerate}

\noindent Note that there is an explicit expression for the steps of this filtration: 
\[M_{r+l}\text{gr}_V^{\gamma}\scM_f = \sum_{m\geq 0}(\partial_tt-\gamma)^m\cdot (\ker(\partial_tt-\gamma)^{2m+l+1}).\]
(See for instance the statement and proof of \cite{Del80}, Proposition 1.6.1.)

Since $\partial_tt-\gamma$ maps to $N$ under the de Rham functor, as mentioned above, and by uniqueness, the monodromy weight filtration centred at $r$ on $\text{gr}_V^{\gamma}\scM_f$ is mapped to the monodromy weight filtration centred at $r$ on the nearby/vanishing cycles.

\begin{defn}

We also define the \emph{kernel filtration} $K_{\bullet}$ on $\text{gr}_V^{\gamma}\scM_f$ by
\[K_l\text{gr}_V^{\gamma}\scM_f := \{ u \in \text{gr}_V^{\gamma}\scM_f\,|\, (\partial_tt-\gamma)^l\cdot u =0\}.\]
    
\end{defn}

\vspace{5pt}

Since we are interested in filtered $\scD$-modules, not just $\scD$-modules, we also define a notion of specialisability for filtered $\scD$-modules, which requires some compatibility between the $V$-filtration with the filtration on our $\scD$-module.

(\cite{Lau83}, Section 5 and \cite{MSai93}, (1.8).) Note first that if $X$ is a complex manifold and $f:X\to\bC$ is a holomorphic function, and $(\scM,F_{\bullet})$ is a filtered left $\scD_X$-module, then the pushforward of the \emph{filtered} $\scD_X$-module $(\scM,F_{\bullet})$ is given by 
\[i_{f,+}\scM = i_{f,*}\scM \otimes_{\bC}\bC[\partial_t], \,\,\,F_ki_{f,+}\scM = \sum_{i\geq 0}(i_{f,*}F_{k-1-i}\scM)\partial_t^i.\]

\begin{defn}[\cite{MSai88}, 3.2.1]

Let $X$ be a complex manifold and $f:X\to\bC$ a holomorphic function, such that the zero locus is a hypersurface. Let $(\scM,F_{\bullet})$ be a filtered regular holonomic left $\scD_X$-module. $(\scM,F_{\bullet})$ has \emph{quasi-unipotent monodromy} (or is \emph{quasi-unipotent}) along the hypersurface $f^{-1}(0)$ if

\begin{enumerate}[label=\roman*)]

\item $\scM$ has quasi-unipotent monodromy along $f^{-1}(0)$.

\item For all integers $k$ and rationals $\gamma$, $t\cdot F_kV^{\gamma}\scM_f\subseteq F_kV^{\gamma+1}\scM_f$, with equality if $\gamma>0$.

\item For all integers $k$ and rationals $\gamma$, $\partial_t\cdot F_k\text{gr}_V^{\gamma}\scM_f\subseteq F_{k+1}\text{gr}_V^{\gamma-1}\scM_f$, with equality if $\gamma<1$.

\end{enumerate}

Moreover, $(\scM,F_{\bullet})$ is \emph{regular} along $f^{-1}(0)$ if the filtration $F_{\bullet}$ on $\scM_f$ induces \emph{good} filtrations on the left $\scD_X$-modules $\text{gr}^M_l\text{gr}_V^{\gamma}\scM_f$, for all $l$ and $\gamma$, where $M_l$ is the monodromy weight filtration centred at $r$ on $\text{gr}_V^{\gamma}\scM_f$ (for any $r$).

\label{defnspecial3}
    
\end{defn}

\begin{rem}

Important to the theory of \emph{rational} mixed Hodge modules is that any filtered $\scD$-module underlying a rational mixed Hodge module is quasi-unipotent and regular along any hypersurface. Quasi-unipotency may not necessarily hold for filtered $\scD$-modules underlying complex mixed Hodge modules a priori, but conveniently does hold for our key examples $(\scM(f^{-\alpha}),F_{\bullet}^H)$. Note in fact that complex mixed Hodge modules are always $\bR$-\emph{specialisable} along hypersurfaces, and satisfy conditions ii) and iii) as well as regularity in Definition \ref{defnspecial3} above. Since the filtered $\scD$-modules underlying complex mixed Hodge modules we consider in this paper are always quasi-unipotent, we won't need this however.
    
\end{rem}

\begin{lem}

Let $X$ be a complex manifold and $f:X\to\bC$ a holomorphic function,  such that the zero locus is a hypersurface. Let $\alpha\in\bQ$. Then the (regular holonomic) filtered left $\scD_X$-module $(\scM(f^{-\alpha}),F_{\bullet}^H)$ is quasi-unipotent and regular along $f^{-1}(0)$.

\label{lemMfquni}
    
\end{lem}

\begin{proof}

This follows immediately by Lemma \ref{lemsummand} and the above remark.
    
\end{proof}

\begin{prop}

Let $X$ be a complex manifold and $f:X\to \bC$ a holomorphic function, such that the zero locus is a hypersurface. 

We then have the following properties concerning the Kashiwara-Malgrange $V$-filtrations for $\scM(f^{-\alpha})$ and $\scO_{X,-\alpha f}$.

\begin{enumerate}[label=\roman*)]

\item Under the natural inclusion $i_{f,+}\scO_{X,-\alpha f} \hookrightarrow i_{f,+}\scM(f^{-\alpha})$, we have an identification
\[V^{\gamma}i_{f,+}\scO_{X,-\alpha f} = V^{\gamma}i_{f,+}\scM(f^{-\alpha}), \,\,\text{ for all }\,\, \gamma >0.\]

\item For all $\gamma \in \bQ$,
\[t\cdot V^{\gamma}i_{f,+}\scM(f^{-\alpha}) = V^{\gamma+1}i_{f,+}\scM(f^{-\alpha}).\]

\item There is an isomorphism of $\scD_X\langle t,t^{-1},s\rangle$-modules\footnote{Here, for the $s$-linearity, $s$ acts as $s+\alpha$ on the right hand side.}
\[\Phi: i_{f,+}\scM(f^{-\alpha}) \isommap i_{f,+}\scO_X(*f)\]
given by
\[\sum_{i=0}^kg_if^{-\alpha}\partial_t^i \,\,\mapsto\,\, \sum_{i=0}^k\sum_{j=i}^kg_jf^{i-j}{j\choose i}Q_{j-i}(-\alpha)\partial_t^i,\]
where $Q_j(s)$ is the polynomial $Q_j(s)=s(s+1) \ldots (s+j-1)$ ($:=1$ if $j=0$).

Moreover, for any $\gamma\in\bQ$, 
\[\Phi(V^{\gamma}i_{f,+}\scM(f^{-\alpha})) = V^{\gamma+\alpha}i_{f,+}\scO_X(*f).\]

\end{enumerate}

\label{propVcomp}

\end{prop}

\begin{proof}

The proofs of all of these statements may be found for instance in Section 2 of \cite{MP20}.
    
\end{proof}

\vspace{5pt}

\noindent As one may hope, the above-defined $\scD$-modules $\psi_{f,e(\gamma)}\scM$ and $\phi_{f,e(\gamma)}\scM$ in fact underlie (complex) mixed Hodge modules, whenever $\scM$ itself does.

\begin{thm}

Let $X$ be a complex manifold and $f:X\to\bC$ a holomorphic function,  such that the zero locus is a hypersurface. Let $M$ be a complex mixed Hodge module on $X$ and write $(\scM,F_{\bullet})$ for the underlying regular holonomic filtered $\scD_X$-module. Assume that $(\scM,F_{\bullet})$ has quasi-unipotent monodromy along the hypersurface $f^{-1}(0)$.

Then, for each $\gamma\in\bQ$, there exist complex mixed Hodge modules 
\[\psi_{f,e(\gamma)}M \in \text{\emph{MHM}}(X,\bC) \,\,\,\text{ and }\,\,\,\phi_{f,e(\gamma)}M \in \text{\emph{MHM}}(X,\bC)\]
respectively overlying the $\scD_X$-modules
\[\psi_{f,e(\gamma)}\scM\,\,\,\text{ and }\,\,\, \phi_{f,e(\gamma)}\scM,\]
with the Hodge filtrations on both induced by the Hodge filtration on $\scM_f$ (given by pushing forward the filtration on $\scM$)\footnote{More precisely, the indexing for $\phi_{f,e(\gamma)}\scM$ is shifted by $-1$, see e.g. \cite{MSai89}.}. If $M$ is pure, the weight filtrations on $\psi_{f,e(\gamma)}\scM$ and $\phi_{f,e(\gamma)}\scM$ equal the monodromy weight filtrations $M_{\bullet}$ centred at $n-1$ and $n$ respectively, as defined above. If $M$ is mixed, the weight filtrations are given by the so-called \emph{relative monodromy weight filtrations}, as defined in Section 1 of \cite{MSai90}.

Moreover, the morphisms $\text{\emph{var}}$, $\text{\emph{can}}$ and $N$ extend to morphisms of complex mixed Hodge modules
\[\text{\emph{var}}:\phi_{f,1}M\to\psi_{f,1}M(-1),\]
\[\text{\emph{can}}:\psi_{f,1}M\to\phi_{f,1}M,\]
\[N = \text{\emph{var}}\circ\text{\emph{can}}:\psi_{f,1}M \to \psi_{f,1}M(-1),\]
and, in the bounded derived category of complex mixed Hodge modules $D^b\text{\emph{MHM}}(X,\bC)$, there are canonical exact triangles
\[i_*i^!M\rightarrow\phi_{f,1}M\xrightarrow{\text{\emph{var}}} \psi_{f,1}M(-1)\xrightarrow{+1},\]
\[\psi_{f,1}M\xrightarrow{\text{\emph{can}}} \phi_{f,1}M\rightarrow i_*i^*M\xrightarrow{+1},\]
where $i:f^{-1}(0)\hookrightarrow X$ is the natural closed embedding (recall Remark \ref{remjjtriangles}).

\label{thmphipsiMHM}
    
\end{thm}

\begin{proof}

The statement follows for instance by Mochizuki's definition of mixed twistor $\scD$-module (\cite{M15a}, Definition 7.2.1), again using \cite{DS13}, Definition 3.2.1 and \cite{M15b}, Theorem 3.35. 
    
\end{proof}

\begin{rem}

From now on, the weight filtrations on $\psi_{f,e(\gamma)}\scM$ and $\phi_{f,e(\gamma)}\scM$ given above (i.e., with the correct choice of centring) will be referred to as the monodromy weight filtrations (relative monodromy weight filtrations) on $\psi_{f,e(\gamma)}\scM$ and $\phi_{f,e(\gamma)}\scM$, and will be denoted $M_{\bullet}$ ($M_{\bullet}^{\text{rel}}$, respectively). 
    
\end{rem}

Finally, under certain additional hypotheses to the ones required for filtered quasi-unipotency, some of which are automatically attained for mixed Hodge modules, we obtain expressions for the Hodge filtration on $\scM_f$ in terms of $V^0\scM_f$ and/or $V^{>0}\scM_f$.

\begin{lem}[\cite{MSai88}, Proposition 3.2.2 and Remarque 3.2.3]

Assume that $X$ and $f$ are as in Theorem \ref{thmphipsiMHM}, and that $(\scM,F_{\bullet})$ is a filtered left $\scD_X$-module, quasi-unipotent and regular along the hypersurface $f^{-1}(0)$. Then:

\begin{enumerate}[label=\roman*)]
\item If $\partial_t\cdot F_k\text{\emph{gr}}_V^1\scM_f = F_{k+1}\text{\emph{gr}}_V^0\scM_f$\footnote{Here, the filtrations are the ones induced by $F_{\bullet}\scM_f$, without any shift in the indexing.} for every $k$, then
\[F_k\scM_f = \sum_{i\geq 0}\partial_t^i\cdot(V^{>0}\scM_f\cap j_{f,*}j_f^*F_{k-i}\scM_f)\]
for every $k$, where $j_f:X\times\bC^*\hookrightarrow X\times\bC$ is the natural open embedding.
\item If $t:\text{\emph{gr}}_V^0\scM_f \to\text{\emph{gr}}_V^1\scM_f$ is injective and strict (with respect to $F_{\bullet}$) then 
\[F_k\scM_f = \sum_{i\geq 0}\partial_t^i\cdot(V^0\scM_f\cap j_{f,*}j_f^*F_{k-i}\scM_f)\]
for every $k$.
\end{enumerate}

\label{lemV0}
    
\end{lem}

\begin{rem}

Note that the strictness of $t:\text{gr}_V^0\scM_f \to\text{gr}_V^1\scM_f$ is indeed automatic whenever $\scM$ underlies a complex mixed Hodge module, so only injectivity need be checked in this situation. See also that the hypothesis of part i) is also satisfied whenever $\scM$ underlies a complex mixed Hodge module and $\partial_t : \text{gr}_V^1\scM_f \to\text{gr}_V^0\scM_f$ is surjective, again by the inherent strictness of morphisms of mixed Hodge modules.
    
\end{rem}

\vspace{5pt}

\subsection{Hodge and weight filtrations via $V$-filtrations} \label{HWV}

\; \vspace{5pt} \\ We now return to our study of the $\scD$-modules $\scO_{X,-\alpha f}$ and $\scM(f^{-\alpha})$, employing the framework of the previous section to obtain an expression for the Hodge and weight filtrations on $\scM(f^{-\alpha})$ in terms of the associated Kashiwara-Malgrange $V$-filtrations.

As usual, let $X$ be a complex manifold and $f\in\scO_X$ a holomorphic function such that the zero locus is a hypersurface. We write $i:Z:=f^{-1}(0)\hookrightarrow X$ for the natural closed embedding. Let $\alpha\in\bQ$.

\begin{prop}

In the category of complex mixed Hodge modules $\text{\emph{MHM}}(X,\bC)$ there are short exact sequences 
\[0 \rightarrow \phi_{f,1}\underline{\bC}_{X,-\alpha f}[n] \xrightarrow{\text{\emph{var}}} \psi_{f,1}\underline{\bC}_{X,-\alpha f}[n](-1)\rightarrow H^1 i_*i^!\underline{\bC}_{X,-\alpha f}[n] \rightarrow 0,\]
\[0 \rightarrow H^{-1} i_*i^*\underline{\bC}_{X,-\alpha f}[n] \rightarrow \psi_{f,1}\underline{\bC}_{X,-\alpha f}[n] \xrightarrow{\text{\emph{can}}} \phi_{f,1}\underline{\bC}_{X,-\alpha f}[n] \rightarrow 0.\]

\vspace{5pt} \noindent We therefore obtain, in $\text{\emph{MHM}}(X,\bC)$, an exact sequence
\[0 \rightarrow H^{-1} i_*i^*\underline{\bC}_{X,-\alpha f}[n] \rightarrow \psi_{f,1}\underline{\bC}_{X,-\alpha f}[n] \xrightarrow{N} \psi_{f,1}\underline{\bC}_{X,-\alpha f}[n](-1) \rightarrow \frac{M(f^{-\alpha})}{\underline{\bC}_{X,-\alpha f}[n]}\rightarrow 0.\]

\label{propcanvar}
    
\end{prop}

\begin{proof}

Of course, we use the exact triangles appearing in Theorem \ref{thmphipsiMHM}. 

Firstly, as seen in Remark \ref{remjjtriangles}, we have the exact triangle
\[i_*i^!\underline{\bC}_{X,-\alpha f}[n]\to\underline{\bC}_{X,-\alpha f}[n] \to j_*j^*\underline{\bC}_{X,-\alpha f}[n] \xrightarrow{+1},\]
where $j$ is the open embedding $j:\{f\neq 0\}\hookrightarrow X$. It is also clear by definition that
\[j_*j^*\underline{\bC}_{X,-\alpha f}[n] \simeq M(f^{-\alpha}).\]

Moreover, $f$ is clearly not a zero-divisor on the $\scO_X$-module $\scO_{X,-\alpha f}$, so the localisation map $\scO_{X,-\alpha f} \to \scM(f^{-\alpha})$ is injective, implying that $H^0i_*i^!\underline{\bC}_{X,-\alpha f}[n]=0$. We obtain therefore a short exact sequence
\[0\to\underline{\bC}_{X,-\alpha f}[n] \to M(f^{-\alpha}) \to H^1i_*i^!\underline{\bC}_{X,-\alpha f}[n] \to 0.\]

Applying the duality functor (Proposition \ref{propMHMfunctors}.iii)), we have also that
\[H^0i_*i^*\underline{\bC}_{X,-\alpha f}[n] = H^0i_*i^*\bD\underline{\bC}_{X,\alpha f}[n] = \bD H^0i_*i^!\underline{\bC}_{X,\alpha f}[n] =0.\]
Note that the dual of $\underline{\bC}_{X,-\alpha f}[n]$ equals $\underline{\bC}_{X,\alpha f}[n]$ since the local systems $\bV^{-\alpha}$ and $\bV^{\alpha}$ are dual, and the duality functor commutes with pullback and direct image.

Using the exact triangles in Theorem \ref{thmphipsiMHM} thus gives us the two short exact sequences as in the statement of the proposition. The final exact sequence then just follows by combining these two short exact sequences together (using that $N=\text{var}\circ\text{can}$), since as proven above we have that
\[H^1i_*i^!\underline{\bC}_{X,-\alpha f}[n] \simeq \frac{M(f^{-\alpha})}{\underline{\bC}_{X,-\alpha f}[n]}.\]

\end{proof}

\begin{lem}\vspace{-3pt}

\[j_{f,*}j_f^*F_{k+1}^Hi_{f,+}\scM(f^{-\alpha})\cap i_{f,+}\scM(f^{-\alpha}) = F_k^{t-\text{\emph{ord}}}i_{f,+}\scM(f^{-\alpha}):=\sum_{i=0}^k(i_{f,*}\scM(f^{-\alpha}))\partial_t^i.\vspace{-8pt}\] 

\label{lemjj}
    
\end{lem}

\begin{proof}

Consider the following Cartesian diagram. \vspace{3pt}

\begin{center}
\begin{tikzcd}
U \arrow[d,"j"] \arrow[rr, "i_f'"] & & X\times\bC^* \arrow[d,"j_f"]\\
X \arrow[rr, "i_f"] & & X\times\bC\\
\end{tikzcd}\vspace{-17pt}
\end{center}

\noindent Here, $i_f'$ is the graph embedding of the function $f$ on $U$.

As we've seen before, $j^*\scM(f^{-\alpha}) = \scO_{U,-\alpha f}$. Thus, recalling that the filtration on the pullback of a filtered $\scD$-module is simply given by pulling back the individual steps of the original filtration, we have that $j^*F_k^H\scM(f^{-\alpha}) = \scO_{U,-\alpha f}$ for all $k\geq 0$. Thus 

\begingroup
\allowdisplaybreaks
\begin{align*}
j_f^*F_{k+1}^Hi_{f,+}\scM(f^{-\alpha}) &= j_f^*\sum_{i\geq 0}i_{f,*}F_{k-i}^H\scM(f^{-\alpha})\partial_t^i\\[1pt]
&=\sum_{i\geq 0}i_{f,*}'j^*F_{k-i}^H\scM(f^{-\alpha})\partial_t^i\\[-4pt]
&=\sum_{i=0}^ki_{f,*}'\scO_{U,-\alpha f}\partial_t^i,\\[-15pt]
\end{align*}
\endgroup
thus implying that
\begin{align*}j_{f,*}j_f^*F_{k+1}^Hi_{f,+}\scM(f^{-\alpha})\cap i_{f,+}\scM(f^{-\alpha})&=\left(\sum_{i=0}^k(i_f\circ j)_*\scO_{U,-\alpha f}\partial_t^i\right)\cap i_{f,+}\scM(f^{-\alpha})\\&=\sum_{i=0}^k(i_{f,*}\scM(f^{-\alpha}))\partial_t^i\end{align*}
as required.
    
\end{proof}

\begin{cor}

\[F_{k+1}^HW_{n+l}i_{f,+}\scM(f^{-\alpha}) = \sum_{i\geq 0}\partial_t^i\cdot(V^0W_{n+l}i_{f,+}\scM(f^{-\alpha})\cap F_{k-i}^{t-\text{\emph{ord}}}i_{f,+}\scM(f^{-\alpha})).\]

\label{corV0}
    
\end{cor}

\begin{proof}

This now follows immediately by combining Lemma \ref{lemV0} and Lemma \ref{lemjj}, since 
\[W_{n+l}i_{f,+}\scM(f^{-\alpha})=i_{f,+}W_{n+l}\scM(f^{-\alpha}).\]
    
\end{proof}

\begin{cor}

\[V^0i_{f,+}\scM(f^{-\alpha})\cap F_{k+1}^Hi_{f,+}\scM(f^{-\alpha})=V^0i_{f,+}\scM(f^{-\alpha})\cap F_k^{t-\text{\emph{ord}}}i_{f,+}\scM(f^{-\alpha}).\]
Thus the Hodge filtrations overlying $\phi_{f,1}M(f^{-\alpha})$ and $\psi_{f,1}M(f^{-\alpha})(-1)$ are equal to the ones induced by the $t$-order filtrations on $i_{f,+}\scM(f^{-\alpha})$.

\label{corV02}
    
\end{cor}

\begin{proof}

This is clear, as Corollary $\ref{corV0}$ in particular implies that
\[F_k^{t-\text{ord}}V^0i_{f,+}\scM(f^{-\alpha})\subseteq F_{k+1}^Hi_{f,+}\scM(f^{-\alpha})\subseteq F_k^{t-\text{ord}}i_{f,+}\scM(f^{-\alpha}).\]
The final statement follows since 
\[\text{var}:\phi_{f,1}M(f^{-\alpha})\to\psi_{f,1}M(f^{-\alpha})(-1)\]
may easily be checked to be an isomorphism of mixed Hodge modules (using Proposition \ref{propVcomp}).
    
\end{proof}

\noindent Now consider, for each $\beta, \gamma \in\bQ$, the map
\[\psi_{-\beta}:i_{f,+}\scM(f^{-\gamma}) \to \scM(f^{-\gamma-\beta})\, ;\, u \partial_t^j \mapsto uQ_j(\beta)f^{-j-\beta},\]
for $u \in \scM(f^{-\gamma})$, where we recall that $Q_j(s)$ is the polynomial $Q_j(s)=s(s+1) \ldots (s+j-1)$ ($:=1$ if $j=0$).

The main theorem of Section \ref{sectionVfilt} is as follows. Using our understanding of the interaction between the Hodge filtrations on $\scM(f^{-\alpha})$ and $\psi_{f,1}\scO_{X,-\alpha f}$ seen above, we have the following expression for the Hodge and weight filtrations on $\scM(f^{-\alpha})$ in terms of the Kashiwara-Malgrange $V$-filtration.

\begin{thm}

For any $k,l \in \bZ_{\geq 0}$, we have the equalities
\[F_k^HW_{n+l}\scM(f^{-\alpha}) = \psi_0\left(F_k^{t-\text{\emph{ord}}}K_lV^0i_{f,+}\scM(f^{-\alpha})\right) = \psi_{-\alpha}\left(F_k^{t-\text{\emph{ord}}}K_lV^{\alpha}i_{f,+}\scO_X(*f)\right).\]

\label{thmmainformula}
    
\end{thm}

\begin{proof}

The second equality follows immediately once we note that $\psi_0=\psi_{-\alpha}\circ\Phi$, where $\Phi$ is the $s$-linear (thus preserving the kernel filtration) isomorphism appearing in Proposition \ref{propVcomp}. Thus it suffices to prove the first equality. 

Now, by Corollary \ref{corV0}, we have the equation
\begin{equation}\label{eq1}
F_{k+1}^HW_{n+l}i_{f,+}\scM(f^{-\alpha})=\sum_{i\geq 0}\partial_t^i\cdot\left(F_{k-i}^{t-\text{ord}}V^0W_{n+l}i_{f,+}\scM(f^{-\alpha})\right),
\end{equation}
noting also that we have an equality $V^0W_{n+l}i_{f,+}\scM(f^{-\alpha}) = W_{n+l}i_{f,+}\scM(f^{-\alpha}) \cap V^0i_{f,+}\scM(f^{-\alpha})$.

We have moreover that
\begin{equation}\label{eq2}
F_{k+1}^HW_{n+l}i_{f,+}\scM(f^{-\alpha})=F_{k+1}^Hi_{f,+}W_{n+l}\scM(f^{-\alpha}) = \sum_{i\geq 0}\left(F_{k-i}^HW_{n+l}\scM(f^{-\alpha})\right)\partial_t^i.
\end{equation}

We thus reduce to obtaining an expression for $F_k^{t-\text{ord}}V^0W_{n+l}i_{f,+}\scM(f^{-\alpha})$.

Firstly, Proposition \ref{propVcomp} implies that we have an equality
\begin{equation}\label{eqnew1}
V^{>0}i_{f,+}\scO_{X,-\alpha f}=V^{>0}W_{n+l}i_{f,+}\scM(f^{-\alpha})=V^{>0}i_{f,+}\scM(f^{-\alpha}),
\end{equation}
which implies in particular that 
\[\text{gr}_V^1W_{n+l}i_{f,+}\scM(f^{-\alpha})=\text{gr}_V^1i_{f,+}\scM(f^{-\alpha}).\]

Now, by \cite{MSai90}, Corollary 1.9 (condition (1.9.2) specifically), the weight filtration $W_{\bullet}$ on $i_{f,+}\scM(f^{-\alpha})$ and relative monodromy weight filtration $M_{\bullet}^{\text{rel}}$ on $\text{gr}_V^0i_{f,+}\scM(f^{-\alpha})$ satisfy the equation (in $\text{gr}_V^0i_{f,+}\scM(f^{-\alpha})$)
\[\text{gr}_V^0W_ii_{f,+}\scM(f^{-\alpha})=\partial_t\cdot\text{gr}_V^1i_{f,+}\scM(f^{-\alpha}) + M_i^{\text{rel}}\text{gr}_V^0i_{f,+}\scM(f^{-\alpha}),\]
since $M_{\bullet}^{\text{rel}}$ is by definition the relative monodromy weight filtration associated to the naïve limit filtration $L_i\text{gr}_V^0i_{f,+}\scM(f^{-\alpha}):=\text{gr}_V^0W_ii_{f,+}\scM(f^{-\alpha})$.

This induces in turn an equality
\[V^0W_{n+l}i_{f,+}\scM(f^{-\alpha})=\partial_t\cdot V^1i_{f,+}\scM(f^{-\alpha}) + M_{n+l}^{\text{rel}}V^0i_{f,+}\scM(f^{-\alpha}),\]
where $M_{\bullet}^{\text{rel}}V^0i_{f,+}\scM(f^{-\alpha})$ is the pullback of the relative monodromy weight filtration to $V^0$.

Now, considering the (bi-strict) morphism $\text{var}$ of mixed Hodge modules described in Theorem \ref{thmphipsiMHM}, and using Proposition \ref{propVcomp}, we have an equality
\[t\cdot M_{n+l}^{\text{rel}}V^0i_{f,+}\scM(f^{-\alpha})=M_{n+l-2}^{\text{rel}}V^1i_{f,+}\scM(f^{-\alpha}).\]
Moreover, Equation \eqref{eqnew1} implies that we have equalities
\[V^1i_{f,+}\scO_{X,-\alpha f}=V^1W_ii_{f,+}\scM(f^{-\alpha})=V^1i_{f,+}\scM(f^{-\alpha})\]
for all $i \geq n$, implying that the naïve limit filtration $L_i\psi_{f,1}\scM(f^{-\alpha})=\text{gr}_V^1W_{i+1}i_{f,+}\scM(f^{-\alpha})$ is trivial, which in turn implies by the definition of the relative monodromy weight filtration that it equals the monodromy weight filtration centred at $n-1$, i.e. that
\[M_{n+l-2}^{\text{rel}}V^1i_{f,+}\scM(f^{-\alpha})=M_{n+l-2}V^1i_{f,+}\scM(f^{-\alpha}).\]
We have an explicit expression for this monodromy weight filtration:
\[M_{n+l-2}V^1i_{f,+}\scM(f^{-\alpha})=\sum_{m\geq 0}(t\partial_t)^m\cdot K_{2m+l}V^1i_{f,+}\scM(f^{-\alpha}),\]
which, combining with the above expression, gives us that
\begin{equation}\label{eqVWK}
V^0W_{n+l}i_{f,+}\scM(f^{-\alpha})=\partial_t\cdot V^1i_{f,+}\scM(f^{-\alpha}) + K_lV^0i_{f,+}\scM(f^{-\alpha}).
\end{equation}

Of course, we have to incorporate the $t$-order filtration as well to conclude the proof. For this, recall that the morphism of filtered $\scD$-modules
\[(\partial_tt)^l\partial_t:\text{gr}_V^1i_{f,+}\scM(f^{-\alpha})\to \left(\text{gr}_V^0i_{f,+}\scM(f^{-\alpha})\right)(-l)\]
underlies a morphism of mixed Hodge modules, and is thus strict with respect to the Hodge filtration, which on both sides is induced by the $t$-order filtration (by Corollary \ref{corV02}). This implies that we have the stronger equality
\[F_k^{t-\text{ord}}V^0W_{n+l}i_{f,+}\scM(f^{-\alpha})=\partial_t\cdot F_{k-1}^{t-\text{ord}}V^1i_{f,+}\scM(f^{-\alpha}) + F_k^{t-\text{ord}}K_lV^0i_{f,+}\scM(f^{-\alpha}).\]

Indeed, the inclusion $\supseteq$ is immediate by Equation \ref{eqVWK}. Conversely, given an element $u\in F_k^{t-\text{ord}}V^0W_{n+l}i_{f,+}\scM(f^{-\alpha})$, we may by the same equation write $u=\partial_t\cdot u_1+u_2$, where $u_1\in V^1i_{f,+}\scM(f^{-\alpha})$ and $u_2\in K_lV^0i_{f,+}\scM(f^{-\alpha})$.

Then $(\partial_tt)^l\cdot \overline{u} \in F_{k+l}^{t-\text{ord}}\text{gr}_V^0i_{f,+}\scM(f^{-\alpha})$ is in the image of the map $(\partial_tt)^l\partial_t$ mentioned above, which, by the strictness of said map, implies that we may write $(\partial_tt)^l\cdot \overline{u} = (\partial_tt)^l\partial_t\cdot \overline{u_3}$, for some $u_3\in F_{k-1}^{t-\text{ord}}V^1i_{f,+}\scM(f^{-\alpha})$, so that
\[(\partial_tt)^l\cdot u\in (\partial_tt)^l\partial_t \cdot F_{k-1}^{t-\text{ord}}V^1i_{f,+}\scM(f^{-\alpha})+V^{>0}i_{f,+}\scM(f^{-\alpha}),\]
which in turn implies that
\[u\in \partial_t \cdot F_{k-1}^{t-\text{ord}}V^1i_{f,+}\scM(f^{-\alpha})+F_k^{t-\text{ord}}K_lV^0i_{f,+}\scM(f^{-\alpha}).\]

Finally, recalling that the map $\psi_0$ is just given by mapping $\partial_t$ to 0, combining Equations \eqref{eq1} and \eqref{eq2} with the above equality now gives us that
\[F_k^HW_{n+l}\scM(f^{-\alpha}) = \psi_0\left(F_k^{t-\text{ord}}K_lV^0i_{f,+}\scM(f^{-\alpha})\right)\]
as required.
    
\end{proof}

\begin{rem}

In the above proof, we actually proved that there is a natural isomorphism of complex mixed Hodge modules
\[\psi_{f,1}\underline{\bC}_{X,-\alpha f}[n] \simeq \psi_{f,1}M(f^{-\alpha}).\]
    
\end{rem}

\begin{cor}

The left $\scD_X$-module map 
\[\psi_{-1}:\text{\emph{gr}}_V^1i_{f,+}\scO_{X,-\alpha f} \rightarrow \frac{\scM(f^{-\alpha})}{\scO_{X,-\alpha f}},\]
induced by the map $\psi_{-1}$ defined above, underlies a morphism of mixed Hodge modules
\[\psi_{f,1}\underline{\bC}_{X,-\alpha f}[n](-1)\to H^1 i_*i^!\underline{\bC}_{X,-\alpha f}[n].\] 
Moreover, this morphism equals up to automorphism the morphism appearing in Proposition \ref{propcanvar}.

\label{cortauH}
    
\end{cor}

\begin{proof}

The first statement follows by Theorem \ref{thmmainformula} and by the commutativity of the following diagram.
\begin{center}
\begin{tikzcd}
    V^0i_{f,+}\scM(f^{-\alpha}) \arrow[rd, "\psi_0"] \arrow[dd, "t"] & \\[-15pt]
    & \scM(f^{-\alpha})\\[-15pt]
    V^1i_{f,+}\scO_{X,-\alpha f} \arrow[ur, "\psi_{-1}"']\\
\end{tikzcd}
\end{center}

\vspace{-17pt} \noindent It is moreover a simple exercise to check that the map $\psi_{-1}$ in the statement of the Corollary is surjective and has kernel equal to the image of $t\partial_t : \text{gr}_V^1 \to \text{gr}_V^1$. Comparing with the final exact sequence appearing in Proposition \ref{propcanvar}, we obtain the second statement in the Corollary.
    
\end{proof}

\vspace{5pt}

\subsection{Examples: SNC and weighted homogeneous isolated singularities}\label{subsectionegs}

\; \vspace{5pt} \\ We now use Theorem \ref{thmmainformula} to calculate the Hodge and weight filtrations for some examples, including for simple normal crossing divisors, the formula of which was used crucially in Section 2 above.

We restate Theorem \ref{thmSNCFW} for ease of reading. Let $X=\bC^n$, $f=\prod_{i=1}^nx_i^{a_i}\in\scO_X$ and $\alpha \in \bQ$. Recall the definitions of $I$, $I_{\alpha}$ and $m_{\alpha}$:
\[I:=\{i \in [n] \; |\; a_i \neq 0\}, \;\;\;\; I_{\alpha}:=\{i \in I\; |\; \alpha a_i \in \bZ\} \;\; \text{ and } \;\; m_{\alpha}:=|I_{\alpha}|.\]
We assume that $I\neq \emptyset$.

\begin{thm}

Assume $\alpha >0$. With $X$ and $f$ as above, we have the following expressions.

\begin{enumerate}[label = \roman*)]

\item \[W_{n+m_{\alpha}}\scM(f^{-\alpha}) = \scM(f^{-\alpha}).\vspace{5pt}\]

\item 
Let $l\in\bZ$ such that $0 \leq l \leq m_{\alpha}$. Then
\[F_0^HW_{n+l}\scM(f^{-\alpha}) = \sum_{J\subseteq I_{\alpha},\, |J|=l}\scO_X\prod_{i \in I_{\alpha}\backslash J}x_i^{\lceil \alpha a_i\rceil}\prod_{i \in (I\backslash I_{\alpha})\cup J}x_i^{\lceil \alpha a_i\rceil-1} f^{-\alpha}\]
and, for $k \in \bZ_{\geq 0}$, 
\[F_k^HW_{n+l}\scM(f^{-\alpha})=F_k\scD_X\cdot F_0^HW_{n+l}\scM(f^{-\alpha}).\]
    
\end{enumerate}

\label{thmSNCFW2}
    
\end{thm}

\begin{proof}

Of course, we are going to use Theorem \ref{thmmainformula} to prove this. To this end, we will prove an expression for $K_lV^{\alpha}i_{f,+}\scO_X$. It is well known (see \cite{MSai90}, Theorem 3.4, as well as  \cite{CDM24} Example 4.8 for instance) that in this situation we have the following expression for the $V$-filtration on $i_{f,+}\scO_X$:
\[V^{\lambda}i_{f,+}\scO_X = \sum_{j\geq 0}\scD_X\cdot f^{(\lambda+j)}\partial_t^j\]
for all $\lambda \in\bQ$, where $f^{(\lambda)}:=\prod_{i=1}^nx_i^{\lceil\lambda a_i\rceil-1}$ (and where $x_i^k:=1$ for any $k\in\bZ_{<0}$). In particular, is is easy to show that this implies
\[V^{\lambda}i_{f,+}\scO_X = \scD_X\cdot f^{(\lambda)}\]
for all $\lambda>0$.

First we prove the following simple claim.

\vspace{10pt}

\textbf{Claim.} Let $b\in\bZ^n_{\geq 0}$ and $k \in \bZ_{\geq 0}$. We write $\text{Supp}(b):=\{i \in[n]\mid b_i\neq 0\}$. Assume that $\text{Supp}(b) \subseteq I$. Then $F_k^{t-\text{ord}}i_{f,+}\scO_X \cap \scD_X\cdot x^b = F_k\scD_X\cdot x^b$.

\emph{Proof.} Of course the inclusion $F_k\scD_X\cdot x^b\subseteq F_k^{t-\text{ord}}i_{f,+}\scO_X \cap \scD_X\cdot x^b$ is clear. To prove the reverse inclusion, we take $u\in F_k^{t-\text{ord}}i_{f,+}\scO_X\cap \scD_X\cdot x^b$ and write $u=P\cdot x^b$, with $P\in F_r\scD_X$ say. If $r\leq k$ then we're done, so we assume that $r>k$. We prove the claim by then proving that $u \in F_{r-1}\scD_X\cdot x^b$ also.

Firstly, we may write 
\[P=\sum_{|\alpha|=r}f_{\alpha}\partial^{\alpha} + P',\]
with $P'\in F_{r-1}\scD_X$, for some $f_{\alpha}\in \scO_X$. Since $\text{Supp}(b)\subseteq I$, $\partial_i \cdot x^b =0$ for any $i \notin I$. Therefore we may assume that $f_{\alpha}=0$ whenever $\text{Supp}(\alpha)\nsubseteq I$.

Now the coefficient of $\partial_t^r$ in the expression for $P\cdot x^b$ is equal to zero, i.e. we have that
\[\sum_{|\alpha|=r}f_{\alpha}\prod_{i=1}^n\partial_i(f)^{\alpha_i}=0\]
(here $0^0:=1$). This in turn of course implies that
\[\sum_{|\alpha|=r}f_{\alpha}\prod_{i=1}^n\left(\frac{a_i}{x_i}\right)^{\alpha_i}=0.\]
Writing
\[f_{\alpha}=\sum_{\beta \in \bZ_{\geq 0}^n}c_{\alpha,\beta}x^{\beta},\]
with $c_{\alpha,\beta}\in\bC$, we then have, for each $\gamma\in\bZ^n$, that
\[\sum_{|\alpha|=r}c_{\alpha,\gamma+\alpha}\prod_{i=1}^na_i^{\alpha_i}=0.\]
Here we define $c_{\alpha,\gamma}:=0$ if $\alpha\notin\bZ^n_{\geq 0}$ or $\gamma\notin\bZ^n_{\geq 0}$. In particular we have that $c_{\alpha,0}=0$ for all $|\alpha|=r$. See also that this equivalently gives, for each $\gamma,\delta\in\bZ^n$, that
\begin{equation}\label{eqc}\sum_{|\alpha|+|\delta|=r}c_{\delta+\alpha,\gamma+\alpha}\prod_{i=1}^na_i^{\alpha_i}=0.\end{equation}

Now take $\gamma,\delta\in\bZ^n$ such that $|\delta|=r-1$. Choose also $j \in I$. Then, since for any $i \in [n]$,
\begin{equation}\label{eqann}(a_jx_i\partial_i-a_ix_j\partial_j-(a_ib_j-a_jb_i))\cdot x^b =0\end{equation}
in $i_{f,+}\scO_X$, there exists some $\widetilde{P}_{\gamma,\delta}\in F_{r-1}\scD_X$ such that
\[\left[a_j\sum_{i=1}^nc_{\delta+1_i,\gamma+1_i}x^{\gamma+1_i}\partial^{\delta+1_i}-\sum_{i=1}^nc_{\delta+1_i,\gamma+1_i}a_ix^{\gamma+1_j}\partial^{\delta+1_j}\right]\cdot x^b = \widetilde{P}_{\gamma,\delta}\cdot x^b.\]
Here, $1_i$ is the vector with a one in the $i$-th entry and zeroes elsewhere. We know however by Equation \eqref{eqc} that $\sum_{i=1}^nc_{\delta+1_i,\gamma+1_i}a_i=0$, so we see that 
\[\left[a_j\sum_{i=1}^nc_{\delta+1_i,\gamma+1_i}x^{\gamma+1_i}\partial^{\delta+1_i}\right]\cdot x^b = \widetilde{P}_{\gamma,\delta}\cdot x^b.\]
We thus finally see that
\begin{align*}
u=P\cdot x^b &= \left[\sum_{\alpha,\beta,|\alpha|=r}c_{\alpha,\beta}x^{\beta}\partial^{\alpha} \right]\cdot x^b + P'\cdot x^b\\
&=\left[\frac{1}{n}\sum_{\gamma,\delta,i, |\delta|=r-1}c_{\delta+1_i,\gamma+1_i}x^{\gamma+1_i}\partial^{\delta+1_i}\right]\cdot x^b+P'\cdot x^b\\
&=\left[\frac{1}{na_j}\sum_{\gamma,\delta,|\delta|=r-1}\widetilde{P}_{\gamma,\delta}+P'\right]\cdot x^b\in F_{r-1}\scD_X\cdot x^b. \hspace{30pt}\square
\end{align*}

\vspace{10pt}

\noindent Combining the statement of this claim with the above expression for the $V$-filtration gives us that
\[F_k^{t-\text{ord}}V^{\lambda}i_{f,+}\scO_X = F_k\scD_X\cdot f^{(\lambda)}\]
for all $\lambda>0$.

Next, we prove the following expression for the kernel filtration.
\begin{equation}\label{eqkernel}F_k^{t-\text{ord}}K_lV^{\lambda}i_{f,+}\scO_X=\sum_{J\subseteq I_{\lambda}, \,|J|=l}F_k\scD_X\cdot f^{(\lambda)}\prod_{i\in I_{\lambda}\backslash J}x_i\,\,\,\text{ for all }\,\, \lambda >0,\end{equation}
where $I_{\lambda}$ is defined analogously to $I_{\alpha}$, i.e. $I_{\lambda}:=\{i\in I\mid\lambda a_i\in\bZ\}$.

We prove this via induction on $l$. For $l=0$, 
\[F_k^{t-\text{ord}}K_0V^{\lambda}i_{f,+}\scO_X=F_k^{t-\text{ord}}V^{>\lambda}i_{f,+}\scO_X=F_k\scD_X\cdot \prod_{i=1}^nx_i^{\lceil(\lambda+\epsilon)a_i\rceil-1} = F_k\scD_X\cdot \prod_{i=1}^nx_i^{\lceil\lambda a_i\rceil-1}\prod_{i\in I_{\lambda}}x_i\]
as required, where $0<\epsilon<<1$.

Now let $l>0$. We prove first the inclusion $\supseteq$. For this, see that, for $b\in\bZ_{\geq 0}^n$ and $j\in I$,
\[(\partial_tt-\lambda)\cdot x^b = -\frac{\partial_j}{a_j}\cdot x_jx^b + \left(\frac{b_j+1}{a_j}-\lambda\right)x^b.\]
Thus, if $b$ is such that $x^b = f^{(\lambda)}\prod_{i\in I_{\lambda}\backslash J}x_i$ for some $J\subseteq I_{\lambda}$ with $|J|=l$ and $j \in J$,
\[(\partial_tt-\lambda)\cdot x^b = -\frac{\partial_j}{a_j}\cdot x_jx^b \in K_{l-1}V^{\lambda}i_{f,+}\scO_X\]
by the inductive hypothesis, thus implying that $x^b\in K_lV^{\lambda}i_{f,+}\scO_X$ as required.

We prove the reverse inclusion via induction on $k$ (within our original induction on $l$). For the case $k=0$, consider $u \in F_0^{t-\text{ord}}K_lV^{\lambda}i_{f,+}\scO_X$. Then, by the claim, we may write $u=\sum_{b \in \bZ^n_{\geq 0}} \nu_bx^bf^{(\lambda)}$, for some $\nu_b \in\bC$. 

Now, $(\partial_tt-\lambda)\cdot u \in F_1^{t-\text{ord}}K_{l-1}V^{\lambda}i_{f,+}\scO_X$, so, by our induction hypothesis on $l$, 
\[(\partial_tt-\lambda)\cdot\sum_{b \in \bZ^n_{\geq 0}} \nu_bx^bf^{(\lambda)}\in \sum_{J\subseteq I_{\lambda}, \,|J|=l-1}F_1\scD_X\cdot f^{(\lambda)}\prod_{i\in I_{\lambda}\backslash J}x_i.\]

Comparing coefficients of $\partial_t^0$, we see that for every $b\in \bZ_{\geq 0}^n$ with $\nu_b\neq 0$ we must have 
\[x^bf^{(\lambda)} = x_{j_b}^{-1}x^{c_b}f^{(\lambda)}\prod_{i \in I_{\lambda}\backslash J}x_i\]
for some $J\subseteq I_{\lambda}$ with $|J|=l-1$, for some $j_b \in [n]$, and some $c_b\in\bZ^n_{\geq 0}$.

This in turn implies that we may write
\[x^bf^{(\lambda)} = x^{\widetilde{c_b}}f^{(\lambda)}\prod_{i \in I_{\lambda}\cup\{j_b\}\backslash J\cup\{j_b\}}x_i,\]
where
\[\widetilde{c_b}:=\begin{cases}c_b & : \,\,\,j_b \in I_{\lambda}\backslash J \\ c_b - 1_{j_b} &: \,\,\,\text{otherwise.}\end{cases}\]
See that, as $b \in \bZ_{\geq 0}^n$, $\widetilde{c_b}\in\bZ_{\geq 0}^n$ also.

So we see that $x^bf^{(\lambda)}$ is of the required form as on the right hand side of Equation \eqref{eqkernel}, implying of course that $u$ is as well.

Now we prove the case $k>0$. The proof is similar to the case $k=0$. Consider $u \in F_k^{t-\text{ord}}K_lV^{\lambda}i_{f,+}\scO_X$. Then we may write $u=P\cdot f^{(\lambda)}$ for some $P\in F_k\scD_X$. Write $P=\sum_{b,c \in \bZ^n_{\geq 0}} \nu_{b,c}\partial^cx^b$, with $\nu_{b,c} \in\bC$. 

Now, $(\partial_tt-\lambda)\cdot u \in F_{k+1}^{t-\text{ord}}K_{l-1}V^{\lambda}i_{f,+}\scO_X$, so, by the induction hypothesis on $l$, 
\[(\partial_tt-\lambda)\cdot\sum_{b,c \in \bZ^n_{\geq 0}} \nu_{b,c}\partial^c\cdot x^bf^{(\lambda)}\in \sum_{J\subseteq I_{\lambda}, \,|J|=l-1}F_{k+1}\scD_X\cdot f^{(\lambda)}\prod_{i\in I_{\lambda}\backslash J}x_i.\]

Comparing coefficients of $\partial_t^0$, we see that for every $b,c\in \bZ_{\geq 0}^n$ with $\nu_{b,c}\neq 0$ we must have 
\[x^{b-c}f^{(\lambda)} = x^{\delta_{b,c}-\gamma_{b,c}}f^{(\lambda)}\prod_{i \in I_{\lambda}\backslash J_{b,c}}x_i\]
for some $J_{b,c}\subseteq I_{\lambda}$ with $|J_{b,c}|=l-1$ and for some $\gamma_{b,c},\delta_{b,c} \in \bZ^n_{\geq 0}$ with $|\gamma_{b,c}|\leq k+1$.

Now see that, since the monomial $f^{(\lambda)}\prod_{i \in I_{\lambda}\backslash J_{b,c}}x_i$ is independent of the variables $x_i$ with $i \notin I$, we have that for any $\nu_{b,c}\neq 0$, $(\delta_{b,c})_i \geq (\gamma_{b,c})_i$ for any $i \notin I$. Therefore without loss of generality we may assume that $(\gamma_{b,c})_i=0$ for any $i\notin I$. Furthermore, we may assume without loss of generality that $|\gamma_{b,c}|=k+1$, by increasing the entries of $\delta_{b,c}$ and $\gamma_{b,c}$ if necessary.

Now, take any $b,c\in\bZ_{\geq 0}^n$ with $\nu_{b,c}\neq 0$ and $|c|=k$. Choose any $j\in [n]$ such that $(\gamma_{b,c})_j \neq 0$. Then Equation \eqref{eqann} gives us that
\begin{align*}\partial^c\cdot x^bf^{(\lambda)} &- \mu_{b,c}\partial^{\gamma_{b,c}-1_j} \cdot x_j^{-1}x^{\delta_{b,c}}f^{(\lambda)}\prod_{i\in I_{\lambda}\backslash J_{b,c}}x_i =\\&\partial^c\cdot x^cx^{\delta_{b,c}-\gamma_{b,c}}f^{(\lambda)}\prod_{i \in I_{\lambda}\backslash J_{b,c}}x_i - \mu_{b,c}\partial_t^{\gamma_{b,c}-1_j} \cdot x_j^{-1}x^{\delta_{b,c}}f^{(\lambda)}\prod_{i\in I_{\lambda}\backslash J_{b,c}}x_i \in F_{k-1}^{t-\text{ord}}i_{f,+}\scO_X, \end{align*}
where $\mu_{b,c}$ is the constant given by
\[\mu_{b,c}:=\prod_{i\in [n]}a_i^{c_i}\prod_{i\in I}a_i^{-(\gamma_{b,c}-1_j)_i}.\]

But $\partial_t^{\gamma_{b,c}-1_j} \cdot x_j^{-1}x^{\delta_{b,c}}f^{(\lambda)}\prod_{i\in I_{\lambda}\backslash J_{b,c}}x_i$ is of the form appearing on the right hand side of Equation \eqref{eqkernel}, so in particular, as we've already shown earlier in the proof, lies in $F_k^{t-\text{ord}}K_lV^{\lambda}i_{f,+}\scO_X$.

Thus we in fact have that
\[u-\sum_{b,c\in\bZ_{\geq 0}^n, |c|=k}\nu_{b,c} \mu_{b,c}\partial_t^{\gamma_{b,c}-1_j} \cdot x_j^{-1}x^{\delta_{b,c}}f^{(\lambda)}\prod_{i\in I_{\lambda}\backslash J_{b,c}}x_i\in F_{k-1}^{t-\text{ord}}K_lV^{\lambda}i_{f,+}\scO_X.\]

By our inductive hypothesis (on $k$) we may thus write the above expression in a form as on the right hand side of Equation \eqref{eqkernel}, implying the same for $u$. This thus finishes the proof of Equation \eqref{eqkernel}.

\vspace{10pt}

\noindent Now we conclude by using the expression in Equation \eqref{eqkernel} to prove the parts i) and ii) in the statement of the Theorem. 

For part i), see that our expression \eqref{eqkernel} gives us that 
\[K_{m_{\alpha}}V^{\alpha}i_{f,+}\scO_X = V^{\alpha}i_{f,+}\scO_X,\]
implying by Theorem \ref{thmmainformula} that
\[W_{n+m_{\alpha}}\scM(f^{-\alpha})=\psi_{-\alpha}\left(K_{m_{\alpha}}V^{\alpha}i_{f,+}\scO_X\right) = \psi_{-\alpha}\left(V^{\alpha}i_{f,+}\scO_X\right) = \scM(f^{-\alpha}).\]

For part ii), see that
\begin{align*}F_0^HW_{n+l}\scM(f^{-\alpha})&=\psi_{-\alpha}\left(F_0^{t-\text{ord}}K_lV^{\alpha}i_{f,+}\scO_X\right) \\&= \psi_{-\alpha}\left(\sum_{J\subseteq I_{\alpha}, \,|J|=l}\scO_X f^{(\alpha)}\prod_{i\in I_{\alpha}\backslash J}x_i\right)\\&=\sum_{J\subseteq I_{\alpha}, \,|J|=l}\scO_X f^{(\alpha)}\prod_{i\in I_{\alpha}\backslash J}x_i f^{-\alpha},\end{align*}
and that
\begin{align*}F_k^HW_{n+l}\scM(f^{-\alpha})&=\psi_{-\alpha}\left(F_k^{t-\text{ord}}K_lV^{\alpha}i_{f,+}\scO_X\right)\\&=\psi_{-\alpha}\left(F_k\scD_X\cdot F_0^{t-\text{ord}}K_lV^{\alpha}i_{f,+}\scO_X\right)\\&=F_k\scD_X\cdot \psi_{-\alpha}\left( F_0^{t-\text{ord}}K_lV^{\alpha}i_{f,+}\scO_X\right)\\&=F_k\scD_X\cdot F_0^HW_{n+l}\scM(f^{-\alpha}).\end{align*}
    
\end{proof}

\noindent The second class of examples we can calculate for is divisors with (positively) weighted homogeneous isolated singularities. For this, we again assume that $X=\bC^n$, now with $n\geq 2$, and that $f$ is a polynomial that is weighted homogeneous with weight vector $w=(w_1 \,w_2\,\ldots\,w_n) \in \bQ^n_{>0}$. This means that
\[E(f)=f, \,\,\,\,\text{ where }\,\,\, E:=\sum_{i=1}^nw_ix_i\partial_i.\]
We assume in addition that $f$ has an isolated singularity at the origin. We say that a polynomial $g\in\scO_X$ is \emph{homogeneous with respect to $w$} if there exists some rational number $\rho(g)$ such that
\[E(g)=\rho(g)g.\]
We also write, for $\gamma\in\bQ$,
\[\scO^{\geq \gamma}:=\text{Span}_{\bC}\{g\in\scO_X\mid g\text{ homogeneous polynomial w.r.t. }w, \,\, \rho(g)\geq\gamma\} \subseteq \scO_X.\]
We define $\scO^{>\gamma}$ and $\scO^{\gamma}$ similarly. Of course, $\scO^{\geq\gamma}$ and $\scO^{>\gamma}$ are ideals of $\scO_X$.

In this case, we also have a formula for the $V$-filtration, proven by M. Saito, \cite{MSai09}, (4.2.1). In this same paper, Saito used this expression to obtain a formula for the Hodge filtration on $\scO_X(*f)$ (Theorem 0.7). Zhang (\cite{Z21}, Theorem A) extended this result to the twisted setting. We are also now able to calculate the Hodge filtration on the weight filtration steps.

\begin{thm}

Assume that $X$ and $f$ are as above. Then 

\begin{enumerate}[label=\roman*)]

\item If $\alpha\in\bQ\cap(0,1)$, then 
\[W_{n+1}\scM(f^{-\alpha})=\scM(f^{-\alpha})\]
and
\[F_k^HW_n\scM(f^{-\alpha})=\sum_{j\geq 0}F_{k-j}\scD_X\cdot \scO^{>\alpha+j-|w|}f^{-j-\alpha}.\]

\item \[W_{n+2}\scO_X(*f)=\scO_X(*f)\]
and
\[F_k^HW_{n+1}\scO_X(*f)=\sum_{j\geq 0}F_{k-j}\scD_X\cdot \scO^{>1+j-|w|}f^{-j-1}.\]
    
\end{enumerate}

\label{thmWHom}
    
\end{thm}

\begin{proof}

The reason why we are able to calculate the Hodge and weight filtrations precisely under these assumptions is that we have the following formula for the $V$-filtration, as seen in \cite{MSai09}, (4.2.1):

For $0<\lambda\leq 1$,
\[F_k^{t-\text{ord}}V^{\lambda}i_{f,+}\scO_X=\sum_{j\geq 0}F_{k-j}\scD_X\cdot \scO^{\geq\lambda+j-|w|}\partial_t^j.\]

One may easily (see for instance the proof of Theorem 6.1 of \cite{CDM24}) then prove that
\[(\partial_tt-\lambda)\cdot V^{\lambda}i_{f,+}\scO_X \subseteq V^{>\lambda}i_{f,+}\scO_X \,\,\,\text{ for }\,\,\, 0<\lambda<1\]
and that
\[(\partial_tt-1)^2\cdot V^1i_{f,+}\scO_X \subseteq V^{>1}i_{f,+}\scO_X,\]
thus proving the first three statements of the theorem by Theorem \ref{thmmainformula}.

It now only remains to find an expression for $F_k^{t-\text{ord}}K_1V^1i_{f,+}\scO_X$. This is relatively straightforward; given $u \in F_k^{t-\text{ord}}V^1i_{f,+}\scO_X$, $u \in K_1V^1i_{f,+}\scO_X$ if and only if $t\partial_t\cdot u \in V^{>1}i_{f,+}\scO_X=t\cdot V^{>0}i_{f,+}\scO_X$, which in turn is the case if and only if $\partial_t\cdot u\in V^{>0}i_{f,+}\scO_X$. Using the above formula for the $V$-filtration, it is easy to see that
\[\text{Im}(\partial_t:i_{f,+}\scO_X \to i_{f,+}\scO_X)\cap F_{k+1}^{t-\text{ord}}V^{>0}i_{f,+}\scO_X = \sum_{j\geq 1}F_{k+1-j}\scD_X\cdot \scO^{>j-|w|}\partial_t^j.\]
This finally gives us the formula
\begin{equation}\label{eqkernelpwhom}F_k^{t-\text{ord}}K_1V^1i_{f,+}\scO_X=\sum_{j\geq 0}F_{k-j}\scD_X\cdot\scO^{>1+j-|w|}\partial_t^j.\end{equation} 
Theorem \ref{thmmainformula} then gives us that
\[F_k^HW_{n+1}\scO_X(*f)=\psi_{-1}\left(F_k^{t-\text{ord}}K_1V^1i_{f,+}\scO_X\right)=\sum_{j\geq 0}F_{k-j}\scD_X\cdot \scO^{>1+j-|w|}f^{-j-1}\]
as required.
\end{proof}

\begin{rem}

In Section 4, we will introduce the weighted microlocal multiplier ideals (or weighted higher multiplier ideals). The above formula for the $V$-filtration gives us a formula for the weighted microlocal multiplier ideals in this case. Some elementary calculation (use Lemma \ref{lemmicromultformulae}) yields:
\[W_0\widetilde{V}^{k+\alpha}\scO_X=\sum_{j=0}^k\scO_X\left(\sum_{\gamma\in\bZ_{\geq 0}^n, \,\, |\gamma|=k-j}\prod_{i=1}^n\partial_i(f)^{\gamma_i}\right)\scO^{>\alpha+j-|w|},\]
for $\alpha\in\bQ\cap(0,1]$.

This is of course nothing new, as $W_0\widetilde{V}^{k+\alpha}=\widetilde{V}^{>k+\alpha}$, and the above formula for the $V$-filtration extends to give a formula for the microlocal $V$-filtration (\cite{MSai09}, (4.2.1)). We could alternatively have used microlocal methods to prove Theorem \ref{thmWHom}.
        
\end{rem}

\newpage

\section{Microlocalisation and weighted minimal exponents}\label{sectionmicrolocal}

\noindent In this section we introduce new invariants called weighted minimal exponents. These generalise the minimal exponent, a well-studied birational invariant. We investigate how these invariants are related to microlocalisation and derive several results relating these invariants to the mixed Hodge module structure overlying $\scM(f^{-\alpha})$ described in previous sections.

\vspace{10pt}

\subsection{Weighted minimal exponents and microlocal multiplier ideals}

\; \vspace{5pt} \\ Let $X$ be a complex manifold of dimension $n$ and $f\in\scO_X$ such that $Z:=f^{-1}(0)$ is a hypersurface in $X$. Let $\fx\in Z$ and write $f_{\fx}\in\scO_{X,\hspace{0.7pt}\fx}$ for the germ of $f$ at $\fx$.

\begin{defn}

The \emph{(local) Bernstein-Sato polynomial of $f$ at $\fx$} is the unique monic polynomial of minimal degree $b_{f,\hspace{0.7pt}\fx}(s)\in\bC[s]$ satisfying that
\[b_{f,\hspace{0.7pt}\fx}(s)\cdot f_{\fx}^s\in\scD_{X,\hspace{1pt}\fx}[s]\cdot f_{\fx}^{s+1}.\]
Write $\rho_{f,\hspace{0.7pt}\fx}\subseteq\bC$ for the set of roots of $b_{f,\hspace{0.7pt}\fx}(s)$. For $\beta \in\bC$, write $m_{f,\beta,\hspace{0.7pt}\fx}$ for the multiplicity of $\beta$ as a root of $b_{f,\hspace{0.7pt}\fx}(s)$.
    
\end{defn}

\begin{note}

The existence of the Bernstein-Sato polynomial in the algebraic setting was first proven by Bernstein (\cite{B73}, Theorem 1'). Its existence in the local analytic setting is due to Kashiwara (\cite{Kash76}, Theorem 3.3). 
    
\end{note}

\begin{rem}[\cite{Kash76}, Section 5 and \cite{MSai94}, Theorem 0.4]

Some simple facts about $b_{f,\hspace{0.7pt}\fx}(s)$ are as follows:

\begin{enumerate}[label=\roman*)]

\item $\rho_{f,\hspace{0.7pt}\fx} \subseteq \bQ\cap(-n,0)$.

\item $-1 \in \rho_{f,\hspace{0.7pt}\fx}$.

\item $b_{f,\hspace{0.7pt}\fx}(s)=s+1$ if and only if $f$ is smooth at $\fx$.
    
\end{enumerate}
    
\end{rem}

\noindent We've already seen that the set $\rho_{f,\hspace{0.7pt}\fx}$ determines (at least locally) when the mixed Hodge module $M(f^{-\alpha})$ is pure (Proposition \ref{lemDmodprops}.iv)). We shall see in this section that there is a more precise relation between the multiplicities of roots of $b_{f,\hspace{0.7pt}\fx}(s)$ and the number of possible weights of the mixed Hodge module $M(f^{-\alpha})$.

This relation is due to the following close relation between the Kashiwara-Malgrange $V$-filtration associated to $\scO_X(*f)$ and the Bernstein-Sato polynomial of $f$.

\begin{lem}

Define a filtration (for $k \in\bZ$)
\[G_ki_{f,+}\scO_X(*f):=t^{-k}\cdot(\scD_X[s]\cdot 1),\]
where $s$ is defined to act as $-\partial_tt$ (throughout this section). Then $b_{f,\hspace{0.7pt}\fx}(s)$ coincides with the minimal monic polynomial of the action of $s$ on $\text{\emph{gr}}^G_0i_{f,+}\scO_X(*f)$ at the point $\fx$.

As a consequence, for $\alpha \in\bQ$ and $k \in\bZ$,
\[m_{f,-\alpha-k,\hspace{0.7pt}\fx} = \text{\emph{min}}\{i \in\bZ_{\geq 0}\,|\, (s+\alpha)^i\cdot\text{\emph{gr}}^G_k\text{\emph{gr}}^{\alpha}_Vi_{f,+}\scO_X(*f)=0\,\text{ \emph{at} }\, \fx\}.\]

In particular,
\[\alpha_{f,\hspace{0.7pt}\fx}:=\text{\emph{min}}\,(-\rho_{f,\hspace{0.7pt}\fx}) = \text{\emph{min}}\{\gamma\in\bQ\,|\,\text{\emph{gr}}^G_0\text{\emph{gr}}_V^{\gamma}i_{f,+}\scO_X(*f) \neq 0\,\text{ \emph{at} }\, \fx\}.\]

\label{lemBSmult}
    
\end{lem}

\begin{proof}

See \cite{MSai17}, Section 1.2.
    
\end{proof}

\begin{defn}

Write $\widetilde{b}_{f,\hspace{0.7pt}\fx}(s):=b_{f,\hspace{0.7pt}\fx}(s)/(s+1) \in \bC[s]$ for the \emph{reduced Bernstein-Sato polynomial of $f$ at $\fx$} (also called the \emph{microlocal Bernstein-Sato polynomial of $f$ at $\fx$}, the reason for which we shall see later). Write $\widetilde{\rho}_{f,\hspace{0.7pt}\fx}\subseteq\bC$ for the set of roots of $\widetilde{b}_{f,\hspace{0.7pt}\fx}(s)$. For $\beta \in\bC$, write $\widetilde{m}_{f,\beta,\hspace{0.7pt}\fx}$ for the multiplicity of $\beta$ as a root of $\widetilde{b}_{f,\hspace{0.7pt}\fx}(s)$.

The \emph{minimal exponent of $f$ at $\fx$} is the smallest rational number $\widetilde{\alpha}_{f,\hspace{0.7pt}\fx}\in\bQ_{>0}$ such that
\[\widetilde{b}_{f,\hspace{0.7pt}\fx}(-\widetilde{\alpha}_{f,\hspace{0.7pt}\fx})=0.\]

\end{defn}

\begin{defn}

We recursively define a collection of polynomials $b_{f,\hspace{0.7pt}\fx}^{(l)}(s)$, for $l \in\bZ_{\geq 0}$, as follows:
\[b_{f,\hspace{0.7pt}\fx}^{(0)}(s):=\widetilde{b}_{f,\hspace{0.7pt}\fx}(s), \hspace{10pt} b_{f,\hspace{0.7pt}\fx}^{(l+1)}(s) := b_{f,\hspace{0.7pt}\fx}^{(l)}(s)/\sqrt{b_{f,\hspace{0.7pt}\fx}^{(l)}(s)} \,\,\text{ for } l \geq 0.\]
Then the \emph{$l$-th weighted minimal exponent of $f$ at $\fx$} (these will collectively be referred to as the \emph{weighted minimal exponents of $f$ at $\fx$}) is the smallest rational number $\widetilde{\alpha}_{f,\hspace{0.7pt}\fx}^{(l)}\in\bQ_{>0}$ such that
\[b_{f,\hspace{0.7pt}\fx}^{(l)}(-\widetilde{\alpha}_{f,\hspace{0.7pt}\fx}^{(l)})=0.\]

\label{defnwminlexp}
    
\end{defn}

\begin{rem}{\blank}

\begin{enumerate}[label=\roman*)]

\item $\widetilde{\alpha}_{f,\hspace{0.7pt}\fx}^{(0)} = \widetilde{\alpha}_{f,\hspace{0.7pt}\fx}$ and $\widetilde{\alpha}_{f,\hspace{0.7pt}\fx}^{(l)} \leq \widetilde{\alpha}_{f,\hspace{0.7pt}\fx}^{(l+1)}$ for all $l \geq 0$. 

\item Note that $\widetilde{\alpha}^{(l)}_{f,\hspace{0.7pt}\fx}$ is only defined when $b_{f,\hspace{0.7pt}\fx}^{(l)}(s) \neq 1$. In particular the minimal exponent of $f$ at $\fx$ is only defined when $f$ is singular at $\fx$. Throughout, whenever we write $\widetilde{\alpha}^{(l)}_{f,\hspace{0.7pt}\fx}$, we make the implicit assumption that $b_{f,\hspace{0.7pt}\fx}^{(l)}(s) \neq 1$.

\item Note also that $-\widetilde{\alpha}_{f,\hspace{0.7pt}\fx}^{(l)}$ may alternatively be defined to be the largest root of $\widetilde{b}_{f,\hspace{0.7pt}\fx}(s)$ whose multiplicity is greater than or equal to $l+1$, i.e.
\[-\widetilde{\alpha}_{f,\hspace{0.7pt}\fx}^{(l)}=\text{max}\{\beta \in \bQ \,|\, \widetilde{m}_{f,\beta,\hspace{0.7pt}\fx} \geq l+1\}.\]

\item Using \cite{MSai94}, Theorem 0.5, we have the bound
\[\widetilde{\alpha}_{f,\hspace{0.7pt}\fx}^{(l)} \leq n-\widetilde{\alpha}_{f,\hspace{0.7pt}\fx}-l.\footnotemark\]
Thus, in conclusion, we have
\[\widetilde{\alpha}_{f,\hspace{0.7pt}\fx}^{(l)}\in \bQ\cap \left[\widetilde{\alpha}_{f,\hspace{0.7pt}\fx},n-\widetilde{\alpha}_{f,\hspace{0.7pt}\fx}-l\right].\]

\footnotetext{Indeed, we have the same bound for any root of $\widetilde{b}_{f,\hspace{0.7pt}\fx}(s)$ of multiplicity $\geq l+1$. In particular, all roots of $\widetilde{b}_{f,\hspace{0.7pt}\fx}(s)$ have multiplicity $\leq n-1$.}

\end{enumerate}
    
\end{rem}

\noindent Now we introduce microlocalisation and relate the above polynomials and invariants to natural filtrations and operators on this module, in a similar vein to Lemma \ref{lemBSmult}. We imitate the notation of \cite{MSai94} and \cite{MSai17}, writing
\[\scB_f:=i_{f,+}\scO_X.\]

\begin{defn}

The \emph{algebraic partial microlocalisation} of $\scB_f$ is the $\scD_{X\times\bC}$-module 
\[\widetilde{\scB}_f:=(i_{f,*}\scO_X)[\partial_t,\partial_t^{-1}].\]
Its action is defined analogously to that of $(i_{f,*}\scO_X)[\partial_t]$, in particular so that (in local coordinates)
\[\partial_{x_i}\cdot (g \partial_t^k) = \partial_{x_i}(g) \partial_t^k - \partial_{x_i}(f)g\partial_t^{k+1}, \,\,\,\,\,\,\,t\cdot (g\partial_t^k)=fg\partial_t^k-kg\partial_t^{k-1}. \]
We define also the filtrations
\[F_k^{t-\text{ord}}\widetilde{\scB}_f := \sum_{i \leq k}(i_{f,*}\scO_X)\partial_t^i, \,\,\,\,\,\,G_k\widetilde{\scB}_f:=\scD_X[s,\partial_t^{-1}]\cdot\partial_t^k.\]
(Again, $s$ is defined to act as $-\partial_tt$ here.)

\label{defnmicrol}
    
\end{defn}

\begin{lem}[\cite{MSai94}, Proposition 0.3]

The reduced Bernstein-Sato polynomial $\widetilde{b}_{f,\hspace{0.7pt}\fx}(s)$ coincides with the minimal monic polynomial of the action of $s$ on $\text{\emph{gr}}_0^G\widetilde{\scB}_f$ at the point $\fx$. i.e. $\widetilde{b}_{f,\hspace{0.7pt}\fx}(s)$ is the minimal monic polynomial satisfying the functional equation
\[\widetilde{b}_{f,\hspace{0.7pt}\fx}(s)\cdot 1 \in \scD_{X,\hspace{1pt}\fx}[s,\partial_t^{-1}] \cdot \partial_t^{-1} \,\,\,\text{ in }\,\,\, \widetilde{\scB}_{f,\hspace{0.7pt}\fx}.\]

\label{lemSaitomicrobfcn}
    
\end{lem}

\begin{defn}[\cite{MSai94}, (2.1.3)]

The \emph{microlocal $V$-filtration} for $\widetilde{\scB}_f$ along the hypersurface $\{t=0\}$ is the decreasing rational (exhaustive) filtration on $\widetilde{\scB}_f$ defined, for $\gamma\in\bQ$, by
\[V^{\gamma}\widetilde{\scB}_f = \begin{cases}V^{\gamma}\scB_f + F_{-1}^{t-\text{ord}}\widetilde{\scB}_f & \text{ if } \gamma \leq 1\\ \partial_t^{-\lfloor\gamma\rfloor}\cdot V^{\gamma -\lfloor\gamma\rfloor}\widetilde{\scB}_f & \text{ otherwise.}\end{cases}\]
    
\end{defn}

\begin{rem}

Firstly, the inclusion $\scB_f\hookrightarrow\widetilde{\scB}_f$ induces an isomorphism of $\scD_X[s]$-modules (which in addition preserves the filtrations $F_{\bullet}^{t-\text{ord}}$)
\[\text{gr}_V^{\gamma}\scB_f \isommap \text{gr}_V^{\gamma}\widetilde{\scB}_f \]
whenever $\gamma<1$.

Secondly, under the bijection $\partial_t : \widetilde{\scB}_f \isommap\widetilde{\scB}_f$, $V^{\gamma}\widetilde{\scB}_f$ is mapped bijectively to $V^{\gamma-1}\widetilde{\scB}_f$ for all $\gamma\in\bQ$ (\cite{MSai94}, Lemma 2.2), i.e.
\[\partial_t \cdot V^{\gamma}\widetilde{\scB}_f =V^{\gamma-1}\widetilde{\scB}_f \,\,\,\,\,\forall \gamma\in\bQ.\]
    
\end{rem}

\noindent Analogously to Lemma \ref{lemBSmult}, Lemma \ref{lemSaitomicrobfcn} allows us to obtain expressions for the roots and multiplicities of the reduced Bernstein-Sato polynomial $\widetilde{b}_{f,\hspace{0.7pt}\fx}(s)$ in terms of the filtrations $V^{\bullet}$ and $G_{\bullet}$ defined on $\widetilde{\scB}_f$.

\begin{lem}

For $\alpha \in\bQ$ and $k \in\bZ$,
\[\widetilde{m}_{f,-\alpha-k,\hspace{0.7pt}\fx} = \text{\emph{min}}\{i \in\bZ_{\geq 0}\,|\, (s+\alpha)^i\cdot\text{\emph{gr}}^G_k\text{\emph{gr}}^{\alpha}_V\widetilde{\scB}_f=0\,\text{ \emph{at} }\, \fx\}.\]

\noindent In particular,
\[\widetilde{\alpha}_{f,\hspace{0.7pt}\fx}^{(l)}=\text{\emph{min}}\{\gamma\in\bQ\,|\,(s+\gamma)^l \cdot \text{\emph{gr}}^G_0\text{\emph{gr}}_V^{\gamma}\widetilde{\scB}_f \neq 0\,\text{ \emph{at} }\, \fx\}.\]

\label{lemredBSmult}
    
\end{lem}

\begin{proof}

This follows by \cite{MSai17}, Section 1.3.
    
\end{proof}

\noindent In order to compare the microlocal $V$-filtration to the Hodge filtration on $\scM(f^{-\alpha})$, we first extend the kernel filtration to the microlocal $V$-filtration in the obvious way.

\begin{defn} 

For $\gamma \in \bQ$, we define the \emph{kernel filtration} on $V^{\gamma}\widetilde{\scB}_f$ as follows. For $l \in \bZ_{\geq 0}$,
\[K_lV^{\gamma}\widetilde{\scB}_f := \{u \in V^{\gamma}\widetilde{\scB}_f \,|\, (s+\gamma)^l \cdot u \in V^{>\gamma}\widetilde{\scB}_f\}.\]
    
\end{defn}

\noindent It is important to note that under this definition we have, for $\gamma\in\bQ$ and $l\in\bZ_{\geq 0}$, that
\[\partial_t\cdot K_lV^{\gamma}\widetilde{\scB}_f = K_lV^{\gamma-1}\widetilde{\scB}_f.\]

\begin{lem}

We have the following expression for the kernel filtration on $V^{\gamma}\widetilde{\scB}_f$:
\[K_lV^{\gamma}\widetilde{\scB}_f = \begin{cases}K_lV^{\gamma}\scB_f + F_{-1}^{t-\text{\emph{ord}}}\widetilde{\scB}_f & \text{ if } \gamma < 1\\ K_{l+1}V^1\scB_f + F_{-1}^{t-\text{\emph{ord}}}\widetilde{\scB}_f & \text{ if } \gamma=1\\ \partial_t^{-\lfloor\gamma\rfloor}\cdot K_lV^{\gamma -\lfloor\gamma\rfloor}\widetilde{\scB}_f & \text{ otherwise.}\end{cases}\]

\label{lemkerVmicro}
    
\end{lem}

\begin{proof}

It suffices to consider $\gamma \leq 1$. We begin by proving the reverse inclusion $\supseteq$.

Firstly, see that Lemma \ref{lemredBSmult} implies that $1 \in V^{\widetilde{\alpha}_{f,\hspace{0.7pt}\fx}}\widetilde{\scB}_f$ (at $\fx$). Thus in particular $1 \in V^{>0}\widetilde{\scB}_f$, so $\partial_t^{-1} \in V^{>1}\widetilde{\scB}_f$. This then also implies that $\partial_t^{-k}\in V^{>1}\widetilde{\scB}_f$ for all $k\geq 1$. Therefore 
\[F_{-1}^{t-\text{ord}}\widetilde{\scB}_f \subseteq K_lV^{\gamma}\widetilde{\scB}_f\]
for all $\gamma \leq 1$.

Now if $\gamma<1$ and $u \in K_lV^{\gamma}\scB_f$, then 
\[(s+\gamma)^l\cdot u \in V^{>\gamma}\scB_f \subseteq V^{>\gamma}\widetilde{\scB}_f,\]
implying that $u \in K_lV^{\gamma}\widetilde{\scB}_f$ as required. 

If instead $u \in K_{l+1}V^1\scB_f$, then 
\[(t\partial_t)^{l+1}\cdot u =(-1)^{l+1}(s+1)^{l+1}\cdot u\in V^{>1}\scB_f \,\,\Rightarrow \,\,\partial_t(t\partial_t)^l\cdot u \in V^{>0}\scB_f\subseteq V^{>0}\widetilde{\scB}_f,\]
so that 
\[(t\partial_t)^l\cdot u\in V^{>1}\widetilde{\scB}_f\]
as required.

Now we prove the reverse inclusion. Let $u \in K_lV^{\gamma}\widetilde{\scB}_f$, $\gamma \leq 1$. As seen above, $F_{-1}^{t-\text{ord}}\widetilde{\scB}_f \subseteq K_lV^{\gamma}\widetilde{\scB}_f$, so without loss of generality we may assume that $u \in V^{\gamma}\scB_f$. Then
\[(s+\gamma)^l\cdot u \in V^{>\gamma}\widetilde{\scB}_f\cap\scB_f.\]
If $\gamma <1$, then this implies that $(s+\gamma)^l\cdot u \in V^{\gamma}\scB_f$, so $u \in K_lV^{\gamma}\scB_f$ as required. 

If $\gamma=1$, then 
\[(s+\gamma)^{l+1}\cdot u \in t\partial_t\cdot \left(V^{>1}\widetilde{\scB}_f\cap\scB_f\right)=t\cdot\left( V^{>0}\widetilde{\scB}_f\cap\scB_f\right) = t\cdot V^{>0}\scB_f = V^{>1}\scB_f,\]
so $u \in K_{l+1}V^1\scB_f$ as required.
    
\end{proof}

\noindent Another set of useful microlocal invariants we now introduce are the so-called \emph{microlocal multiplier ideals} and their weighted counterparts, defined in terms of the above kernel filtrations. As we shall see, there is a relation between the weighted microlocal multiplier ideals and the weighted minimal exponents.

\begin{defn}

Consider the map
\[\psi_0 : \widetilde{\scB}_f \to \scO_X \,;\, \sum_{i\in\bZ}u_i\partial_t^i \mapsto u_0.\]
The \emph{microlocal multiplier ideals associated to $f$} (c.f. \cite{MSai17}, (1.5.4)) are the ideals, for $\gamma \in \bQ$, defined by
\[\widetilde{V}^{\gamma}\scO_X := \psi_0(F_0^{t-\text{ord}}V^{\gamma}\widetilde{\scB}_f).\]
Similarly, the \emph{weighted microlocal multiplier ideals} associated to $f$ are the ideals, for $\gamma \in \bQ$ and $l \in \bZ_{\geq 0}$, defined by
\[W_l\widetilde{V}^{\gamma}\scO_X := \psi_0(F_0^{t-\text{ord}}K_lV^{\gamma}\widetilde{\scB}_f).\]

\label{defnmicromultideals}
    
\end{defn}

\begin{rem}

By definition, $\widetilde{V}^{>\gamma}\scO_X=W_0\widetilde{V}^{\gamma}\scO_X\subseteq W_l\widetilde{V}^{\gamma}\scO_X\subseteq \widetilde{V}^{\gamma}\scO_X$ for all $\gamma$ and $l$, where $\widetilde{V}^{>\gamma}\scO_X:=\bigcup_{\gamma'>\gamma}\widetilde{V}^{\gamma'}\scO_X=\widetilde{V}^{\gamma + \epsilon}\scO_X$ for $0<\epsilon<<1$. 

Moreover, $W_l\widetilde{V}^{\gamma}\scO_X=\scO_X$ for all $l$ and $\gamma \leq 0$, since as we've already seen $1 \in V^{>0}\widetilde{\scB}_f$.
    
\end{rem}

\begin{rem}

The microlocal multiplier ideals associated to $f$ coincide with the \emph{higher multiplier ideals} defined in \cite{SY24} (see Section 5.4). In the notation of this paper, 
\[\widetilde{V}^{\gamma}\scO_X=\scI_{\lceil\gamma\rceil-1,\gamma-\lceil\gamma\rceil+1}(f).\]
Moreover, a filtration $W_{\bullet}$ is defined on the higher multiplier ideals in \cite{SY24} (Definition 5.7), and we have in addition that
\[W_l\widetilde{V}^{\gamma}\scO_X = W_l\scI_{\lceil\gamma\rceil-1,\gamma-\lceil\gamma\rceil+1}(f).\]
Namely, we have the following expressions for $\widetilde{V}^{\gamma}\scO_X$ and $W_l\widetilde{V}^{\gamma}\scO_X$, when $\gamma>0$.
    
\end{rem}

\begin{lem}

For $k \in\bZ_{\geq 0}$, $\alpha \in (0,1]$ and $l\in\bZ_{\geq 0}$,
\[W_l\widetilde{V}^{k+\alpha}\scO_X=\{v \in \scO_X\,|\, \text{there exists } \sum_{i=0}^kv_i\partial_t^i \in K_{l+\lfloor\alpha\rfloor}V^{\alpha}\scB_f\,\text{ such that }\, v_k=v\}.\]

\label{lemmicromultformulae}
    
\end{lem}

\begin{proof}

This follows almost immediately by Lemma \ref{lemkerVmicro} and from the property of microlocalisation
\[\partial_t\cdot F_k^{t-\text{ord}}K_lV^{\gamma}\widetilde{\scB}_f=F_{k+1}^{t-\text{ord}}K_lV^{\gamma-1}\widetilde{\scB}_f.\]

Namely, if $v\in W_l\widetilde{V}^{k+\alpha}\scO_X$, then there exists some $u \in F_0^{t-\text{ord}}K_lV^{k+\alpha}\widetilde{\scB}_f$ such that $\psi_0(u)=v$. Then
\[\partial_t^k\cdot u \in F_k^{t-\text{ord}}K_lV^{\alpha}\widetilde{\scB}_f = K_{l+\lfloor\alpha\rfloor}V^{\alpha}\scB_f+F_{-1}^{t-\text{ord}}\widetilde{\scB}_f,\]
so $v$ lies in the right hand side of the expression given in the statement of the lemma.

If, conversely, $v\in\scO_X$ lies in the right hand side of the expression given in the statement, then there exists some $\sum_{i=0}^kv_i\partial_t^i\in K_{l+\lfloor\alpha\rfloor}V^{\alpha}\scB_f\subseteq K_lV^{\alpha}\widetilde{\scB}_f$ such that $v_k=v$, and thus
\[v = \psi_0\left(\partial_t^{-k}\cdot \sum_{i=0}^kv_i\partial_t^i\right)\in \psi_0\left(F_0^{t-\text{ord}}K_lV^{k+\alpha}\widetilde{\scB}_f\right)=W_l\widetilde{V}^{k+\alpha}\scO_X.\]
    
\end{proof}

\begin{cor}

$W_l\widetilde{V}^{k+\alpha}\scO_X=\scO_X$ if and only if $\partial_t^k \in K_{l+\lfloor\alpha\rfloor}V^{\alpha}\scB_f$.

\label{corpartial}
    
\end{cor}

\begin{proof}

If $\partial_t^k \in K_{l+\lfloor\alpha\rfloor}V^{\alpha}\scB_f$, then by the expression of the previous lemma it is clear that $1\in W_l\widetilde{V}^{k+\alpha}\scO_X$, which of course implies that $W_l\widetilde{V}^{k+\alpha}\scO_X=\scO_X$ as $W_l\widetilde{V}^{k+\alpha}\scO_X$ is an ideal of $\scO_X$.

The converse we prove via induction on $k$. For $k=0$ this is clear by the above lemma. For $k>0$, assume that $1 \in W_l\widetilde{V}^{k+\alpha}\scO_X$. Then, by the lemma, there exists some
\[v=\sum_{i=0}^kv_i\partial_t^i\in K_{l+\lfloor\alpha\rfloor}V^{\alpha}\scB_f\]
satisfying that $v_k=1$. Then, see that, for any $1\leq j\leq k$,
\[(t-f)^j\cdot v \in V^{\alpha+1}\scB_f + f^j\cdot K_{l+\lfloor\alpha\rfloor}V^{\alpha}\scB_f \subseteq K_{l+\lfloor\alpha\rfloor}V^{\alpha}\scB_f.\]
But
\[(t-f)^j\cdot v =(-1)^j\sum_{i=j}^k\frac{i!}{(i-j)!}v_i\partial_t^{i-j},\]
so we see in particular that $1\in W_l\widetilde{V}^{j+\alpha}\scO_X$ for all $j < k$. By the inductive hypothesis, this implies that $\partial_t^j\in K_{l+\lfloor\alpha\rfloor}V^{\alpha}\scB_f$ for all $j <k$. This finally implies that
\[\partial_t^k = v- \sum_{i=0}^{k-1}v_i\partial_t^i \in K_{l+\lfloor\alpha\rfloor}V^{\alpha}\scB_f.\]
    
\end{proof}

\noindent Comparing with the expression for the Hodge and weight filtrations obtained in Theorem \ref{thmmainformula}, Lemma \ref{lemmicromultformulae} gives the following (versions of this result having already been observed in \cite{MSai17}, Theorem 1, \cite{MP20}, Theorem A and \cite{SY24}, Corollary 5.25):

\begin{cor}

For $k \in\bZ_{\geq 0}$, $\alpha \in (0,1]$ and $l\in\bZ_{\geq 0}$,
\[(W_l\widetilde{V}^{k+\alpha}\scO_X)f^{-k-\alpha} +\scO_Xf^{-k+1-\alpha} = F_k^HW_{n+l+\lfloor\alpha\rfloor}\scM(f^{-\alpha}) + \scO_Xf^{-k+1-\alpha}\]
(as $\scO_X$-submodules of $\scM(f^{-\alpha})$).

\label{corHodgemicrocomp}
    
\end{cor}

\begin{rem}

In the case $\alpha=1$, there is actually a stronger relationship between the weighted microlocal multiplier ideals $W_k\widetilde{V}^{k+1}\scO_X$ and $F_k^HW_{n+l+1}\scO_X(*f)$ (see \cite{SY24}, Lemma 5.22). This recursive relationship may also be extended to the twisted setting, using for instance the weighted analogue of \cite{DY25}, Corollary 1.2.
    
\end{rem}

\noindent A final use of microlocalisation in this subsection is to obtain yet another equivalent definition for the weighted minimal exponents. This is an extension of \cite{MSai17}, (1.3.4) and (1.3.8), to the weighted setting.

\begin{lem}

Let $l\in\bZ_{\geq 0}$. Write
\[\widetilde{\scJ}_{f,\hspace{0.7pt}\fx}:=\{\gamma \in \bQ_{>0}\mid \text{\emph{gr}}_V^{\gamma}\widetilde{\scB}_f \neq 0 \text{ at }\fx\}\]
for the set of (positive) \emph{jumping numbers} of $V^{\bullet}\widetilde{\scB}_f$ at $\fx$. Then
\begin{align*}\widetilde{\alpha}_{f,\hspace{0.7pt}\fx}^{(l)}&=\text{\emph{max}}\{\gamma\in\bQ\,|\,\prod_{\gamma'\in\widetilde{\scJ}_{f,\hspace{0.7pt}\fx}, \,\gamma'<\gamma}(s+\gamma')^l\cdot 1 \in V^{\gamma}\widetilde{\scB}_{f,\hspace{0.7pt}\fx}+G_{-1}\widetilde{\scB}_{f,\hspace{0.7pt}\fx}\}\\ &\leq\text{\emph{min}}\{\gamma\in\bQ\,|\,\widetilde{V}^{\gamma}\scO_X \neq W_{l+\sum_{i=0}^{l-1}\lfloor\gamma-\widetilde{\alpha}_{f,\hspace{0.7pt}\fx}^{(i)}\rfloor}\widetilde{V}^{\gamma}\scO_X\text{ at }\fx\}.\end{align*}
In particular, 
\begin{align*}\widetilde{\alpha}_{f,\hspace{0.7pt}\fx}&=\text{\emph{max}}\{\gamma\in\bQ\,|\,1 \in \widetilde{V}^{\gamma}\scO_X\text{ at }\fx\}\\&=\text{\emph{min}}\{\gamma\in\bQ\,|\,\widetilde{V}^{\gamma}\scO_X\neq\widetilde{V}^{>\gamma}\scO_X\text{ at }\fx\}.\\\end{align*}

\label{lemmicrominlexp}
    
\end{lem}

\begin{proof}

Recall Lemma \ref{lemredBSmult}, which says that
\[\widetilde{\alpha}_{f,\hspace{0.7pt}\fx}^{(l)}=\text{min}\{\gamma\in\bQ\,|\,(s+\gamma)^l \cdot \text{gr}^G_0\text{gr}_V^{\gamma}\widetilde{\scB}_f \neq 0\,\text{ at }\, \fx\}.\]
We will prove (for any $\gamma \in \bQ$) that
\[\prod_{\gamma'\in\widetilde{\scJ}_{f,\hspace{0.7pt}\fx}, \,\gamma'<\gamma}(s+\gamma')^l\cdot 1 \in V^{\gamma}\widetilde{\scB}_{f,\hspace{0.7pt}\fx}+G_{-1}\widetilde{\scB}_{f,\hspace{0.7pt}\fx}\Longleftrightarrow \,\,(s+\gamma')^l \cdot \text{gr}^G_0\text{gr}_V^{\gamma'}\widetilde{\scB}_{f,\hspace{0.7pt}\fx} = 0\,\,\text{ for all }\,\, \gamma'< \gamma,\]
which will prove the first equality in the statement.

Note first, by Zassenhaus' lemma (see for instance \cite{Del71}, (1.2.1)), that we have a canonical isomorphism
\[\text{gr}^G_0\text{gr}_V^{\gamma}\widetilde{\scB}_f \simeq G_0V^{\gamma}\widetilde{\scB}_f/(G_0V^{>\gamma}\widetilde{\scB}_f+G_{-1}V^{\gamma}\widetilde{\scB}_f).\]

Assume firstly that $\prod_{\gamma'\in\widetilde{\scJ}_{f,\hspace{0.7pt}\fx}, \,\gamma'<\gamma}(s+\gamma')^l\cdot 1 \in V^{\gamma}\widetilde{\scB}_{f,\hspace{0.7pt}\fx}+G_{-1}\widetilde{\scB}_{f,\hspace{0.7pt}\fx}$. Then, since 
\[(s+\gamma''):\text{gr}^G_0\text{gr}_V^{\gamma'}\widetilde{\scB}_{f,\hspace{0.7pt}\fx} \isommap \text{gr}^G_0\text{gr}_V^{\gamma'}\widetilde{\scB}_{f,\hspace{0.7pt}\fx}\]
for all $\gamma'\neq\gamma''$, we have for all $\gamma'<\gamma$ that
\begin{align*}
G_0V^{\gamma'}\widetilde{\scB}_{f,\hspace{0.7pt}\fx} &\subseteq \prod_{\gamma''\in\widetilde{\scJ}_{f,\hspace{0.7pt}\fx}, \,\gamma''<\gamma, \gamma''\neq \gamma'}(s+\gamma'')^l\cdot G_0V^{\gamma'}\widetilde{\scB}_{f,\hspace{0.7pt}\fx}+G_0V^{>\gamma'}\widetilde{\scB}_{f,\hspace{0.7pt}\fx} +G_{-1}V^{\gamma'}\widetilde{\scB}_{f,\hspace{0.7pt}\fx}
\end{align*}
\noindent This then implies by assumption that 
\[(s+\gamma')^l\cdot G_0V^{\gamma'}\widetilde{\scB}_{f,\hspace{0.7pt}\fx} \subseteq G_0V^{>\gamma'}\widetilde{\scB}_{f,\hspace{0.7pt}\fx} + G_{-1}V^{\gamma'}\widetilde{\scB}_{f,\hspace{0.7pt}\fx},\]
which is equivalent by Zassenhaus' lemma to $(s+\gamma')^l\cdot \text{gr}^G_0\text{gr}_V^{\gamma'}\widetilde{\scB}_{f,\hspace{0.7pt}\fx}=0$, as required.

Assume conversely that $(s+\gamma')^l\cdot \text{gr}^G_0\text{gr}_V^{\gamma'}\widetilde{\scB}_{f,\hspace{0.7pt}\fx}=0$ for all $\gamma'<\gamma$.

We may without loss of generality assume that $\gamma\in\widetilde{\scJ}_{f,\hspace{0.7pt}\fx}$, and proceed via induction on $\gamma\in\widetilde{\scJ}_{f,\hspace{0.7pt}\fx}$. 

If $\gamma = \text{min}\widetilde{\scJ}_{f,\hspace{0.7pt}\fx}$, then as seen above $1 \in V^{\text{min}\widetilde{\scJ}_{f,\hspace{0.7pt}\fx}}\widetilde{\scB}_{f,\hspace{0.7pt}\fx}$.

Otherwise, we write $\gamma'$ for the jumping number such that $V^{>\gamma'}\widetilde{\scB}_{f,\hspace{0.7pt}\fx}=V^{\gamma}\widetilde{\scB}_{f,\hspace{0.7pt}\fx}$ and assume that $\prod_{\gamma''\in\widetilde{\scJ}_{f,\hspace{0.7pt}\fx}, \,\gamma''<\gamma'}(s+\gamma'')^l\cdot 1 \in V^{\gamma'}\widetilde{\scB}_{f,\hspace{0.7pt}\fx}+G_{-1}\widetilde{\scB}_{f,\hspace{0.7pt}\fx}$. Write 
\[\prod_{\gamma''\in\widetilde{\scJ}_{f,\hspace{0.7pt}\fx}, \,\gamma''<\gamma'}(s+\gamma'')^l\cdot 1=v_1+v_2\]
with $v_1\in G_0V^{\gamma'}\widetilde{\scB}_{f,\hspace{0.7pt}\fx}$ and $v_2\in G_{-1}\widetilde{\scB}_{f,\hspace{0.7pt}\fx}$. Then, by hypothesis (and using Zassenhaus' lemma), $(s+\gamma')^l\cdot v_1 \in G_0V^{>\gamma'}\widetilde{\scB}_{f,\hspace{0.7pt}\fx}+G_{-1}\widetilde{\scB}_{f,\hspace{0.7pt}\fx}$, implying that
\[\prod_{\gamma''\in\widetilde{\scJ}_{f,\hspace{0.7pt}\fx}, \,\gamma''<\gamma}(s+\gamma'')^l\cdot 1 \in V^{\gamma}\widetilde{\scB}_{f,\hspace{0.7pt}\fx}+G_{-1}\widetilde{\scB}_{f,\hspace{0.7pt}\fx}\]
as required.

The final statement of the lemma is proven in \cite{MSai17}, and follows from the first equality in the lemma using the same argument as (1.3.4)-(1.3.6) in this paper.

Now we prove the inequality appearing in the statement of the lemma. For this it suffices to show that
\[(s+\gamma')^l \cdot \text{gr}^G_0\text{gr}_V^{\gamma'}\widetilde{\scB}_{f,\hspace{0.7pt}\fx} = 0\,\,\,\forall \gamma'< \gamma \Longrightarrow (s+\gamma')^{l+\sum_{i=0}^{l-1}\lfloor\gamma'- \widetilde{\alpha}_{f,\hspace{0.7pt}\fx}^{(i)}\rfloor}\cdot G_0V^{\gamma'}\widetilde{\scB}_{f,\hspace{0.7pt}\fx} \subseteq V^{>\gamma'}\widetilde{\scB}_{f,\hspace{0.7pt}\fx}\,\,\,\forall\gamma'<\gamma,\]
since the latter statement implies in particular that
\[F_0^{t-\text{ord}}V^{\gamma'}\widetilde{\scB}_{f,\hspace{0.7pt}\fx} = F_0^{t-\text{ord}}K_{l+\sum_{i=0}^{l-1}\lfloor\gamma'- \widetilde{\alpha}_{f,\hspace{0.7pt}\fx}^{(i)}\rfloor}V^{\gamma'}\widetilde{\scB}_{f,\hspace{0.7pt}\fx}\,\,\text{ for all }\,\gamma'<\gamma,\]
so that
\[\widetilde{V}^{\gamma'}\scO_{X,\hspace{0.7pt}\fx}=W_{l+\sum_{i=0}^{l-1}\lfloor\gamma'- \widetilde{\alpha}_{f,\hspace{0.7pt}\fx}^{(i)}\rfloor}\widetilde{V}^{\gamma'}\scO_{X,\hspace{0.7pt}\fx}\,\,\text{ for all }\,\gamma'<\gamma.\]
We prove this via induction. The case $l=0$ follows from the final statement in the lemma. 

Assume that $l>0$. Our assumption is equivalent to
\[(s+\gamma')^l\cdot G_0V^{\gamma'}\widetilde{\scB}_{f,\hspace{0.7pt}\fx} \subseteq G_0V^{>\gamma'}\widetilde{\scB}_{f,\hspace{0.7pt}\fx}+G_{-1}V^{\gamma'}\widetilde{\scB}_{f,\hspace{0.7pt}\fx}\,\,\text{ for all }\,\gamma'<\gamma.\]
However, $G_{-1}V^{\gamma'}\widetilde{\scB}_{f,\hspace{0.7pt}\fx}=\partial_t^{-1}\cdot G_0V^{\gamma'-1}\widetilde{\scB}_{f,\hspace{0.7pt}\fx}$, so we may iterate to in particular obtain the statement
\[(s+\gamma')^{l(k_{l-1}+1)}\cdot G_0V^{\gamma'}\widetilde{\scB}_{f,\hspace{0.7pt}\fx} \subseteq G_0V^{>\gamma'}\widetilde{\scB}_{f,\hspace{0.7pt}\fx}+G_{-k_{l-1}-1}V^{\gamma'}\widetilde{\scB}_{f,\hspace{0.7pt}\fx},\]
where we write $k_i:=\lfloor\gamma'-\widetilde{\alpha}_{f,\hspace{0.7pt}\fx}^{(i)}\rfloor$.

Now, $\gamma'-k_{l-1}-1<\widetilde{\alpha}_{f,\hspace{0.7pt}\fx}^{(l-1)}$, so by our inductive hypothesis we have that
\[(s+\gamma'-k_{l-1}-1)^{l-1+\sum_{i=0}^{l-2}(k_i-k_{l-1}-1)}\cdot G_0V^{\gamma'-k_{l-1}-1}\widetilde{\scB}_{f,\hspace{0.7pt}\fx} \subseteq V^{>\gamma'-k_{l-1}-1}\widetilde{\scB}_{f,\hspace{0.7pt}\fx}.\]
Combining, we thus have
\[(s+\gamma')^{l+\sum_{i=0}^{l-1}k_i}\cdot G_0V^{\gamma'}\widetilde{\scB}_{f,\hspace{0.7pt}\fx} \subseteq V^{>\gamma'}\widetilde{\scB}_{f,\hspace{0.7pt}\fx}\]
as required.

\end{proof}

\vspace{10pt}

\subsection{Generating level and highest weights for $M(f^{-\alpha})$}

\; \vspace{5pt} \\ In this subsection we use the technology we introduced in the previous subsection to learn about some of the most basic invariants attached to the bi-filtered $\scD_X$-module $(\scM(f^{-\alpha}),F_{\bullet}^H,W_{\bullet})$. Namely, the Hodge filtration $F_{\bullet}^H$ is a good filtration so has some finite \emph{generating level}, while the weight filtration $W_{\bullet}$ is a finite filtration so we can ask how many steps this filtration has.

Consider firstly the following consequence of the expression for the minimal exponent given in Lemma \ref{lemmicrominlexp} (as well as Corollary \ref{corHodgemicrocomp}), as observed in \cite{MP20}, Corollary C.

\begin{prop}

Let $k\in\bZ_{\geq 0}$ and $\alpha \in (0,1]$. Then
\[F_k^H\scM(f^{-\alpha})=\scO_Xf^{-k-\alpha} \text{ at }\fx\]
if and only if 
\[k+\alpha \leq \widetilde{\alpha}_{f,\hspace{0.7pt}\fx}.\]

\label{prophodgepole1}
    
\end{prop}

We obtain an analogous result concerning the weight filtration steps, as an immediate consequence of Corollary \ref{corHodgemicrocomp}. This is an extension of \cite{Olano23}, Corollary 5.10, which covers the case $\alpha =1$, $l=0$.

\begin{prop}

For $k \in \bZ_{\geq 0}$, $\alpha \in (0,1]$ and $l \in \bZ_{\geq 0}$, we have that
\[F_k^HW_{n+l+\lfloor\alpha\rfloor}\scM(f^{-\alpha})=\scO_Xf^{-k-\alpha} \text{ at }\fx\]
if and only if 
\[\text{either } k+\alpha < \widetilde{\alpha}_{f,\hspace{0.7pt}\fx} \text{ or } k+\alpha =\widetilde{\alpha}_{f,\hspace{0.7pt}\fx} \text{ and }\widetilde{\alpha}_{f,\hspace{0.7pt}\fx} \neq \widetilde{\alpha}_{f,\hspace{0.7pt}\fx}^{(l)}\footnotemark.\]
(i.e. if and only if $k \leq \text{\emph{min}}\{\lfloor\widetilde{\alpha}_{f,\hspace{0.7pt}\fx}-\alpha\rfloor,\lceil\widetilde{\alpha}_{f,\hspace{0.7pt}\fx}^{(l)}-\alpha\rceil-1\}$.)

\footnotetext{Taken to hold superfluously if $\widetilde{\alpha}_{f,\hspace{0.7pt}\fx}^{(l)}$ isn't defined.}

\label{prophodgepole2}
    
\end{prop}

\begin{proof}

By Corollary \ref{corHodgemicrocomp}, 
\[F_k^HW_{n+l+\lfloor\alpha\rfloor}\scM(f^{-\alpha})=\scO_Xf^{-k-\alpha} \text{ at }\fx\,\,\Longleftrightarrow W_l\widetilde{V}^{k+\alpha}\scO_{X,\hspace{0.7pt}\fx}=\scO_{X,\hspace{0.7pt}\fx}.\]
The condition on the right hand side is equivalent to
\[1\in K_lV^{k+\alpha}\widetilde{\scB}_{f,\hspace{0.7pt}\fx}+F_{-1}^{t-\text{ord}}\widetilde{\scB}_{f,\hspace{0.7pt}\fx}.\]
Using recursion, this is equivalent to 
\[1\in K_lV^{k+\alpha}\widetilde{\scB}_{f,\hspace{0.7pt}\fx}+F_{-k'}^{t-\text{ord}}\widetilde{\scB}_{f,\hspace{0.7pt}\fx}\]
for all $k' \geq 1$, which is equivalent to
\[1\in K_lV^{k+\alpha}\widetilde{\scB}_{f,\hspace{0.7pt}\fx},\]
since $K_lV^{k+\alpha}\widetilde{\scB}_{f,\hspace{0.7pt}\fx}$ contains $F_{-k'}^{t-\text{ord}}\widetilde{\scB}_{f,\hspace{0.7pt}\fx}$ for $k'$ sufficiently large. By Proposition \ref{prophodgepole1}, it then suffices to show that
\[\widetilde{\alpha}_{f,\hspace{0.7pt}\fx}\neq \widetilde{\alpha}_{f,\hspace{0.7pt}\fx}^{(l)} \,\,\Leftrightarrow \,\, 1\in K_lV^{\widetilde{\alpha}_{f,\hspace{0.7pt}\fx}}\widetilde{\scB}_{f,\hspace{0.7pt}\fx}.\]
By Lemma \ref{lemmicrominlexp}, $\widetilde{\alpha}_{f,\hspace{0.7pt}\fx}\neq \widetilde{\alpha}_{f,\hspace{0.7pt}\fx}^{(l)} \Leftrightarrow (s+\widetilde{\alpha}_{f,\hspace{0.7pt}\fx})^l\cdot 1\in V^{>\widetilde{\alpha}_{f,\hspace{0.7pt}\fx}}\widetilde{\scB}_{f,\hspace{0.7pt}\fx}+G_{-1}V^{\widetilde{\alpha}_{f,\hspace{0.7pt}\fx}}\widetilde{\scB}_{f,\hspace{0.7pt}\fx}$. Since (by Lemma \ref{lemmicrominlexp}) $G_{-1}V^{\widetilde{\alpha}_{f,\hspace{0.7pt}\fx}}\widetilde{\scB}_{f,\hspace{0.7pt}\fx}=G_{-1}\widetilde{\scB}_{f,\hspace{0.7pt}\fx}=G_{-1}V^{>\widetilde{\alpha}_{f,\hspace{0.7pt}\fx}}\widetilde{\scB}_{f,\hspace{0.7pt}\fx}$, this is equivalent to the statement $(s+\widetilde{\alpha}_{f,\hspace{0.7pt}\fx})^l\cdot 1\in V^{>\widetilde{\alpha}_{f,\hspace{0.7pt}\fx}}\widetilde{\scB}_{f,\hspace{0.7pt}\fx}$, which is equivalent to
\[1\in K_lV^{\widetilde{\alpha}_{f,\hspace{0.7pt}\fx}}\widetilde{\scB}_{f,\hspace{0.7pt}\fx}.\]
    
\end{proof}

\begin{eg}

Given $f$ arbitrary, write $\widetilde{\alpha}_{f,\hspace{0.7pt}\fx}=r +\alpha$ with $r \in\bZ_{\geq 0}$ and $\alpha\in(0,1]$, and write $m$ for the multiplicity of $-\widetilde{\alpha}_{f,\hspace{0.7pt}\fx}$ as a root of $\widetilde{b}_{f,\hspace{0.7pt}\fx}(s)$. Then Proposition \ref{prophodgepole1} tells us (as always, locally at $\fx$) that 
\[F_k^H\scM(f^{-\alpha})=\scO_Xf^{-k-\alpha} \text{ for }k \leq r,\]
and Proposition \ref{prophodgepole2} in the case $l=0$ gives us that
\[F_k^HW_{n+\lfloor\alpha\rfloor}\scM(f^{-\alpha})=\scO_Xf^{-k-\alpha} \text{ for }k < r.\]
Then the multiplicity of $\widetilde{\alpha}_{f,\hspace{0.7pt}\fx}$ tells us exactly the largest index of $F_r^HW_{n+\bullet+\lfloor\alpha\rfloor}\scM(f^{-\alpha})$ we need consider, i.e.
\[F_r^HW_{n+l+\lfloor\alpha\rfloor}\scM(f^{-\alpha})=\scO_Xf^{-r-\alpha} \Leftrightarrow l \geq m.\]
We'll see momentarily that a stronger relation between highest weights and multiplicities of roots may be inferred in general.
    
\end{eg}

\begin{cor}

Write $D$ for the divisor associated to $f$.

\begin{enumerate}[label=\roman*)]

\item The pair $(X,\alpha D)$ is log canonical if and only if $\alpha \leq \widetilde{\alpha}_{f,\hspace{0.7pt}\fx}$.

\item The pair $(X,\alpha D)$ is purely log terminal if and only if $\alpha < \widetilde{\alpha}_{f,\hspace{0.7pt}\fx}$ or $\alpha=\widetilde{\alpha}_{f,\hspace{0.7pt}\fx}$ and $\alpha \neq 1$, $\widetilde{\alpha}_{f,\hspace{0.7pt}\fx}\neq \widetilde{\alpha}_{f,\hspace{0.7pt}\fx}^{(1)}$.

\item The pair $(X,\alpha D)$ is Kawamata log terminal if and only if $\alpha \neq 1$ and $\alpha < \widetilde{\alpha}_{f,\hspace{0.7pt}\fx}$.

\end{enumerate} 

\label{corlct}

\end{cor}

\begin{proof}

This now follows immediately by Corollary \ref{corpair}. Note that for an arbitrary $\alpha \in \bQ_{>0}$, if the pair $(X,\alpha D)$ is log canonical or purely log terminal then necessarily $\alpha \leq 1$, and if the pair is Kawamata log terminal then $\alpha < 1$. Thus we may indeed restrict to the case $\alpha \in(0,1]$.
    
\end{proof}

We now obtain upper and lower bounds for the highest weight of $\scM(f^{-\alpha})$ (i.e. the largest number $l$ such that $\text{gr}_l^W\scM(f^{-\alpha})\neq 0$) in terms of the multiplicities of roots of the Bernstein-Sato polynomial associated to $f$. First we prove the following useful lemma.

\begin{lem}

Let $\alpha \in (0,1]$ and $l,k\in\bZ_{\geq 0}$. Then
\[F_k^HW_{n+l+\lfloor\alpha\rfloor}\scM(f^{-\alpha})=F_k^H\scM(f^{-\alpha}) \,\,\Leftrightarrow \,\,\widetilde{V}^{k'+\alpha}\scO_X = W_l\widetilde{V}^{k'+\alpha}\scO_X \text{ for all }k'\leq k.\]

\label{lemWcomp}
    
\end{lem}

\begin{proof}

\underline{$\Rightarrow$} By Theorem \ref{thmmainformula}, the assumption implies that
\[F_k^{t-\text{ord}}V^{\alpha}\scB_f =F_k^{t-\text{ord}}K_{l+\lfloor\alpha\rfloor}V^{\alpha}\scB_f + \ker(\psi_{-\alpha})\cap F_k^{t-\text{ord}}V^{\alpha}\scB_f.\]
It is a simple exercise to check that 
\[\ker(\psi_{-\alpha})=\text{im}((s+\alpha):i_{f,+}\scO_X(*f) \to i_{f,+}\scO_X(*f)),\]
so that 
\[F_k^{t-\text{ord}}V^{\alpha}\scB_f =F_k^{t-\text{ord}}K_{l+\lfloor\alpha\rfloor}V^{\alpha}\scB_f + (s+\alpha)\cdot F_{k-1}^{t-\text{ord}}V^{\alpha}\scB_f.\]
Now let $m\geq 0$ be minimal such that
\[(s+\alpha)^m\cdot F_k^{t-\text{ord}}V^{\alpha}\scB_f \subseteq V^{>\alpha}\scB_f.\]
If $l+\lfloor\alpha\rfloor < m$, then
\[(s+\alpha)^{m-1}\cdot F_k^{t-\text{ord}}V^{\alpha}\scB_f \subseteq (s+\alpha)^{m-1}\cdot K_{l+\lfloor\alpha\rfloor}V^{\alpha}\scB_f + (s+\alpha)^m\cdot F_{k-1}^{t-\text{ord}}V^{\alpha}\scB_f \subseteq V^{>\alpha}\scB_f,\]
a contradiction. Thus
\[F_k^{t-\text{ord}}V^{\alpha}\scB_f\subseteq K_{l+\lfloor\alpha\rfloor}V^{\alpha}\scB_f.\]
By Lemma \ref{lemmicromultformulae}, this implies that
\[\widetilde{V}^{k'+\alpha}\scO_X = W_l\widetilde{V}^{k'+\alpha}\scO_X \text{ for all }k'\leq k.\]

\underline{$\Leftarrow$} We induct on $k$. For $k=0$, 
\[F_0^HW_{n+l'+\lfloor\alpha\rfloor}\scM(f^{-\alpha})=(W_{l'}\widetilde{V}^{\alpha}\scO_X)f^{-\alpha}\]
for all $l'$, so the statement is immediate.

Otherwise, our assumption implies (by Lemma \ref{lemmicromultformulae}) that
\[F_k^{t-\text{ord}}V^{\alpha}\scB_f = F_k^{t-\text{ord}}K_{l+\lfloor\alpha\rfloor}V^{\alpha}\scB_f + F_{k-1}^{t-\text{ord}}V^{\alpha}\scB_f,\]
implying through application of Theorem \ref{thmmainformula} that
\[F_k^H\scM(f^{-\alpha})= F_k^HW_{n+l+\lfloor\alpha\rfloor}\scM(f^{-\alpha})+F_{k-1}^H\scM(f^{-\alpha}).\]
Our inductive hypothesis however says that $F_{k-1}^H\scM(f^{-\alpha})=F_{k-1}^HW_{n+l+\lfloor\alpha\rfloor}\scM(f^{-\alpha})$, thus giving us that
\[F_k^H\scM(f^{-\alpha})= F_k^HW_{n+l+\lfloor\alpha\rfloor}\scM(f^{-\alpha})\]
as required.
    
\end{proof}




    

    

\begin{thm}\hfill

\begin{enumerate}[label=\roman*)]

\item Let $\alpha\in(0,1]$ and $l\in\bZ_{\geq 0}$. Write $\rho_{f,\hspace{0.7pt}\fx}^{(l)}$ for the set of roots of $b_{f,\hspace{0.7pt}\fx}^{(l)}(s)$. Then
\[W_{n+l+\lfloor\alpha\rfloor}\scM(f^{-\alpha})=\scM(f^{-\alpha}) \text{ at }\fx\,\, \Rightarrow\,\, \rho_{f,\hspace{0.7pt}\fx}^{(l)}\cap(-\alpha +\bZ) = \emptyset.\]

\item Let $k \in\bZ$ and $\alpha\in(0,1]$. Then
\[F_k^HW_{n+\sum_{i=0}^k\widetilde{m}_{f,-\alpha-i,\hspace{0.7pt}\fx}+\lfloor\alpha\rfloor}\scM(f^{-\alpha})=F_k^H\scM(f^{-\alpha}) \text{ at }\fx.\]
In particular, 
\[W_{n+\sum_{i\geq 0}\widetilde{m}_{f,-\alpha-i,\hspace{0.7pt}\fx}+\lfloor\alpha\rfloor}\scM(f^{-\alpha})=\scM(f^{-\alpha}) \text{ at }\fx.\]

\end{enumerate}

\label{thmlargestweight}
    
\end{thm}

\begin{proof}
    
\begin{enumerate}[label=\roman*)]

\item  Using the same argument as used at the start of the proof of Lemma \ref{lemWcomp}, we see that
\[F_k^{t-\text{ord}}V^{\alpha}\scB_{f,\hspace{0.7pt}\fx} \subseteq K_{l+\lfloor\alpha\rfloor}V^{\alpha}\scB_{f,\hspace{0.7pt}\fx} \,\,\text{ for all }\, k\geq 0,\]
i.e. that
\[V^{\alpha}\scB_{f,\hspace{0.7pt}\fx} = K_{l+\lfloor\alpha\rfloor}V^{\alpha}\scB_{f,\hspace{0.7pt}\fx},\]
i.e. that
\[V^{\alpha}\widetilde{\scB}_{f,\hspace{0.7pt}\fx} = K_lV^{\alpha}\widetilde{\scB}_{f,\hspace{0.7pt}\fx}.\]
In particular, for any $k\in\bZ$,
\[(s+k+\alpha)^l \cdot G_0V^{k+\alpha}\widetilde{\scB}_{f,\hspace{0.7pt}\fx}=(s+\alpha)^l \cdot G_kV^{\alpha}\widetilde{\scB}_{f,\hspace{0.7pt}\fx} \subseteq V^{>\alpha}\widetilde{\scB}_{f,\hspace{0.7pt}\fx},\]
so that $-k-\alpha\notin\rho_{f,\hspace{0.7pt}\fx}^{(l)}$ by Lemma \ref{lemredBSmult}.

\item To ease notation, we write $l_i:=\widetilde{m}_{f,-\alpha-i,\hspace{0.7pt}\fx}$. Then, by Lemma \ref{lemredBSmult}, we have that
\[(s+\alpha)^{l_i}\cdot\text{gr}^G_0\text{gr}_V^{i+\alpha}\widetilde{\scB}_{f,\hspace{0.7pt}\fx}=0\,\,\text{ for all }i \leq k.\]
Applying Zassenhaus' lemma and iterating, we obtain that 
\[(s+\alpha)^{\sum_{i=0}^{k'}l_i}\cdot G_0V^{k'+\alpha}\widetilde{\scB}_{f,\hspace{0.7pt}\fx}\subseteq V^{>k'+\alpha}\widetilde{\scB}_{f,\hspace{0.7pt}\fx}+G_{-k'-1}V^{k'+\alpha}\widetilde{\scB}_{f,\hspace{0.7pt}\fx}\,\,\text{ for all }k'\leq k.\]
As seen before, $\partial_t^{-k'-1}\in V^{>k'+\alpha}\widetilde{\scB}_{f,\hspace{0.7pt}\fx}$, so 
\[(s+\alpha)^{\sum_{i=0}^{k'}l_i}\cdot G_0V^{k'+\alpha}\widetilde{\scB}_{f,\hspace{0.7pt}\fx}\subseteq V^{>k'+\alpha}\widetilde{\scB}_{f,\hspace{0.7pt}\fx}.\]
In particular
\[F_0^{t-\text{ord}}V^{k'+\alpha}\widetilde{\scB}_{f,\hspace{0.7pt}\fx}=F_0^{t-\text{ord}}K_{\sum_{i=0}^{k'}l_i}V^{k'+\alpha}\widetilde{\scB}_{f,\hspace{0.7pt}\fx},\]
so
\[\widetilde{V}^{k'+\alpha}\scO_{X,\hspace{0.7pt}\fx} = W_{\sum_{i=0}^{k'}l_i}\widetilde{V}^{k'+\alpha}\scO_{X,\hspace{0.7pt}\fx}\,\,\text{ for all }k'\leq k.\]
By Lemma \ref{lemWcomp}, this implies in particular that 
\[F_k^HW_{n+\sum_{i=0}^kl_i+\lfloor\alpha\rfloor}\scM(f^{-\alpha})_{\fx}=F_k^H\scM(f^{-\alpha})_{\fx}.\]

\end{enumerate}

\end{proof}



    

\begin{eg}

In particular, the highest weight $w_{\text{max}}$ of $\scM(f^{-\alpha})_{\fx}$ satisfies, writing $w_{\text{max}}=n+l_{\text{max}}+\lfloor\alpha\rfloor$, the inequalities
\[\text{max}\{\widetilde{m}_{f,-\alpha-i,\hspace{0.7pt}\fx}\,|\,i \in \bZ\} \leq l_{\text{max}}\leq \sum_{i\geq 0}\widetilde{m}_{f,-\alpha-i,\hspace{0.7pt}\fx}.\]
In certain situations, these upper and lower bounds coincide and thus give equalities:

\begin{enumerate}[label=\roman*)]
    \item If $n=2$ and $\alpha=1$, then
    \[w_{\text{max}}=n+\widetilde{m}_{f,-1,\hspace{0.7pt}\fx}+1,\]
    which, since in general all roots of the reduced Bernstein-Sato polynomial have multiplicity $\leq n-1$, equals either 3 or 4.
    \item If $f$ defines a \emph{strongly Koszul free divisor}, then it is shown in \cite{Nar15} that the set of roots of the Bernstein-Sato polynomial is contained in the open interval $(-2,0)$, so for $\alpha=1$ we have that 
    \[w_{\text{max}}=n+\widetilde{m}_{f,-1,\hspace{0.7pt}\fx}+1.\]
    \item If $f$ is a (reduced, central) arrangement of hyperplanes, then Saito has shown in \cite{MSai16} that all roots of the Bernstein-Sato polynomial lie in the interval $(-2,0)$, and that the multiplicity of $-1$ as a root of the Bernstein-Sato polynomial is $n$. Thus, for $\alpha=1$, we have the equality
    \[w_{\text{max}}=2n.\]
\end{enumerate}

\label{egweightbounds}
    
\end{eg}

\begin{rem}

Note that it is unclear to us whether the length of the weight filtration on $\scM(f^{-\alpha})$ (i.e., the size of the set $\{l \in \bZ \,|\, \text{gr}_l^W\scM(f^{-\alpha})\neq 0\}$) is equal in general to $1$ plus the difference between the highest weight and smallest weight ($=n$) of $\scM(f^{-\alpha})$.
    
\end{rem}

\noindent Now we switch our attention to the Hodge filtration $F_{\bullet}^H$ on $W_{n+l}\scM(f^{-\alpha})$, for fixed $l$, and obtain a bound for its generating level using our understanding of this filtration in terms of $V$-filtrations and microlocal multiplier ideals.

Mustață and Popa (\cite{MP19c}, Theorem E) proved\footnote{This result is only proved in the case that $f$ is reduced in \cite{MP19c}, but the proof otherwise is completely analogous.} the following generating level bound for the Hodge filtration on $\scM(f^{-\alpha})$:

\begin{thm}

Assume that $\alpha\in(0,1]$. The generating level of $(\scM(f^{-\alpha}),F_{\bullet}^H)$ at the point $\fx\in X$ is always bounded above by
\[n-\lceil\widetilde{\alpha}_{f,\hspace{0.7pt}\fx}+\alpha\rceil.\]

\label{thmMPgenlevalg}
    
\end{thm}

In the remainder of this subsection we prove an analogous result for the filtered $\scD_X$-module $(W_{n+l}\scM(f^{-\alpha}),F_{\bullet}^H)$. We employ the same method as is used in \cite{MP19c}. Namely, we use the following lemma to reduce the proof to finding a bound on the generating levels of weight filtration steps on the vanishing cycles.

\begin{lem}

Let $\alpha\in(0,1]$ and $k \in \bZ_{\geq 0}$, and assume that the generating level of the filtered $\scD_X$-module $(\text{\emph{gr}}^W_{n-2+l}\text{\emph{gr}}_V^{\alpha}\scO_X,F_{\bullet}^H)$ is less than or equal to $k+1$. Then the generating level of $(\text{\emph{gr}}^W_{n+l}\scM(f^{-\alpha}),F_{\bullet}^H)$ is less than or equal to $k$.

\label{lemgenlevphi}
    
\end{lem}

\begin{proof}

This follows by Theorem \ref{thmmainformula}.
    
\end{proof}

We will then use the following two criteria, as in \cite{MP19c}, to obtain bounds on the generating level of the modules $(\text{gr}^W_{n-2+l}\text{gr}_V^{\alpha}\scO_X,F_{\bullet}^H)$. The proofs are then close to analogous to the proofs used in \cite{MP19c}.

\begin{lem}[\cite{MP19c}, Proposition 3.1 and Proposition 3.3]

Let $(\scM,F_{\bullet})$ be a left filtered $\scD_X$-module underlying a complex mixed Hodge module $M$.

\begin{enumerate}[label=\roman*)]

\item If the filtered $\scD_X$-module $(\bD\scM,F_{\bullet})$ underlying the dual Hodge module $\bD M$ satisfies $F_{2n-k-1}\bD\scM=0$, then the generating level of $(\scM,F_{\bullet})$ is less than or equal to $k$. 

\item Assume further that $\bD(\scM,F_{\bullet})\simeq (\scM,F_{\bullet})(d)$ for some integer $d$. Then the generating level of $(\scM,F_{\bullet})$ is less than or equal to $k$ if
\[\scE xt_{\scO_X}^j(\text{\emph{gr}}_{j-p+n-d}^F\scM,\scO_X)=0 \hspace{15pt} \forall \,0\leq j \leq n, \,\,\,p>k.\]
    
\end{enumerate}

\label{lemgenlevcrit}
    
\end{lem}

As in \cite{MP19c}, we separate into two separate cases, applying one of each of the two above criteria in each case. Using the first of the two criteria, we get the following bound in the case $\alpha \in (0,1)$.

\begin{thm}

Assume $\alpha\in(0,1)$. Let $l \geq 0$. The generating level of $(\text{\emph{gr}}^W_{n+l}\scM(f^{-\alpha}),F_{\bullet}^H)$ at a point $\fx\in X$ is less than or equal to 
\[n-l-\lceil\alpha+\widetilde{\alpha}_{f,\hspace{0.7pt}\fx}\rceil+1.\]

\label{thmgenlevelW1}
    
\end{thm}

\begin{proof}

Assume that $k$ is an integer satisfying
\[k \geq n-l-\lceil\alpha+\widetilde{\alpha}_{f,\hspace{0.7pt}\fx}\rceil+1.\]
By Lemma \ref{lemgenlevphi} and Lemma \ref{lemgenlevcrit}.i), it suffices to prove that 
\[F_{2n-k-2}^H\bD(\text{gr}^W_{n-2+l}\text{gr}_V^{\alpha}\scO_X)=0.\]
Now, there is a natural isomorphism of complex mixed Hodge modules
\[\bD(\psi_{f,e(\alpha)}\underline{\bC}_X[n])\simeq (\psi_{f,e(-\alpha)}\underline{\bC}_X[n])(n-1)\]
induced by the natural polarisation on $\underline{\bC}_X[n]$ (see \cite{MSai89}, Theorem 1.6). Moreover, the indexing of weight filtrations reverse under the duality functor, i.e.
\[\bD \circ \text{gr}^W_{\bullet}=\text{gr}^W_{-\bullet}\circ \bD.\]
(See \cite{MSai90}, Proposition 2.6 and 2.17. See also \cite{Sch14}, Section 29 and \cite{CNS22}, Theorem 4.9.) Therefore we obtain the following isomorphism:
\[\bD(\text{gr}^W_{n-2+l}\psi_{f,e(\alpha)}\underline{\bC}_X[n])\simeq \text{gr}^W_{-n+2-l}\bD(\psi_{f,e(\alpha)} \underline{\bC}_X[n]) \simeq (\text{gr}^W_{n-l}(\psi_{f,e(-\alpha)} \underline{\bC}_X[n]))(n-1).\]
Recall that the weight filtration on the vanishing cycles is given by the monodromy weight filtration centred at $n-1$, whose graded pieces are symmetric about $n-1$, in that there is an isomorphism of complex mixed Hodge modules 
\[\text{gr}^W_{n-1+l'}(\psi_{f,e(\gamma)} \underline{\bC}_X[n]) \isommap \text{gr}^W_{n-1-l'}(\psi_{f,e(\gamma)} \underline{\bC}_X[n])(-l')\]
for any $\gamma$ and $l'$, whose underlying morphism of $\scD_X$-modules is given by multiplication by $(\partial_tt-\gamma)^{l'}$. Thus, combining, we have a natural isomorphism
\[\bD(\text{gr}^W_{n-2+l}\psi_{f,e(\alpha)}\underline{\bC}_X[n])\simeq\text{gr}^W_{n-2+l}(\psi_{f,e(-\alpha)} \underline{\bC}_X[n])(n-2+l).\footnotemark\]

\footnotetext{This isomorphism is in fact precisely the one inducing a canonical polarisation on the (complex) pure Hodge module $\text{gr}^W_{n-2+l}\psi_{f,e(\alpha)}\underline{\bC}_X[n]$ of weight $n-2+l$.}

Thus, by the discussion at the beginning of the proof, we see that it suffices to prove that 
\[F_{n-k-l}^H\text{gr}^W_{n-2+l}\text{gr}^{1-\alpha}_V\scO_X=F_{2n-k-2-(n-2+l)}^H\text{gr}^W_{n-2+l}\text{gr}^{1-\alpha}_V\scO_X=0.\]

Now, $n-k-l-1+1-\alpha<\widetilde{\alpha}_{f,\hspace{0.7pt}\fx}$, so by Lemma \ref{lemmicrominlexp}, it follows that 
\[\widetilde{V}^{>n-k-l-1+1-\alpha}\scO_X=\scO_X. \]
Now, by Corollary \ref{corpartial}, this implies that
\[\partial_t^{n-k-l-1}\in V^{>1-\alpha}\scO_X.\]
Recalling (see Theorem \ref{thmphipsiMHM}) that the Hodge filtration on $\text{gr}_V^{1-\alpha}\scO_X$ is induced by the Hodge filtration on $i_{f,+}\scO_X$, which equals the $t$-order filtration up to a shift by $-1$, this finally implies that
\[F_{n-k-l}^HV^{1-\alpha}\scO_X=F_{n-k-l-1}^{t-\text{ord}}i_{f,+}\scO_X=F_{n-k-l}^HV^{>1-\alpha}\scO_X,\]
so that in particular
\[F_{n-k-l}^H\text{gr}^W_{n-2+l}\text{gr}_V^{1-\alpha}\scO_X=0\]
as required.
    
\end{proof}


Running through the same proof as above in the case $\alpha=1$, we obtain the bound $n-l-\lceil\widetilde{\alpha}_{f,\hspace{0.7pt}\fx}\rceil+1$. However, this can be slightly improved, using our second generating level criterion.

\begin{thm}

The generating level of $(\text{\emph{gr}}_{n+l}^W\scO_X(*f),F_{\bullet}^H)$ at the point $\fx\in X$ is less than or equal to 
\[n-l-\lceil\widetilde{\alpha}_{f,\hspace{0.7pt}\fx}\rceil.\]

\label{thmgenlevelW2}

\end{thm}

\begin{proof}

Let $k$ be an integer satisfying $k \geq n-l-\lceil\widetilde{\alpha}_{f,\hspace{0.7pt}\fx}\rceil+1$.

As in the proof of Theorem \ref{thmgenlevelW1}, we have a canonical isomorphism
\[\bD(\text{gr}^W_{n-2+l}\psi_{f,1}\underline{\bC}_X[n])\simeq\text{gr}^W_{n-2+l}(\psi_{f,1} \underline{\bC}_X[n])(n-2+l).\]
Therefore, we may apply the second criterion from Lemma \ref{lemgenlevcrit}, and we see (using Lemma \ref{lemgenlevphi}) that it suffices to show that
\[\scE xt_{\scO_X}^j(\text{gr}_{j-p}^{F^H}\text{gr}^W_{n-2+l}\text{gr}_V^1\scO_X,\scO_X)=0 \hspace{5pt} \text{ for all } \,\,0\leq j \leq n, \,\,\,p\geq k+l-1.\]
Now, we of course have a strict filtered short exact sequence
\[0\to W_{n-2+l-1}\text{gr}_V^1\scO_X\to W_{n-2+l}\text{gr}_V^1\scO_X\to \text{gr}^W_{n-2+l}\text{gr}_V^1\scO_X \to 0,\]
and the strictness induces for any $j$ and $p$ a short exact sequence
\[0\to \text{gr}^{F^H}_{j-p}W_{n-2+l-1}\text{gr}_V^1\scO_X\to \text{gr}^{F^H}_{j-p}W_{n-2+l}\text{gr}_V^1\scO_X\to \text{gr}^{F^H}_{j-p}\text{gr}^W_{n-2+l}\text{gr}_V^1\scO_X\to 0.\]
Thus, it now suffices to prove the vanishing
\begin{equation}\scE xt_{\scO_X}^j(\text{gr}_{j-p}^{F^H}W_{n-2+l}\text{gr}_V^1\scO_X,\scO_X)=0 \hspace{15pt} \text{ for all } \,\,0\leq j \leq n, \,\,\,p\geq k+l-1.\label{eqext}\end{equation}
To prove this vanishing, we need to better understand the $\scO_X$-modules $\text{gr}_{j-p}^{F^H}W_{n-2+l}\text{gr}_V^1\scO_X$ for the necessary values of $j$ and $p$. To this end, we have the following claim.

\vspace{10pt}

\textbf{Claim.} Write $W_lJ_q:=\{g\in\scO_X\mid g\partial_t^q\in M_{n-2+l}V^1\scO_X\}\subseteq\scO_X.$ Then, for $q\in\bZ$ satisfying $q \leq \lceil\widetilde{\alpha}_{f,\hspace{0.7pt}\fx}\rceil-1$, we have an isomorphism
\[\text{gr}_{q+1}^{F^H}W_{n-2+l}\text{gr}_V^1\scO_X \simeq \frac{W_lJ_q}{(f)}\]
at $\fx$. Moreover, if $q+1\leq \lceil\widetilde{\alpha}_{f,\hspace{0.7pt}\fx}\rceil-1$, then $W_lJ_q=\scO_X$ at $\fx$.

\emph{Proof of claim.} If $q \leq \lceil\widetilde{\alpha}_{f,\hspace{0.7pt}\fx}\rceil-1$, then $q <\widetilde{\alpha}_{f,\hspace{0.7pt}\fx}$, so Proposition \ref{prophodgepole2} implies that we have (at $\fx$) an equality
\[F_{k-1}^HW_{n+1}\scO_X(*f)=\scO_Xf^{-k},\]
for every $k \leq q$. Corollary \ref{corHodgemicrocomp} then implies that $\widetilde{V}^{>k}\scO_X=\scO_X$, which by Corollary \ref{corpartial} then implies in turn that
\[\partial_t^{k-1}\in K_1V^1\scO_X \,\,\text{ for all }\,\, k \leq q.\]
This then implies that 
\[t\cdot\partial_t^k = f\partial_t^k-k\partial_t^{k-1}\in V^{>1}\scO_X \,\,\text{ for all }\,\, k \leq q,\]
so that
\[F_{q+1}^HW_{n-2+l}\text{gr}_V^1\scO_X \simeq \frac{W_lJ_q}{J_q'},\]
where $J_q'=\{g\in\scO_X\mid g\partial_t^q\in V^{>1}\scO_X\}\subseteq\scO_X$. 

Now, if $q+1\leq \lceil\widetilde{\alpha}_{f,\hspace{0.7pt}\fx}\rceil-1$, then as above we see that $\widetilde{V}^{>q+1}\scO_X=\scO_X$, implying in particular by Corollary \ref{corpartial} that 
\[\partial_t^k\in K_lV^1\scO_X\subseteq M_{n-2+l}V^1\scO_X\,\,\text{ for all }\,\,k \leq q,\]
so that $W_lJ_q = \scO_X$.

Finally, we return again to considering $q \leq \lceil\widetilde{\alpha}_{f,\hspace{0.7pt}\fx}\rceil-1$. See that $(q-1)+1 \leq  \lceil\widetilde{\alpha}_{f,\hspace{0.7pt}\fx}\rceil-1$, so $W_lJ_{q-1}=\scO_X$. Now the inclusion map
\[F_q^HW_{n-2+l}\text{gr}_V^1\scO_X \hookrightarrow F_{q+1}^HW_{n-2+l}\text{gr}_V^1\scO_X\]
corresponds under the above isomorphisms to the map
\[\frac{\scO_X}{J_{q-1}'}\to\frac{W_lJ_q}{J_q'}\] given by multiplication by $\frac{1}{q}f$, since
\[\partial_t^{q-1} = \frac{1}{q}f\partial_t^q-\frac{1}{q}t\cdot\partial_t^q,\]
and $t\cdot\partial_t^q \in V^{>1}\scO_X$ as seen above.
Therefore we have the required isomorphism
\[\text{gr}^{F^H}_{q+1}W_{n-2+l}\text{gr}_V^1\scO_X \simeq \frac{W_lJ_q}{J_q'+(f)}=\frac{W_lJ_q}{(f)},\]
since, as may be seen easily using that $V^{>1}\scO_X=t\cdot V^{>0}\scO_X$, $J_q'\subseteq (f)$. \hspace{15pt} $\square$

\vspace{10pt}

Now we use the above claim to prove the required vanishing. Recall $k$ is an integer such that $k\geq n-l-\lceil\widetilde{\alpha}_{f,\hspace{0.7pt}\fx}\rceil+1$, and we are required to show that the Ext groups \eqref{eqext} all vanish.

See firstly that 
\[j-p\leq j-(k+l-1) \leq n-(n-l-\lceil\widetilde{\alpha}_{f,\hspace{0.7pt}\fx}\rceil+1+l-1)=\lceil\widetilde{\alpha}_{f,\hspace{0.7pt}\fx}\rceil.\]
So we may apply the claim to see that
\[\text{gr}_{j-p}^{F^H}W_{n-2+l}\text{gr}_V^1\scO_X \simeq \frac{\scO_X}{(f)}\]
if $j-p < \lceil\widetilde{\alpha}_{f,\hspace{0.7pt}\fx}\rceil$, and
\[\text{gr}_{j-p}^{F^H}W_{n-2+l}\text{gr}_V^1\scO_X \simeq \frac{W_lJ_{j-p-1}}{(f)}\]
if $j-p = \lceil\widetilde{\alpha}_{f,\hspace{0.7pt}\fx}\rceil$. See also that this equality $j-p = \lceil\widetilde{\alpha}_{f,\hspace{0.7pt}\fx}\rceil$ happens only if $j=n$, so only the $n$-th Ext group needs to be checked to vanish. The two short exact sequences 
\[0 \to\scO_X\xrightarrow{f}\scO_X \to \frac{\scO_X}{(f)} \to 0\]
and 
\[0 \to \frac{W_lJ_{j-p-1}}{(f)} \to \frac{\scO_X}{(f)}\to \frac{\scO_X}{W_lJ_{j-p-1}}\to 0\]
thus give us all of the required vanishings.
    
\end{proof}

\begin{cor}

Assume $\alpha\in(0,1]$. The generating level of $(W_{n+l}\scM(f^{-\alpha}),F_{\bullet}^H)$ at $\fx$ is less than or equal to the minimum of $n-1$ and
$n-\lceil\alpha+\widetilde{\alpha}_{f,\hspace{0.7pt}\fx}\rceil+1-\lfloor\alpha\rfloor$.

\label{corgenlevelW}
    
\end{cor}

\begin{proof}

Combining Theorems \ref{thmgenlevelW1} and \ref{thmgenlevelW2} gives us that $(\text{gr}^W_{n+l}\scM(f^{-\alpha}),F_{\bullet}^H)$ has its generating level bounded above by $n-\lceil\alpha+\widetilde{\alpha}_{f,\hspace{0.7pt}\fx}\rceil$ when $l\geq 1$, which is less than or equal to both $n-1$ and $n-\lceil\alpha+\widetilde{\alpha}_{f,\hspace{0.7pt}\fx}\rceil+1-\lfloor\alpha\rfloor$.

Combining Theorems \ref{thmgenlevelW1} and Corollary \ref{corgenlevWn} (and recalling that $W_n\scO_X(*f)=\scO_X$ so has generating level equal to zero) also tells us that the generating level of $(W_n\scM(f^{-\alpha}), F_{\bullet}^H)$ is bounded above by the minimum of $n-1$ and $n-\lceil\alpha+\widetilde{\alpha}_{f,\hspace{0.7pt}\fx}\rceil+1-\lfloor\alpha\rfloor$.

The result now follows from the simple fact that in a strict short exact sequence of filtered objects
\[0 \to \scM_1\to\scM_2\to\scM_3\to 0,\]
the generating level of $\scM_2$ is less than or equal to the maximum of the generating levels of $\scM_1$ and $\scM_3$.

\end{proof}

\newpage

\section{Parametrically prime divisors}\label{sectionPPD}



In this section, we calculate the filtrations $F_{\bullet}^H$ and $W_{\bullet}$ on $\scM(f^{-\alpha})$ for a large class of divisors. This may be seen as an extension of the main results of \cite{BD24} and \cite{D24}, in which only the Hodge filtration is considered. 

As before, $X$ is an arbitrary complex manifold, $f\in\scO_X$ is reduced\footnote{Reducedness is actually completely unnecessary and is only included here in our assumptions as the same assumption is used in \cite{BD24} and \cite{D24}.} and is such that $f^{-1}(0)$ defines a hypersurface of $X$, and $\alpha \in\bQ_{\geq 0}$. The main results of this section require the following hypotheses on our holomorphic function $f$, as well as the inclusion $\rho_{f,\hspace{0.7pt}\fx}\subseteq (-2-\alpha,-\alpha)$.

\begin{defn}\label{defnPP}

\hfill

\begin{enumerate}[label=\alph*)]
    \item We say that $f$ is \emph{Euler homogeneous at $\fx\in X$} if there exists some $E \in\text{Der}_{\bC}(\scO_{X,\hspace{0.7pt}\fx})$ such that $E(f_{\fx})=f_{\fx}$, where $f_{\fx}\in\scO_{X,\hspace{1pt}\fx}$ is the germ of $f$ at $\fx\in X$. Such an $E$ is called an \emph{Euler vector field} for $f$ at $\fx$. $f$ is \emph{Euler homogeneous} if it is Euler homogeneous at $\fx$ for all $\fx\in X$.
    \item We say that $f$ is \emph{parametrically prime at $\fx\in X$} if $\text{gr}^{\sharp}(\text{ann}_{\scD_{X,\hspace{1pt}\fx}[s]} f_{\fx}^{s-1}) \subseteq \text{gr}^{\sharp}\scD_{X,\hspace{1pt}\fx}[s]$ is prime, where $\text{gr}^{\sharp}(\text{ann}_{\scD_{X,\hspace{1pt}\fx}[s]} f_{\fx}^{s-1})$ is the ideal of symbols of elements of $\text{ann}_{\scD_{X,\hspace{1pt}\fx}[s]} f_{\fx}^{s-1}$ with respect to the \emph{total order filtration} $F_k^{\sharp}\scD_{X,\hspace{1pt}\fx}[s]:=\sum_{i=0}^k(F_{k-i}\scD_{X,\hspace{1pt}\fx})s^i$ on $\scD_{X,\hspace{1pt}\fx}[s]$. $f$ is \emph{parametrically prime} if it is parametrically prime at $\fx$ for all $\fx\in X$.
\end{enumerate}

\end{defn}

\begin{egs}

There are numerous classes of examples satisfying these conditions, as discussed in \cite{BD24} and \cite{D24}.

\begin{enumerate}[label=\roman*)]

\item If $f$ is of linear Jacobian type at $\fx$ (see \cite{BD24}, Definition 5.3, as well as \cite{CN08} and \cite{Nar15}, and \cite{Vas97}), then it is parametrically prime and (strongly) Euler homogeneous at $\fx$ (by \cite{BD24}, Propositions 5.4 and 5.9, and Proposition 4.17).

\item If $f$ is tame, strongly Euler homogeneous and Saito holonomic at $\fx$, then it is of linear Jacobian type at $\fx$ (see \cite{Wal17}).

\item Strongly Koszul free divisors satisfy the conditions in ii), see \cite{Nar15}.

\item Any positively weighted homogeneous locally everywhere divisor is Saito holonomic. Thus for instance all tame hyperplane arrangements satisfy the conditions in ii). And if $n=3$, any positively weighted homogeneous locally everywhere divisor satisfies the conditions in ii).

\item Note also that strongly Koszul free divisors (\cite{Nar15}, Theorem 4.1) and hyperplane arrangements (\cite{MSai16}, Theorem 1) satisfy locally everywhere the condition $\rho_{f,\hspace{0.7pt}\fx}\subseteq (-2,0)$. As we will see, this means such divisors satisfy our hypotheses in the case $\alpha=0$. Note also that \cite{Bath24}, Corollary 7.10 (a strengthening of \cite{BS23}, Proposition 1) develops a criterion in the case that $X=\bC^3$, $f$ is a positively weighted homogeneous polynomial locally everywhere, for determining precisely when the condition $\rho_{f,\hspace{0.7pt}\fx}\subseteq (-2-\alpha,-\alpha)$ holds, giving us more examples satisfying our later hypotheses.
    
\end{enumerate}
    
\end{egs}

A more thorough discussion of these conditions can be found in \cite{BD24}.

We now recall some of the main results from \cite{BD24} and \cite{D24}, with the aim of extending them to include weight filtration steps as well. We need to consider the following list of maps, some of which have already been used in this paper:

\begin{defn}

\hfill

\begin{enumerate}[label =\roman*)] 

\item $\pi_f:\scD_X[s]\to i_{f,+}\scO_X(*f)\,;\, P(s) \mapsto P(-\partial_tt)\cdot f^{-1}.$

\item $\psi_{-\beta}:i_{f,+}\scM(f^{-\alpha}) \to \scM(f^{-\alpha-\beta})\, ;\, u \partial_t^j \mapsto uQ_j(\beta)f^{-j-\beta}.$

\item $\phi_{-\alpha}:\scD_X[s]\to \scD_X\,;\,P(s)\mapsto P(-\alpha)$.

\item $\widetilde{\phi}_{f,-\alpha}:\scD_X[s]\to \scM(f^{-\alpha})\,;\,P(s)\mapsto P(-\alpha)\cdot f^{-1-\alpha}$.
    
\end{enumerate}

\vspace{2pt}

\noindent See that $\widetilde{\phi}_{f,-\alpha}=\psi_{-\alpha}\circ\pi_f$, and that the kernel of $\pi_f$ equals $\text{ann}_{\scD_{X,\hspace{1pt}\fx}}f_{\fx}^{s-1}$. We will also use the same notation when localising at a point $\fx\in X$, hoping that context will help differentiate.

\end{defn}

\vspace{3pt}

\begin{defn}

Let $\fx\in X$. We define
\[\beta_{f,-\alpha,\hspace{0.7pt}\fx}(s):=\prod_{\lambda \in \rho_{f,\hspace{0.7pt}\fx}\cap (-\alpha-1,-\alpha)}(s+\lambda+1)^{m_{f,\lambda,\hspace{0.7pt}\fx}},\]
where we recall that $\rho_{f,\hspace{0.7pt}\fx}$ is the set of roots of the local Bernstein-Sato polynomial $b_{f,\hspace{0.7pt}\fx}(s)$, and that $m_{f,\lambda,\hspace{0.7pt}\fx}$ is the multiplicity of $\lambda$ as a root of $b_{f,\hspace{0.7pt}\fx}(s)$. 

We also define the $\scD_{X,\hspace{1pt}\fx}[s]$-ideal
\[\Gamma_{f,-\alpha,\hspace{0.7pt}\fx}:=\scD_{X,\hspace{1pt}\fx}[s]f_{\fx}+\scD_{X,\hspace{1pt}\fx}[s]\beta_{f,-\alpha,\hspace{0.7pt}\fx}(-s)+\text{ann}_{\scD_{X,\hspace{1pt}\fx}[s]}f_{\fx}^{s-1}\subseteq \scD_{X,\hspace{1pt}\fx}[s].\]

\end{defn}

\vspace{3pt}

\begin{thm}[\cite{D24}, Corollary 3.8 and Theorem 1.2]

Assume $\alpha\geq 0$. Let $\fx\in X$ and assume that $\rho_{f,\hspace{0.7pt}\fx}\subseteq(-2-\alpha,-\alpha)$. Then
\begin{enumerate}[label=\roman*)]

\item \, $\pi_f(\Gamma_{f,-\alpha,\hspace{0.7pt}\fx})=V^{\alpha}i_{f,+}\scO_X(*f)_{(\fx,0)}$.\vspace{4pt}

\item \, $F_0^H\scM(f^{-\alpha})_{\fx}=\left(\Gamma_{f,-\alpha,\hspace{0.7pt}\fx}\cap\scO_{X,\hspace{0.7pt}\fx}\right)\cdot f_{\fx}^{-1-\alpha}$.

\end{enumerate}

\label{thmgamma}
    
\end{thm}

Under the same assumptions, we obtain as a simple consequence an analogous result concerning weight filtration steps.

\begin{defn}

For $\fx\in X$, define 
\[W_0\Gamma_{f,-\alpha,\hspace{0.7pt}\fx}:=\Gamma_{f,-\alpha-\epsilon,\hspace{0.7pt}\fx}\subseteq \Gamma_{f,-\alpha,\hspace{0.7pt}\fx},\]
where $\epsilon>0$ is sufficiently small. For $l \in\bZ_{\geq 0}$, define
\[W_l\Gamma_{f,-\alpha,\hspace{0.7pt}\fx}:=\{\gamma \in \Gamma_{f,-\alpha,\hspace{0.7pt}\fx}\mid (s+\alpha)^l\gamma\in W_0\Gamma_{f,-\alpha,\hspace{0.7pt}\fx}\}\subseteq \Gamma_{f,-\alpha,\hspace{0.7pt}\fx}.\]
    
\end{defn}

\begin{thm}

Let $\fx\in X$ and assume that $\rho_{f,\hspace{0.7pt}\fx}\subseteq(-2-\alpha,-\alpha)$. Let $l \in\bZ_{\geq 0}$. Then
\begin{enumerate}[label=\roman*)]

\item \, $\pi_f(W_l\Gamma_{f,-\alpha,\hspace{0.7pt}\fx})=K_lV^{\alpha}i_{f,+}\scO_X(*f)_{(\fx,0)}$.\vspace{4pt}

\item \, $W_{n+l}\scM(f^{-\alpha})_{\fx}=\phi_{-\alpha}(W_l\Gamma_{f,-\alpha,\hspace{0.7pt}\fx})\cdot f_{\fx}^{-1-\alpha}$.\vspace{4pt}

\item \, $F_0^HW_{n+l}\scM(f^{-\alpha})_{\fx}=\left(W_l\Gamma_{f,-\alpha,\hspace{0.7pt}\fx}\cap\scO_{X,\hspace{0.7pt}\fx}\right)\cdot f_{\fx}^{-1-\alpha}$.

\end{enumerate}

\label{thmgammaW}
    
\end{thm}

\begin{proof}

\begin{enumerate}[label=\roman*)]

\item For $\epsilon>0$ sufficiently small, we also have that $\rho_{f,\hspace{0.7pt}\fx}\subseteq(-2-\alpha-\epsilon,-\alpha-\epsilon)$. Thus, by Theorem \ref{thmgamma}, 
\[\pi_f(\Gamma_{f,-\alpha-\epsilon,\hspace{0.7pt}\fx}) = V^{\alpha +\epsilon}i_{f,+}\scO_X(*f)_{(\fx,0)} = V^{>\alpha}i_{f,+}\scO_X(*f)_{(\fx,0)},\]
i.e.
\[\pi_f(W_0\Gamma_{f,-\alpha,\hspace{0.7pt}\fx}) =K_0V^{\alpha}i_{f,+}\scO_X(*f)_{(\fx,0)}.\]

\noindent Now, for $l >0$, if $\gamma\in W_l\Gamma_{f,-\alpha,\hspace{0.7pt}\fx}$, then $\pi_f(\gamma)\in V^{\alpha}i_{f,+}\scO_X(*f)_{(\fx,0)}$ by Theorem \ref{thmgamma}, and 
\[(s+\alpha)^l\cdot \pi_f(\gamma)=\pi_f((s+\alpha)^l\gamma)\in \pi_f(W_0\Gamma_{f,-\alpha,\hspace{0.7pt}\fx})=V^{>\alpha}i_{f,+}\scO_X(*f)_{(\fx,0)},\]
so $\pi_f(\gamma)\in K_lV^{\alpha}i_{f,+}\scO_X(*f)_{(\fx,0)}$.

If, conversely, $u \in K_lV^{\alpha}i_{f,+}\scO_X(*f)_{(\fx,0)}$, then there exists by Theorem \ref{thmgamma} some $\gamma\in\Gamma_{f,-\alpha,\hspace{0.7pt}\fx}$ such that $\pi_f(\gamma)=u$. Then
\[\pi_f((s+\alpha)^l\gamma)=(s+\alpha)^l\cdot u \in V^{>\alpha}i_{f,+}\scO_X(*f)_{(\fx,0)} = \pi_f(W_0\Gamma_{f,-\alpha,\hspace{0.7pt}\fx}).\]
Thus
\[(s+\alpha)^l\gamma \in W_0\Gamma_{f,-\alpha,\hspace{0.7pt}\fx} + \ker\pi_f =W_0\Gamma_{f,-\alpha,\hspace{0.7pt}\fx} + \text{ann}_{\scD_{X,\hspace{1pt}\fx}[s]}f_{\fx}^{s-1} = W_0\Gamma_{f,-\alpha,\hspace{0.7pt}\fx},\]
so $\gamma \in W_l\Gamma_{f,-\alpha,\hspace{0.7pt}\fx}$, so $u \in\pi_f(W_l\Gamma_{f,-\alpha,\hspace{0.7pt}\fx})$.

\item By Theorem \ref{thmmainformula}, 
\begin{align*}W_{n+l}\scM(f^{-\alpha})_{\fx} = \psi_{-\alpha}(K_lV^{\alpha}i_{f,+}\scO_X&(*f)_{(\fx,0)})=(\psi_{-\alpha}\circ\pi_f)(W_l\Gamma_{f,-\alpha,\hspace{0.7pt}\fx})\\ &=\widetilde{\phi}_{f,-\alpha}(W_l\Gamma_{f,-\alpha,\hspace{0.7pt}\fx}) = \phi_{-\alpha}(W_l\Gamma_{f,-\alpha,\hspace{0.7pt}\fx})\cdot f_{\fx}^{-1-\alpha},
\end{align*}
where we have used part i).

\item Using the same argument as part ii), it suffices to show that 
\[\pi_f(W_l\Gamma_{f,-\alpha,\hspace{0.7pt}\fx}\cap\scO_{X,\hspace{0.7pt}\fx})=F_0^{t-\text{ord}}K_lV^{\alpha}i_{f,+}\scO_X(*f)_{(\fx,0)}.\]
$\subseteq$ is clear by part i). If $u \in F_0^{t-\text{ord}}K_lV^{\alpha}i_{f,+}\scO_X(*f)_{(\fx,0)}$, then by Theorem \ref{thmgamma} and part i), we may write
\[u=\pi_f(\gamma_1)=\pi_f(\gamma_2) \,\,\,\text{ with }\,\, \gamma_1 \in W_l\Gamma_{f,-\alpha,\hspace{0.7pt}\fx} \,\text{ and }\, \gamma_2 \in \Gamma_{f,-\alpha,\hspace{0.7pt}\fx}\cap\scO_{X,\hspace{0.7pt}\fx}.\]
Then
\[\gamma_2-\gamma_1 \in \ker\pi_f =\text{ann}_{\scD_{X,\hspace{1pt}\fx}}f_{\fx}^{s-1} \subseteq W_l\Gamma_{f,-\alpha,\hspace{0.7pt}\fx},\]
so we have that $\gamma_2 \in W_l\Gamma_{f,-\alpha,\hspace{0.7pt}\fx}$ also, implying that $u \in \pi_f(W_l\Gamma_{f,-\alpha,\hspace{0.7pt}\fx}\cap\scO_{X,\hspace{0.7pt}\fx})$ as required.

\end{enumerate}
\end{proof}

\begin{cor}

Let $\fx\in X$ and assume that $\rho_{f,\hspace{0.7pt}\fx}\subseteq(-2-\alpha,-\alpha)$. Assume moreover that $f$ is Euler homogeneous at $\fx$, with associated Euler vector field $E\in\text{\emph{Der}}_{\bC}(\scO_{X,\hspace{0.7pt}\fx})$. Let $l \in\bZ_{\geq 0}$ such that $l<m_{f,-\alpha-1,\hspace{0.7pt}\fx}$ (the multiplicity of $-\alpha-1$ as a root of $b_{f,\hspace{0.7pt}\fx}(s)$).

Choose a generating set $\zeta_1,\ldots,\zeta_m$ of $\text{\emph{ann}}_{\scD_{X,\hspace{1pt}\fx}}f_{\fx}^{s-1}$ and consider the left $\scD_{X,\hspace{1pt}\fx}$-linear map
\[\Phi_{l,-\alpha} :\scD_{X,\hspace{1pt}\fx}^{m+2}\to\scD_{X,\hspace{1pt}\fx}\,;\,(P_0,\ldots,P_{m+1})\mapsto P_0(E+\alpha+1)^l+\sum_{i=1}^mP_i\zeta_i+P_{m+1}f_{\fx}.\]
Write $K_{l,-\alpha}$ for the kernel of $\Phi_{l,-\alpha}$ and $p_1:\scD_{X,\hspace{1pt}\fx}^{m+2}\to\scD_{X,\hspace{1pt}\fx}$ for the projection onto the first coordinate. Then
\[W_{n+l}\scM(f^{-\alpha})_{\fx}=p_1(K_{l,-\alpha})\cdot f_{\fx}^{-1-\alpha}.\]

\label{corWformulagenl}
    
\end{cor}

\begin{proof}

\fbox{$\subseteq$} By Theorem \ref{thmgammaW}, if $u \in W_{n+l}\scM(f^{-\alpha})_{\fx}$, then there exists some $\gamma \in W_l\Gamma_{f,-\alpha,\hspace{0.7pt}\fx}$ such that $u = \gamma(-\alpha)\cdot f_{\fx}^{-1-\alpha}$. Since $(s+\alpha)^l \gamma \in W_0\Gamma_{f,-\alpha,\hspace{0.7pt}\fx}$, we may write 
\[(s+\alpha)^l\gamma = P(s)\beta_{f,-\alpha,\hspace{0.7pt}\fx}(-s)(s+\alpha)^{m_{f,-\alpha-1,\hspace{0.7pt}\fx}}+Q(s)(E-s+1)+\sum_iR^{(i)}(s)\zeta_i + T(s)f_{\fx}\]
for some $P,Q,R^{(i)},T \in \scD_{X,\hspace{1pt}\fx}[s]$. Write $Q= \sum_j Q_j (s+\alpha)^j$, $Q_j\in\scD_{X,\hspace{1pt}\fx}$ and similarly for $P,R^{(i)},T$. 

Comparing coefficients of $(s+\alpha)^l$, we see that 
\[\gamma(-\alpha) = Q_l(E+\alpha+1)-Q_{l-1}+\sum_iR_l^{(i)}\zeta_i+T_lf_{\fx}.\]
See in particular that 
\[u = \gamma(-\alpha)\cdot f_{\fx}^{-1-\alpha} = -Q_{l-1}\cdot f_{\fx}^{-1-\alpha} + T_l \cdot f_{\fx}^{-\alpha}.\]

Comparing coefficients of $(s+\alpha)^k$ with $k < l$, we have that 
\[0 = Q_k(E+\alpha+1)-Q_{k-1}+\sum_iR_k^{(i)}\zeta_i + T_kf_{\fx}\]
Combining each of these equations together, and noting that $Q_{-1}=0$, we see that $Q_{l-1} \in p_1(K_{l,-\alpha})$. See also finally that $T_lf_{\fx}$ is clearly also in $p_1(K_{l,-\alpha})$, as $\Phi_{l,-\alpha}(T_lf_{\fx},0,\ldots,0,-T_l(E+\alpha)^l)=0$. Thus $u \in p_1(K_{l,-\alpha})\cdot f_{\fx}^{-1-\alpha}$.

\fbox{$\supseteq$} Conversely, let $Q \in p_1(K_{l,-\alpha})$. Then by definition there exist $R^{(i)}, T \in \scD_{X,\hspace{1pt}\fx}$ such that 
\[0 = Q(E+\alpha+1)^l + \sum_iR^{(i)}\zeta_i + Tf_{\fx}.\]
Now define 
\[\gamma := \frac{Q \beta_{f,-\alpha,\hspace{0.7pt}\fx}(-s)}{\beta_{f,-\alpha,\hspace{0.7pt}\fx}(\alpha)}.\]
$\gamma \in \Gamma_{f,-\alpha,\hspace{0.7pt}\fx}$ clearly, and $\gamma(-\alpha) = Q$. Finally, $\gamma (s+\alpha)^l - \gamma (E+\alpha+1)^l \in W_0\Gamma_{f,-\alpha,\hspace{0.7pt}\fx}$ while 
\[\gamma(E+\alpha+1)^l = -\frac{\beta_{f,-\alpha,\hspace{0.7pt}\fx}(-s)}{\beta_{f,-\alpha,\hspace{0.7pt}\fx}(\alpha)}\left(\sum_iR^{(i)}\zeta_i +Tf_{\fx}\right) \in W_0\Gamma_{f,-\alpha,\hspace{0.7pt}\fx},\]
so $(s+\alpha)^l\gamma \in W_0\Gamma_{f,-\alpha,\hspace{0.7pt}\fx}$, so $\gamma \in W_l\Gamma_{f,-\alpha,\hspace{0.7pt}\fx}$. Thus 
\[Q\cdot f_{\fx}^{-1} = \gamma(-\alpha)\cdot f_{\fx}^{-1-\alpha} \in \phi_{-\alpha}(W_l\Gamma_{f,-\alpha,\hspace{0.7pt}\fx})\cdot f_{\fx}^{-1-\alpha}=W_{n+l}\scM(f^{-\alpha})_{\fx}\]
by Theorem \ref{thmgammaW}.

\end{proof}

\begin{rem}

This expression has the advantage of being somewhat easier to calculate on the face of it than the expression given in Theorem \ref{thmgammaW}; one need not calculate the Bernstein-Sato polynomial.
    
\end{rem}

We obtain results about the Hodge filtration when we admit the hypotheses appearing in Definition \ref{defnPP}, as seen in \cite{BD24} and \cite{D24}. 

\begin{thm}[\cite{BD24}, Theorem 1.6 and \cite{D24}, Theorem 1.1]

Assume that $f$ is parametrically prime and Euler homogeneous at a point $\fx\in X$, and assume that $\rho_{f,\hspace{0.7pt}\fx}\subseteq(-2-\alpha,-\alpha)$. Then, for all $k \in \bZ$,
\[F_k^H\scM(f^{-\alpha})_{\fx}=\phi_{-\alpha}(\Gamma_{f,-\alpha,\hspace{0.7pt}\fx}\cap F_k^{\sharp}\scD_{X,\hspace{1pt}\fx}[s])\cdot f_{\fx}^{-1-\alpha}.\]

\label{thmPPformula}
    
\end{thm}

A simple corollary is the following:

\begin{cor}

Assume that $f$ is parametrically prime and Euler homogeneous at a point $\fx\in X$. Assume moreover that $\rho_{f,\hspace{0.7pt}\fx}\subseteq(-2,-1]$. Then, for all $k\in\bZ$,
\[F_k^H\scO_X(*f)_{\fx}=F_k\scD_{X,\hspace{1pt}\fx}\cdot f_{\fx}^{-1}.\]

\label{cor-2-1}
    
\end{cor}

Using Theorem \ref{thmmainformula}, we can now also obtain an expression for the Hodge filtration when restricted to the weight filtration steps on $\scM(f^{-\alpha})$.

\begin{thm}

Assume that $f$ is Euler homogeneous and parametrically prime at a point $\fx\in X$, and assume that $\rho_{f,\hspace{0.7pt}\fx}\subseteq(-2-\alpha,-\alpha)$. Then, for all $k \in \bZ$ and $l\in\bZ_{\geq 0}$,
\[F_k^HW_{n+l}\scM(f^{-\alpha})_{\fx}=\phi_{-\alpha}(W_l\Gamma_{f,-\alpha,\hspace{0.7pt}\fx}\cap F_k^{\sharp}\scD_{X,\hspace{1pt}\fx}[s])\cdot f_{\fx}^{-1-\alpha}.\]

\label{thmPPformulaW}
    
\end{thm}

\begin{proof}

It suffices by Theorem \ref{thmmainformula} to prove that
\[\pi_f(W_l\Gamma_{f,-\alpha,\hspace{0.7pt}\fx}\cap F_k^{\sharp}\scD_{X,\hspace{1pt}\fx}[s])=F_k^{t-\text{ord}}K_lV^{\alpha}i_{f,+}\scO_X(*f)_{(\fx,0)}.\]
$\subseteq$ is clear. Let $u \in F_k^{t-\text{ord}}K_lV^{\alpha}i_{f,+}\scO_X(*f)_{(\fx,0)}$. Then, by \cite{D24}, Lemma 4.4, and by Theorem \ref{thmgammaW} above, we may write
\[u=\pi_f(\gamma_1)=\pi_f(\gamma_2) \,\,\,\text{ with }\,\, \gamma_1 \in W_l\Gamma_{f,-\alpha,\hspace{0.7pt}\fx} \,\text{ and }\, \gamma_2 \in \Gamma_{f,-\alpha,\hspace{0.7pt}\fx}\cap F_k^{\sharp}\scD_{X,\hspace{1pt}\fx}[s].\]
Then
\[\gamma_2-\gamma_1 \in \ker\pi_f =\text{ann}_{\scD_{X,\hspace{1pt}\fx}}f_{\fx}^{s-1} \subseteq W_l\Gamma_{f,-\alpha,\hspace{0.7pt}\fx},\]
so we have that $\gamma_2 \in W_l\Gamma_{f,-\alpha,\hspace{0.7pt}\fx}$ also, implying that $u \in \pi_f(W_l\Gamma_{f,-\alpha,\hspace{0.7pt}\fx}\cap F_k^{\sharp}\scD_{X,\hspace{1pt}\fx}[s])$ as required.
    
\end{proof}

\begin{cor}

Assume that $f$ is Euler homogeneous and parametrically prime at a point $\fx\in X$, and that $\rho_{f,\hspace{0.7pt}\fx}\subseteq(-2,-1]$. Choose $E,\zeta_1,\ldots,\zeta_m$ as in Corollary \ref{corWformulagenl}, and carry over the same notation. Let $l\in\bZ_{\geq 0}$ such that $l<m_{f,-1,\hspace{0.7pt}\fx}$. Then
\[F_k^HW_{n+l}\scO_X(*f)_{\fx} = ((p_1(K_{l,0})+\scD_{X,\hspace{1pt}\fx}(E+1))\cap F_k\scD_{X,\hspace{1pt}\fx})\cdot f_{\fx}^{-1}.\]

\label{corWformulagenlPP}
    
\end{cor}

\begin{proof}

The inclusion $\supseteq$ is immediate by Corollary \ref{corWformulagenl} and Corollary \ref{cor-2-1}. Also by Corollaries \ref{corWformulagenl} and \ref{cor-2-1}, if $u\in F_k^HW_{n+l}\scO_X(*f)_{\fx}$, then we may write
\[u= P\cdot f_{\fx}^{-1} = Q \cdot f_{\fx}^{-1} \,\,\, \text{ where }\,\, P\in p_1(K_{l,0}) \,\text{ and }\, Q\in F_k\scD_{X,\hspace{1pt}\fx}.\]
Then $Q-P\in\text{ann}_{\scD_{X,\hspace{1pt}\fx}}f_{\fx}^{-1}$, implying by \cite{D24}, Lemma 3.2, that
\[Q\in p_1(K_{l,0})+\text{ann}_{\scD_{X,\hspace{1pt}\fx}}f_{\fx}^s + \scD_{X,\hspace{1pt}\fx}(E+1).\]
However, $\left(\text{ann}_{\scD_{X,\hspace{1pt}\fx}}f_{\fx}^s\right)(E+1)^l \subseteq \text{ann}_{\scD_{X,\hspace{1pt}\fx}}f_{\fx}^s$, so $\text{ann}_{\scD_{X,\hspace{1pt}\fx}}f_{\fx}^s\subseteq p_1(K_{l,0})$. 

Therefore we see that 
\[Q\in p_1(K_{l,0})+\scD_{X,\hspace{1pt}\fx}(E+1),\]
so that
\[u \in ((p_1(K_{l,0})+\scD_{X,\hspace{1pt}\fx}(E+1))\cap F_k\scD_{X,\hspace{1pt}\fx})\cdot f_{\fx}^{-1}\]
as required.
    
\end{proof}

\newpage

\bibliographystyle{siam}
\bibliography{bibliography}

\end{document}